%% file: stableigusa.tex
\documentclass[11pt,reqno]{amsart}

\usepackage[all]{xy}
\usepackage{amscd}
\usepackage{mathrsfs}
\usepackage{amssymb}
\usepackage{amsthm}
\usepackage{amsmath}
\usepackage{enumerate}
\usepackage{indentfirst}
\usepackage{dsfont}
\usepackage{tikz-cd}

\usepackage[active]{srcltx}
\usepackage{xcolor}
\usepackage{comment}

\usepackage{hyperref}
\usepackage[style=super,nogroupskip,nopostdot]{glossaries} 

\newtheorem{thm}[subsubsection]{Theorem}
\newtheorem{conj}[subsubsection]{Conjecture}
\newtheorem{cor}[subsubsection]{Corollary}
\newtheorem{lem}[subsubsection]{Lemma}
\newtheorem{prop}[subsubsection]{Proposition}
\theoremstyle{definition}
\newtheorem{defn}[subsubsection]{Definition}
\newtheorem{ex}[subsubsection]{Example}
\theoremstyle{remark}
\newtheorem{rem}[subsubsection]{Remark}

\numberwithin{equation}{subsection}
\newtheorem{hypothesis}[subsubsection]{\textbf{Hypothesis}}

\newcommand{\A}{\mathbb{A}}

\newcommand{\cE}{{\mathcal{E}}}

\newcommand{\fkc}{{\mathfrak c}}
\newcommand{\fke}{{\mathfrak e}}

\newcommand{\fkp}{{\mathfrak p}}

\newcommand{\fkw}{{\mathfrak w}}

\newcommand{\fkD}{{\mathfrak D}}
\newcommand{\fkE}{{\mathfrak E}}
\newcommand{\fkX}{{\mathfrak X}}
\newcommand{\fkK}{{\mathfrak K}}
\newcommand{\KP}{{\mathcal{KP}}}

\newcommand{\mE}{{\mathscr E}}

\newcommand{\mS}{{\mathscr S}}

\newcommand{\cL}{{\mathcal L}}

\newcommand{\Z}{\mathbb{Z}}
\newcommand{\D}{\mathbb{D}}
\newcommand{\Q}{\mathbb{Q}}

\newcommand{\F}{\mathbb{F}}
\newcommand{\G}{\mathbb{G}}
\newcommand{\bb}{\mathbf{b}}

\newcommand{\R}{\mathbb{R}}

\newcommand{\C}{\mathbb{C}}

\newcommand{\scusp}{\mathrm{sc}}

\newcommand{\Ig}{\mathrm{Ig}}
\newcommand{\IG}{\mathfrak{Ig}}
\newcommand{\ST}{{\mathrm{ST}}}
\newcommand{\TN}{{\mathrm{TN}}}
\newcommand{\el}{{\mathrm{ell}}}
\newcommand{\reg}{\mathrm{reg}}
\newcommand{\im}{\mathrm{im}}

\newcommand{\semis}{\mathrm{ss}}
\newcommand{\sr}{\mathrm{sr}}
\newcommand{\Sh}{\mathrm{Sh}}

\newcommand{\inv}{{\rm inv}}
\newcommand{\tor}{{\rm Tor}}

\newcommand{\Fr}{{\rm Fr}}

\newcommand{\tr}{{\rm tr}\,}

\newcommand{\Hom}{{\rm Hom}}

\newcommand{\Aut}{{\rm Aut}}
\newcommand{\Out}{{\rm Out}}

\newcommand{\alg}{ {\rm alg}}
\newcommand{\bas}{ {\rm bas}}

\newcommand{\Lie}{{\rm Lie}\,}

\newcommand{\ur}{{\mathrm{ur}}}

\newcommand{\cG}{{\mathcal{G}}}
\newcommand{\cH}{{\mathcal{H}}}

\newcommand{\cO}{\mathcal{O}}

\newcommand{\sE}{\mathsf{E}}

\newcommand{\ad}{{\rm ad}}
\newcommand{\Int}{{\rm Int}}

\newcommand{\Groth}{{\rm Groth}}

\newcommand{\Fpbar}{\overline{\F}_p}
\newcommand{\Qlbar}{\overline{\Q}_\ell}
\newcommand{\Qpbar}{\overline{\Q}_p}
\newcommand{\Qbar}{\overline{\Q}}
\newcommand{\vol}{{\rm vol}}

\newcommand{\der}{\mathrm{der}}
\newcommand{\ab}{\mathrm{ab}}
\newcommand{\Res}{\mathrm{Res}}
\newcommand{\Iso}{{\rm isoc}}

\newcommand{\op}{{\rm op}}
\newcommand{\Rig}{\rm rig}
\newcommand{\Loc}{{\rm loc}}
\newcommand{\LL}{ {}^L }
\newcommand{\ov}{\overline}
\newcommand{\eff}{\mathrm{eff}}
\newcommand{\ef}{\mathrm{ef}}
\newcommand{\emb}{\mathrm{emb}}
\newcommand{\ES}{\mE\Sigma}
\newcommand{\EK}{\mE\fkK}

\def\tu{\textup}
\def\hat{\widehat}

\def\hra{\hookrightarrow}
\def\ra{\rightarrow}

\def\isom{\stackrel{\sim}{\ra}}
\def\ol{\overline}
\def\tilde{\widetilde}

\loadglsentries{entries}

\makeglossaries

\begin{document}

\title{The stable trace formula for Igusa varieties, II}

\author{Alexander Bertoloni Meli}\email{abertolo@bu.edu}
\address{Department of Mathematics, Boston University, Boston, MA 02215, USA}

\author{Sug Woo Shin}\email{sug.woo.shin@berkeley.edu}
\address{Department of Mathematics, UC Berkeley, Berkeley, CA 94720, USA / Korea Institute for Advanced Study, Seoul 02455, Republic of Korea}

\date{\today}


\begin{abstract}
Assuming the trace formula for Igusa varieties in characteristic $p$, which is known by Mack-Crane in the case of Hodge type with good reduction at $p$, we stabilize the formula via Kaletha's theory of rigid inner twists when the reductive group in the underlying Shimura datum is quasi-split at $p$. This generalizes our earlier work under more restrictive hypotheses. 

\end{abstract}

\maketitle

\tableofcontents

\section{Introduction}\label{s:intro}

Igusa \cite{Igusa} introduced Igusa curves to understand the mod $p$ geometry of modular curves when the level is divisible by powers of $p$. Over time, generalizations have been given in the setting of fairly general Shimura varieties (for example, \cite{HT01,HidaPAF,Man05,CS17,HamacherKim}). Igusa varieties help us understand the geometry of the Hodge--Tate period morphisms and Mantovan's product structure for Shimura varieties. Moreover Igusa varieties have found wide-ranging applications to the Langlands correspondence, the Kottwitz conjecture, vanishing of cohomology of Shimura varieties with torsion coefficients, $p$-adic automorphic forms, and the discrete Hecke orbit conjecture. (See the introduction of \cite{KretShin} for more details and references.) Most of these rely on a description of the $\ell$-adic cohomology of Igusa varieties. The computation is divided into three steps.
\begin{enumerate}
    \item[Step 1.] Adapt the Langlands--Kottwitz method to prove a trace formula computing the cohomology of Igusa varieties by means of counting fixed points. 
    \item[Step 2.] Stabilize the trace formula from Step 1. 
    \item[Step 3.] Compare the stabilized formula in Step 2 with the stabilized Arthur--Selberg trace formula to describe the cohomology in terms of automorphic representations. An endoscopic classification is required in general.
\end{enumerate}

For Igusa varieties associated with some simple Shimura varieties, all this has been carried out in \cite[Ch.V]{HT01}, where the simple situation made the second step unnecessary. 
More generally, Step~1 was done in the context of many PEL-type Shimura varieties in \cite{Shin09} and for Hodge-type Shimura varieties in \cite{MackCraneShin}, both in the case of good reduction at $p$ (i.e., with hyperspecial level at $p$). Step~3 was worked out in \cite{ShinGalois,ShinRZ,KretShin} in the minimal generality for intended applications and in \cite{BMaveraging} in many PEL-type cases assuming an endoscopic classification of automorphic representations and a technical assumption (see \cite[Assumption 4.8]{BMaveraging}). 

The goal of this paper is to complete the technically demanding Step 2 in considerable generality. 
The preceding works \cite{Shin10,BMaveraging} solved this problem for PEL-type Shimura varieties under technical hypotheses (see Assumption 5.1 in \cite{BMaveraging}). As a follow-up and significant improvement over \cite{Shin10,BMaveraging}, this paper deals with arbitrary Shimura varieties, only assuming that the reductive group in the Shimura datum is quasi-split at $p$ and that Step 1 is valid. (This assumption includes a reasonable formalism of Igusa varieties along with integral models for Shimura varieties. These are missing in general.) It is interesting to note that the analogue of Step 2 for Shimura varieties in \cite{KSZ} also works for arbitrary Shimura data but with hyperspecial level at $p$ (so the reductive group is assumed to be unramified at $p$). 

We hope that our paper will broaden the prospect for unconditional results. For instance, as a result of this paper (together with \cite{MackCraneShin}), Steps 1 and 2 are complete in the case of Hodge type with hyperspecial level at $p$, which led to an unconditional proof of the discrete Hecke orbit conjecture in \cite{KretShin}.

Our paper is written to be a reliable and relatively self-contained reference on the stable trace formula for Igusa varieties. For this reason, we have chosen to give a somewhat detailed exposition with basic definitions and facts spelled out, with an index of notation at the end.

\subsection{Igusa varieties}\label{ss:Igusa-intro}

To explain our results more precisely, let us recall the setup for Igusa varieties. When a field $k$ is in the subscript for a mathematical object, it will mean the base change of it to $k$.
Let $(G,X)$ be a Shimura datum. Denote by $E$ the reflex field of $(G,X)$, which is a finite extension of $\Q$. In particular $G$ is a connected reductive group over $\Q$ such that $G_{\R}$ contains an elliptic maximal torus, and $X$ determines a conjugacy class of cocharacters $\{\mu\}$ of $G_{\C}$. Fix a prime $p$ and a parahoric model $\cG$ of $G_{\Q_p}$ over $\Z_p$. Set $K_p:=\cG(\Z_p)$.
Write $\breve \Q_p$ for the completion of a maximal unramified extension of $\Q_p$. The arithmetic Frobenius automorphism on $\breve\Q_p$ is denoted $\sigma$. Let $\bb\in G(\breve\Q_p)$ be such that its image in the Kottwitz set $B(G_{\Q_p})$ lies in the finite subset $B(G_{\Q_p},\mu^{-1})$. Write $J_{\bb}(\Q_p)$ for the group of $\sigma$-centralizers of $\bb$ in $G(\breve\Q_p)$. Fix isomorphisms $\C\cong \Qlbar$ and $\C\cong \Qpbar$ and suppress them from the notation, but note that $E\hookrightarrow\C\cong \Qpbar$ induces a place $\fkp$ of $E$ over $p$. Throughout the paper, we assume that
\begin{itemize}
    \item $G_{\Q_p}$ is quasi-split.
\end{itemize}
This condition is used mainly for group-theoretic reasons in this paper. (It should also be helpful to assume for Step 1.) Removing the condition might be possible but would cause nontrivial changes that are elusive to us at this time.

Given $(G,X,p,\cG)$ and $\bb$ as above and an irreducible algebraic representation $\xi$ of $G_{\C}$, our basic hypotheses are that there exist
\begin{itemize}
    \item a natural integral model $\mS_{K_p}$ over $\cO_{E_{\fkp}}$ for the Shimura variety with level $K_p$ at $p$ (Hypothesis \ref{hypo:integral-model-Shimura}; also see Remark \ref{rem:natural}), and
    \item an Igusa variety $\IG_{\bb}\ra \mS_{K_p,\Fpbar}$ equipped with an $\ell$-adic local system $\cL_{\xi}$, which satisfies certain desiderata (Hypothesis \ref{hypo:Igusa-existence}). In particular $G(\A^{\infty,p})\times J_{\bb}(\Q_p)$ acts on $\IG_{\bb}$, and the map $\IG_{\bb}\ra \mS_{K_p,\Fpbar}$ is  equivariant for the $G(\A^{\infty,p})$-action (prime-to-$p$ Hecke action).
\end{itemize}
These hypotheses are known for Shimura data of Hodge type and a reductive model $\cG$, and also in many cases for Shimura data of abelian type and parahoric models $\cG$. See \S\ref{sss:known} below.

We define $[H_c]_{\bb,\xi}$ to be the virtual $G(\A^{\infty,p})\times J_{\bb}(\Q_p)$-module equal to the alternating sum of compactly supported $\ell$-adic cohomology of $\IG_{\bb}$ with coefficient sheaf $\cL_{\xi}$, cf.~\S\ref{sss:H(IG)} below.
A representation-theoretic description of $[H_c]_{\bb,\xi}$ essentially amounts to computing the trace
$$\tr (\phi^{\infty,p}\phi_p |[H_c]_{\bb,\xi})\in \C, \qquad \phi^{\infty,p}\phi_p\in C^\infty_c(G(\A^{\infty,p})\times J_{\bb}(\Q_p)).$$

\subsection{The main theorem}\label{ss:main-thm}

Our main theorem corresponds to Step 2 in the introductory paragraph above. As a prerequisite, we formulate the expected outcome of Step 1.
For simplicity of exposition, we ignore central character data, $z$-extensions, and choices of Haar measures even though they are necessary for stating and proving precise results, which are found in the main text. (It is one of our improvements over \cite{Shin10,BMaveraging} to deal with such technicalities in full.) 

Define $\Sigma_{\R\tu{-ell}}(G)$ to be a set of representatives for stable conjugacy classes in $G(\Q)$ which are elliptic in $G(\R)$. Write $\KP_{\bb}$ for the set of $\bb$-admissible Kottwitz parameters for $(G,X,p,\cG)$, (Definition \ref{def:Kottwitz-parameter}). To get a feel for $\KP_{\bb}$ in the case of Hodge type, a Kottwitz parameter is a group-theoretic datum to parametrize an adelic isogeny class of abelian varieties with extra structure in the mod $p$ Shimura variety $\mS_{K_p}(\Fpbar)$, and the $\bb$-admissibility prescribes the isogeny class of the $p$-divisible group in terms of $\bb$. 

Each $\gamma_0\in \Sigma_{\R\tu{-ell}}(G)$ 
determines a subset $\KP_{\bb}(\gamma_0)\subset \KP_{\bb}$, which can be thought of as the fiber over $\gamma_0$. Each $\fkc\in \KP_{\bb}(\gamma_0)$ determines (see~\S\ref{sss:summary-Kottwitz-parameters} below)
\begin{itemize}
    \item $\gamma_{\fkc}\in G(\A^{\infty,p})$ and $\delta_{\fkc}\in J_{\bb}(\Q_p)$ up to conjugacy, which are stably conjugate to $\gamma_0$, and
    \item a cohomological invariant $\alpha(\fkc)$ in a finite abelian group.
\end{itemize}
In particular, orbital integrals $O^G_{\gamma_\fkc}$ and $O^{J_{\bb}}_{\delta_{\fkc}}$ are well defined 
on the Hecke algebras $\cH(G(\A^{\infty,p}))$ and $\cH(J_{\bb}(\Q_p))$, respectively. 
Write $\KP^{\Fr}_{\bb}(\gamma_0)$ for the subset of $\fkc\in \KP_{\bb}(\gamma_0)$ with trivial $\alpha(\fkc)$.

The element $\bb$ determines a fractional Newton cocharacter $\nu_{\bb}\in X_*(J_{\bb})\otimes_{\Z}\Q$, which is central in $J_{\bb}$. Let $r\in \Z_{>0}$ be such that $r\nu_{\bb}\in X_*(J_{\bb})$.
For $\phi_p\in \cH(J_{\bb}(\Q_p))$ and $j\in r\Z$, let $\phi_p^{(j)}$ denote the translate of $\phi_p$ by $(j\nu_{\bb})(p)$. We are finally ready to state the anticipated output of Step 1.

\begin{conj}\label{conj:Igusa-simplified}
Given $\phi^{\infty,p}\in \cH(G(\A^{\infty,p}))$ and $\phi_p\in \cH(J_{\bb}(\Q_p))$, there exists $j_0\in \Z_{>0}$ such that the following holds:
  for every integer $j\ge j_0$ divisible by $r$, we have
  \begin{equation}\label{eq:MackCrane-simplified}
  \tr\big(\phi^{\infty,p}\phi^{(j)}_p \left|[H_c]_{\bb,\xi} \right.\big)
 = \!
   \sum_{\gamma_0\in \Sigma_{\R\tu{-ell}}(G)}
   \!\!  \tr \xi (\gamma_0)  \! \sum_{\fkc\in \KP^{\Fr}_{\bb}(\gamma_0)} c (\mathfrak c) O^G_{\gamma_{\mathfrak c}}(\phi^{\infty,p})O^{J_{\bb}}_{\delta_{\mathfrak c}}(\phi^{(j)}_p),
  \end{equation}
  where $c (\mathfrak c)\in \Q$ is an explicit constant. 
\end{conj}

The latest work on this conjecture is \cite{MackCraneShin} based on Mack-Crane's thesis, where it is proved for Hodge-type data $(G,X)$ and reductive models $\cG$ (with no constraints on $p$). This was done by formulating and proving a suitable analogue of the Langlands--Rapoport conjecture for Igusa varieties, via \cite{KisinPoints,KSZ} on the Langlands--Rapoport conjecture for Shimura varieties.

The stabilization problem (Step 2) is to rewrite the right hand side of \eqref{eq:MackCrane-simplified} in terms of stable (elliptic) orbital integrals on endoscopic groups of $G$. To this end, let $\sE^\heartsuit_{\el}(G)$ denote a set of representatives for isomorphism classes of elliptic endoscopic data for $G$. Each $\fke\in\sE^\heartsuit_{\el}(G)$ consists of a  quasi-split group over $\Q$, to be denoted $H$, and other information. Write $\ST^{H}_{\el}$ for the elliptic part of the stable trace formula for $H$, which is essentially a sum of stable elliptic orbital integrals on $H$; see \S\ref{eq:def-of-ST} below. Our main theorem is as follows.

\begin{thm}\label{thm:intro}
Keep the hypotheses in the bullet list of \S\ref{ss:Igusa-intro} and
assume Conjecture \ref{conj:Igusa-simplified}.
For $\phi^{\infty,p}$, $\phi_p$, $\xi$, and sufficiently large $j$ as in the conjecture, there exists $h\in \cH(H(\A))$ (depending on~$j$) such that the following stabilized formula holds true:
  $$\tr\big(\phi^{\infty,p}\phi^{(j)}_p \big|[H_c]_{\bb,\xi}\big) = 
  \sum_{\fke\in \sE^\heartsuit_{\el}(G)}\iota(\fke) \ST^{H}_{\el} (h),$$
  where $\iota(\fke)\in \Q_{>0}$ is an explicit constant. 
\end{thm}

For the theorem to be useful for a spectral interpretation of $[H_c]_{\bb,\xi}$ as in Step 3, a good control of $h$ is needed at each place of $\Q$. This is possible and built into the construction of $h$ place by place, to be discussed further in the next subsection.

The reader may have noticed that Theorem \ref{thm:intro} is very similar to the stabilized trace formula for Shimura varieties in the usual Langlands--Kottwitz method, cf.~\cite[Thm.~1]{KSZ}.
This is not a coincidence as the starting point \eqref{eq:MackCrane-simplified} resembles the counterpart for Shimura varieties, cf.~\cite[Thm.~2]{KSZ}. 
In fact, the construction of $h$ away from $p$ proceeds in the same way as in that setting, but the situation at $p$ is entirely different. The $p$-part of the picture is also the most complicated and interesting, not only in Step 2 but in all three steps. Besides various other improvements, our main novelty lies in the precise stabilization at $p$ through Kaletha's theory of rigid inner twists and their relationship with Kaletha--Kottwitz's extended pure inner twists. 
\subsection{Sketch of proof}\label{ss:sketch-of-proof}
Away from $p$ and $\infty$, we obtain $h$ from the Langlands--Shelstad transfer of $h$. At $\infty$, the function $h$ is built from the pseudo-coefficients of discrete series representations associated with $\xi$. The construction of $h$ at $p$ from $\phi_p^{(j)}$ is the most complicated and interesting, since this ``transfer'' does not fit in the usual setup for endoscopy. 

We briefly explain how $h$ is constructed at $p$. The group $J_{\bb}$ is an inner twist of a Levi subgroup $M_{\bb}$ of $G_{\Q_p}$ so there exists a map from endoscopic data of $J_{\bb}$ to those of $G_{\Q_p}$. Let $\fke_p$ be an endoscopic datum for $G_{\Q_p}$. For each endoscopic datum $\fke_{\bb}$ of $J_{\bb}$ in the fiber over $\fke_p$, if $H_{\bb}$ (resp. $H$) denote the endoscopic groups in the data $\fke_{\bb}$ (resp. $\fke_p$), then we obtain a function $\phi^{H_{\bb}}_p \in \cH(H_{\bb}(\Q_p))$ from $\phi_p^{(j)}$ by the Langlands--Shelstad transfer. Exploiting the fact that $H_{\bb}$ is also a Levi subgroup of $H$ that is the centralizer group of a Newton cocharacter given by $\bb$, the function $h_p \in \cH(H(\Q_p))$ is then constructed as a sum of \emph{ascents} (see Definition \ref{def:ascent}) of the $\phi^{H_{\bb}}_p$ to $H$. The existence of ascents satisfying the desired orbital integral and trace identities relies on a support condition on $\phi^{H_{\bb}}_p$, which is fulfilled if $j$ is sufficiently large.
(In this paper we want $j$ to be large only for this purely group-theoretic reason. Another reason, not needed here but relevant in the proof of Conjecture \ref{conj:Igusa-simplified}, is that a large $j$ corresponds to twisting by a high power of Frobenius when applying the Fujiwara--Varshavsky trace formula to Igusa varieties.)

To complete the proof, one verifies that $h$ as constructed satisfies Theorem \ref{thm:intro} via an extensive computation using the fundamental equations for endoscopic transfers and ascents (Definition \ref{def:Delta-transfer}, Definition \ref{def:ascent}). The argument proceeds as in the case for Shimura varieties (see \cite[Part 3]{KSZ}), taking the new phenomena at $p$ into account.

At a technical level, the main novelties in our project arise from the interplay between Kaletha's theory of rigid inner twists and the trace formula for Igusa varieties. We must relate on the one hand the invariant $\langle \beta_p[b], s_p \rangle$ arising in the $p$-part of the pre-stabilized trace formula of Igusa varieties and on the other hand, the invariants $\langle \inv[z^{\Rig}](\gamma, \gamma_0), \dot{s}_p \rangle$ arising in Kaletha's normalization of transfer factors for rigid inner twists. This comparison is accomplished in \S \ref{sss:endoscopy-at-p} by relating both invariants to those arising in the theory of $B(G)$ and extended pure inner twists.  The previous comparison in \cite{Shin10} was done with respect to unspecified normalizations of transfer factors (resulting in an undetermined constant for each $e_{\bb}$); the comparison was later made precise in \cite{BMaveraging} but under extra assumptions.

We are able to drop all technical assumptions on $G$ (other than quasi-split at $p$) by working with $z$-extensions, central character data, and Kottwitz parameters (as in \cite{KSZ}) as opposed to the usual Kottwitz triples. Our precise normalization of local transfer factors using rigid/extended pure inner twists implies that Theorem \ref{thm:intro} will be suitable for local representation-theoretic applications such as those of \cite{BMaveraging} in cases where Step 3 can be completed.

\subsection{A guide for the reader}

The structure of the paper should be clear from the table of contents but we make a few remarks to facilitate easier navigation. On a first reading, all details relating to $z$-extensions and central character data can be safely ignored. Section \ref{s:prelim} is preliminary and deals primarily with various notions of endoscopy, inner twists, and their interactions. The reader might therefore prefer to start at \S \ref{s:Igusa} referring back as necessary. The goal of \S \ref{s:Igusa} is to state the pre-stabilized trace formula for the cohomology of Igusa varieties (Conjecture \ref{conj:Igusa-general}) and formulate the required geometric and group-theoretic setup. 

The stabilization of this formula is carried out in \S \ref{s:stabilization} using the preliminaries in \S \ref{s:prelim}. The most technical part of the paper is \S \ref{ss:stab-at-p}, where the contribution to the trace formula at $p$ is stabilized (\S \ref{sss:stab-at-p}). The  subsection \S \ref{sss:comparison-trans-factors} is devoted to comparing transfer factors of rigid inner twists with certain invariants arising in the trace formula of Igusa varieties. The reader may wish to black box the main result of this subsection (Corollary \ref{transfactcomparison}) on a first read. The full stabilization is then completed in \S \ref{ss:final-steps} by combining the results at each place to deduce our main formula which is Theorem \ref{thm:stable-Igusa}.

\subsection{Notation and conventions}\label{ss:notation}

Let $F$ be a local or global field of characteristic 0, with an algebraic closure $\ol{F}$.
Given $E/F$ a Galois extension, we denote the Galois group by \gls{GammaEF}. We denote $\Gamma_{\ol{F}/F}$ by \gls{GammaF}. We let $W_F$ denote the absolute Weil group of $F$. We use the Deligne normalization of the local Artin map so that uniformizers correspond to geometric Frobenius elements.

When $G$ is a linear algebraic group over $F$, the connected component of the identity is denoted by $G^\circ$. For a field extension $E/F$, write $G_E$ for the base change of $G$ to $E$. When $F/F_0$ is a finite extension, write $\Res_{F/F_0} G$ for the Weil restriction of scalars.
When $F$ is a number field (e.g., $F=\Q$) and $v$ is a place of $F$, we abbreviate $G_{F_v}$ as $G_v$.

For $G$ a connected reductive group over $F$, let $Z_G$ and $Z(G)$ denote the center of $G$, and let $G_{\der}$ and $G_{\ad}$ denote the derived and adjoint groups of $G$. We let $G_{\scusp}$ denote the simply connected cover of $G_{\der}$. Write $G(F)_{\semis}$ (resp.~$G(F)_{\sr}$) for the subset of semi-simple (resp.~strongly regular semi-simple) elements of $G(F)$. The meaning of strong regularity is that the centralizer is a torus. For $\gamma\in G(F)_{\semis}$, we write $G_\gamma$ for the centralizer of $\gamma$ in $G$, and $I_\gamma=I^G_\gamma$ for $(G_\gamma)^\circ$. We use the notation \gls{Gamma(G(F))} (resp. \gls{SigmaG(F)})  for the set of $F$-conjugacy (resp. stable conjugacy) classes in $G(F)$. Recall that the stable conjugacy class of $\gamma \in G(F)$ is the set of $\gamma' \in G(F)$ conjugate to $\gamma$ by some $g \in G(\ov{F})$ and such that the cohomology class in $H^1(F, G_{\gamma})$ of $\sigma \mapsto g^{-1} \sigma(g)$ lies in the image of $H^1(F, I_{\gamma}) \to H^1(F, G_{\gamma})$. An inner twist of $G$ consists of a triple $(G', \psi, z)$ such that $\psi: G_{\ov{F}} \to G'_{\ov{F}}$ is an isomorphism and $z \in Z^1(\Gamma_F, G_{\ad})$ equals $\sigma \mapsto \psi^{-1} \circ \sigma(\psi)$. 

We denote the Langlands dual group of $G$ by $\hat{G}$. As in \cite[\S1.2]{KottwitzShelstad} the dual group $\hat{G}$ always comes equipped with a $\Gamma_F$-pinning $(\hat{T}, \hat{B}, \{X_{\alpha}\})$, an action $W_F \to \Aut(\hat{G})$ that factors through a finite quotient of $W_F$, and a $\Gamma_F$-equivariant bijection between the resulting based root datum of $\hat G$ and the dual of a based root datum for $G$ (coming from an inner twist to a quasi-split form equipped with an $F$-pinning, which induces a pinning for $G$). This information for $\hat G$ will be implicit whenever dual groups appear in this paper. The $W_F$-action is used to form the $L$-group $^L G:=\hat G \rtimes W_F$.
We let $\Out(\hat{G})$ denote the quotient of $\Aut(\hat{G})$ modulo inner automorphisms. The action of $W_F$ induces a map $\gls{rhoG}: W_F \to \Out(\hat{G})$. For a maximal torus $T$ in $G$, write $R(G,T)$ for the set of roots of $T_{\ol F}$ in $G_{\ol F}$.

Assume further that $F$ is a local field. 
Write $|\cdot|:F^{\times} \ra \C^{\times}$ for the absolute value such that every uniformizer maps to the inverse cardinality of the residue field of $F$. When $M$ is an $F$-rational Levi subgroup of $G$, and $\gamma\in M(F)_{\semis}$, define the Weyl discriminant
$D^G_M(\gamma):=\mathrm{det}\big(1- \mathrm{ad}(\gamma)|_{\mathrm{Lie}(G) /\mathrm{Lie}(M)}\big)$.

Let $S$ be a finite set of places of $\Q$. We write $\A^S$ for the ring of ad\`eles away from $S$, i.e., the restricted product $\prod'_{v\notin S}\Q_v$.
We write $\ol \A^S:=\A^S\otimes_{\Q} \Qbar$.

When $A$ is a set or a group, 
$\mathbf{1}_A$ means the characteristic function on $A$ or the trivial character on $A$ (if $A$ is a group); the context will make it clear.

In this paper we commonly abuse notation by confusing an isomorphism class of objects with a representative of the isomorphism class, and likewise for conjugacy classes etc. Implicit in this abuse is that the assertions or computations we make about them do not depend on the choice of the representative. For instance, when computing the orbital integral at $\gamma\in G(F)_{\semis}$ with $F$ a local field, the value is well defined whether $\gamma$ is a conjugacy class or an arbitrary representative thereof. However, when it comes to endoscopic data, we do choose a set of representatives for isomorphism classes in order to avoid confusion stemming from different notions of endoscopic data and isomorphisms between them.

\subsection*{Acknowledgments}

We are grateful to Alex Youcis for numerous comments and suggestions to improve the paper.
We thank Paul Hamacher and Wansu Kim for comments on \S3.1 of this paper.
We also thank the anonymous referee for sending us invaluable suggestions and corrections.
S.W.S.~was partially supported by NSF grant DMS-1802039/2101688, NSF RTG grant DMS-1646385, and a Miller Professorship. A.B.M.~was partially supported by NSF grant DMS-1840234.

\section{Preliminaries}\label{s:prelim}

Throughout this section, $F$ is a local or global field, and $G$ is a connected reductive group over~$F$. The characteristic of $F$ is always 0, and we will not repeat this condition every time.

\subsection{Rigid inner twists}

In this subsection, we recall some facts from the theory of local and global rigid inner twists as developed in 
\cite{KalethaLocalRigid,KalethaGlobalRigid, KalethaTaibi}.

\subsubsection{The local gerbe}
We first describe the local case. Let $F$ be a local field of characteristic~$0$, let $E/F$ be a finite Galois extension, and let $N$ be a positive integer. We then define  a finite $\Gamma_{E/F}$-module
\begin{equation*}
    \gls{MrigEN} := (\Z/ N \Z)[ \Gamma_{E/F} ]_0,
\end{equation*}
where the $0$-subscript denotes the elements whose coefficients sum to $0$. Let \gls{uEFN} denote the finite multiplicative group with character group $M^{\Rig}_{E,N}$ and let \gls{u} denote the pro-finite multiplicative group equal to the inverse limit over $E$ and $N$ of the $u_{E/F, N}$ where the transition maps are induced by the norm.
We have canonical isomorphisms
\begin{align*}
    H^1(F, u) & = 0,\\
    H^2(F, u) & = \begin{cases}
    \Z/2\Z, & F = \R, \\
    \widehat{\Z}, & F = \text{ non-archimedean. }
    \end{cases}
\end{align*}
Choose a cocycle representing the class of $-1$ in $H^2(F, u)$. This cocycle gives rise to a pro-Galois gerbe $\cE^{\Rig}_F$:
\begin{equation*}
    1 \to u(\ol{F}) \to \gls{cErig} \to \Gamma_F \to 1.
\end{equation*}

Now let $G$ be a connected reductive group over $F$ and let $Z \to G$ denote the data of a multiplicative group $Z$ and an embedding $Z \hookrightarrow Z(G)$. Let $Z^1(u \to \cE^{\Rig}_F, Z \to G)$ denote the set of continuous cocycles of $\cE^{\Rig}_F$ valued in $G(\ol{F})$ whose restriction to $u(\ol{F})$ factors through $Z(\ol{F})$ and corresponds to an algebraic homomorphism $u \to Z$. Define $H^1(u \to \cE^{\Rig}_F, Z \to G)$ to be the quotient of $Z^1(u \to \cE^{\Rig}_F, Z \to G)$ by the action of $G(\ol{F})$ as in the normal definition of non-abelian group cohomology. Since $H^1(F, u) = 0$, the only automorphisms of $\cE^{\Rig}_F$ are inner and hence the set $H^1(u \to \cE^{\Rig}_F, Z \to G)$ is independent, up to canonical isomorphism, of the choice of cocycle representing the class of $-1$ in $H^2(F, u)$.

In this paper, we will work primarily in the case where $Z=Z(G)$. This cohomology set admits the following Tate--Nakayama morphism. We define $Z_n$ to be the pre-image in $Z(G)$ of $(Z(G)/Z(G_{\der}))[n]$ and set $\gls{Gn} := G/Z_n$.  Note that $Z(G_1) = Z(G)/Z(G_{\der})$ is a torus and $Z(G_n) = Z(G_1)/Z(G_1)[n]$. Moreover, we have $G_n = G_{\ad} \times Z(G_n)$. 

We then set $\gls{hatbarG} := \varprojlim \widehat{G_n}$ and define \gls{ZhatG+} to be the pre-image of $Z(\hat{G})^{\Gamma_F}$ under the natural map $\hat{\bar{G}} \to \hat{G}$. We have the formula
\begin{equation*}
    \hat{\bar{G}} = \hat{G}_{\scusp} \times \hat{C_{\infty}},
\end{equation*}
where $\hat{C_{\infty}} = \varprojlim\limits_n \hat{Z(G_n)}$. We can identify $\hat{Z(G_n)}$ with $Z(\hat{G})^{\circ}$. The map $\hat{Z(G_m)} \to \hat{Z(G_n)}$ for $n \mid m$ becomes the $\frac{m}{n}$th power map under this identification.

\begin{prop}[{\cite[Equation (3.12)]{KalethaRigidvsIsoc}}]{\label{TateNakayama}}
We have a canonical morphism 
\begin{equation*}
    \gls{TN}: H^1(u \to \cE^{\Rig}_F, Z(\hat{G}) \to G) \to \pi_0(Z(\hat{\bar{G}})^+)^D.
\end{equation*}
\end{prop}

\subsubsection{The global gerbe}
We now describe the case where $F$ is a global field of characteristic $0$. Let $V$ be the set of places of $F$ and pick a set of lifts $\dot{V}$ of $V$ to $\ol{F}$ such that $\bigcup_{v \in V} \Gamma_{F_v}$ is dense in $\Gamma_F$, where the chosen lift of $v$ determines the embedding $\Gamma_{F_v}\hookrightarrow \Gamma_F$. As in \cite[Lem.~3.3.2]{KalethaTaibi}, we can choose an increasing sequence $(S_i)_{i \geq 0}$ of finite subsets of $V$ and an exhaustive tower of finite Galois extensions $(E_i)_{i \geq 0}$ of $F$ such that the sequence $(E_i, S_i, \dot{S_i})_{i \geq 0}$ satisfies \cite[Conditions 3.3.1]{KalethaGlobalRigid}, where $\dot{S_i}$ is the intersection of the set $S_{E_i}$ of places of $E_i$ lying over $S_i$ and the image of $\dot{V}$ in $V_{E_i}$. We then define the finite $\Gamma_{E_i/F}$-module \gls{MrigES} consisting of the subgroup of  $(\Z/ [E_i:F]\Z)[\Gamma_{E_i/F} \times S_{E_i}]$ containing the elements $\sum_{w,v} a_{w,v}[(w,v)]$ such that $\sum_{w'} a_{w',v} = \sum_{v'} a_{w,v'}=0$ for all $(w,v) \in \Gamma_{E_i/F} \times S_{E_i}$ and such that $a_{w,v}=0$ unless $w^{-1}(v) \in \dot{S_i}$.

Let \gls{PrigES} be the finite multiplicative group over $F$ with character group equal to $M^{\Rig}_{E_i, \dot{S}_i}$ and let $\gls{PrigV} := \varprojlim P^{\Rig}_{E_i, \dot{S}_i}$. We have  $H^1(F, P^{\Rig}_{\dot{V}})=0$ from \cite[\S 3.4]{KalethaGlobalRigid}. Further, Kaletha (\cite[Def.~3.5.4]{KalethaGlobalRigid}) constructs a canonical class $\xi_{\dot{V}} \in H^2(F, P^{\Rig}_{\dot{V}})$ that, by choosing a cocycle lift of this class, gives a Galois gerbe
\begin{equation*}
    1 \to P^{\Rig}_{\dot{V}}(\ol{F}) \to \gls{cErigdotV} \to \Gamma_F \to 1.
\end{equation*}
The class $\xi_{\dot{V}}$ does not depend on the sequence $(E_i, S_i, \dot{S_i})$ as long as two sequences yield the same $\dot{V}$. However, it does depend on $\dot{V}$ as described in \cite[Rem.~3.8.5]{KalethaTaibi}. By virtue of the vanishing of $H^1(F, P^{\Rig}_{\dot{V}})$, the cohomology set $H^1(P^{\Rig}_{\dot{V}} \to \cE^{\Rig}_{\dot{V}}, Z \to G)$ is independent, up to canonical isomorphism, of the choice of cocycle lift of $\xi_{\dot{V}}$.

As in the local case, there is a Tate--Nakayama morphism
\begin{equation*}
   \TN:  H^1(P^{\Rig}_{\dot{V}} \to \cE^{\Rig}_{\dot{V}}, Z(G) \to G) \to \ol{Y}^{\Rig}_{\dot{V}},
\end{equation*}
where $\ol{Y}^{\Rig}_{\dot{V}}$ is a limit of the linear algebraic objects $\ol{Y}[V_{\ol{F}}, \dot{V}]_{0,+, \tor}(Z_n \to G)$ described in \cite[\S\S3.7--3.8]{KalethaGlobalRigid}.

\subsubsection{Localization}{\label{sss:localization}}

Let $F$ be a number field. For each place $v \in \dot{V}$, Kaletha \cite[\S3.6]{KalethaGlobalRigid} constructs a localization map:
\begin{equation*}
    \Loc_v : H^1(P^{\Rig}_{\dot{V}} \to \cE^{\Rig}_{\dot{V}}, Z(G) \to G) \to H^1(u_v \to \cE^{\Rig}_v, Z(G) \to G),
\end{equation*}
which is compatible with the Tate--Nakayama morphisms in that the following diagram commutes (\cite[Thm.~3.8.1]{KalethaGlobalRigid}):
\begin{equation*}
\begin{tikzcd}
H^1(P^{\Rig}_{\dot{V}} \to \cE^{\Rig}_{\dot{V}}, Z(G) \to G) \arrow[r, "\oplus_v \Loc_v"] \arrow[d, "\TN"] & \bigoplus\limits_{v \in \dot{V}} H^1(u_v \to \cE^{\Rig}_v, Z(G) \to G) \arrow[d, "\oplus_v \TN_v"]  \\
 \ol{Y}^{\Rig}_{\dot{V}} \arrow[r]  & \bigoplus\limits_{v \in \dot{V}} \pi_0(Z(\hat{\bar{G}})^+),
\end{tikzcd}    
\end{equation*}
where the direct sum in the upper right corner denotes the subset of the direct product of pointed sets consisting of families of elements that are trivial at almost every $v$.

\subsubsection{Local rigid inner twists}
Let $F$ be a local field of characteristic $0$. Given a cocycle $z \in H^1(u \to \cE^{\Rig}_F, Z \to G)$, we can post-compose with the map $G \to G_{\ad}$ to get a cocycle $z_{\ad}: \cE^{\Rig}_F \to G_{\ad}$. This cocycle is trivial on $u$ and hence factors to give a class of $Z^1(F, G_{\ad})$. We have an entirely analogous map in the case that $F$ is a number field.

\begin{lem}{\label{H1surjectionlem}}
\begin{itemize}
\item If $F$ is local, the above map induces a surjection
\begin{equation*}
    H^1(u \to \cE^{\Rig}_F, Z(G) \to G) \twoheadrightarrow H^1(F, G_{\ad}).
\end{equation*}
\item If $F$ is global, the above map induces a surjection
\begin{equation*}
    H^1(P^{\Rig}_{\dot{V}} \to \cE^{\Rig}_{\dot{V}}, Z(G) \to G) \twoheadrightarrow H^1(F, G_{\ad}). 
\end{equation*}
\end{itemize}
\end{lem}
\begin{proof}
\cite[Cor.~3.8]{KalethaLocalRigid} and \cite[Lem.~3.6.1]{KalethaGlobalRigid}.
\end{proof}
\begin{defn}\label{def:local-rigid-inner-twist}
For $F$ a local field of characteristic $0$ and $G$ a fixed connected reductive group, a local \emph{rigid inner twist} of $G$ is a triple $(G', \psi, z)$ consisting of 
\begin{itemize}
\item a connected reductive group $G'$ defined over $F$,
\item an isomorphism $\psi: G_{\ol{F}} \to G'_{\ol{F}}$, and
\item a cocycle $z \in Z^1(u \to \cE^{\Rig}_F, Z(G) \to G)$ such that for each $\sigma \in \Gamma_F$, we have
\begin{equation*}
    \psi^{-1} \circ \sigma(\psi) = \Int(z_{\ad}(\sigma)),
\end{equation*}
where $z_{\ad}$ is the projection of $z$ to $G_{\ad}$.
\end{itemize}

Suppose we have  rigid inner twists $(G_1, \psi_1, z_1)$ and $(G_2, \psi_2, z_2)$. Then an isomorphism from $(G_1, \psi_1, z_1)$ to $(G_2, \psi_2, z_2)$ is a pair $(f,g)$ consisting of
\begin{itemize}
    \item an element $g \in G(\ol{F})$ satisfying $z_2(w) = g z_1(w)w(g^{-1})$ for each $w \in \cE^{\Rig}_F$, and
    \item an isomorphism $f: G_1 \to G_2$ defined over $F$ such that the following diagram commutes
    \begin{equation*}
        \begin{tikzcd}
            G_{\ol F} \arrow[r, "\psi_1"] \arrow[d, "\Int(g)", swap] & G_{1, \ol F} \arrow[d, "f"] \\
            G_{\ol F} \arrow[r, "\psi_2"] & G_{2, \ol F}.
        \end{tikzcd}
    \end{equation*}
\end{itemize}
\end{defn}
\begin{rem}
A naive analogue of Definition \ref{def:local-rigid-inner-twist} gives a notion of a global rigid inner twist for a connected reductive group over a number field, which yields a local rigid inner twist at each place. However, this is not the correct notion of global rigid inner twist for our purposes. In particular, the family of local rigid inner twists one gets from the above process need not have the necessary coherence properties to give an adequate normalization of local transfer factors. Instead, we first pass to $G^*_{\scusp}$, see ~\S\ref{ss:global transfer factors construction} below.
\end{rem}

We need the following compatibilities between inner twists. 

\begin{lem}{\label{cocyclelem}}
Let $G_1, G_2, G_3$ be reductive groups defined over a field $F$. Suppose that we have a commutative diagram of isomorphisms $\psi_1, \psi_2, \psi_3$ defined over $\ov{F}$:
\begin{equation*}
    \begin{tikzcd}
    G_{1, \ol F} \arrow[r, "\psi_1"] \arrow[rd, "\psi_3"'] & \arrow[d, "\psi_2"] G_{2, \ol F} \\
        & G_{3, \ol F},
    \end{tikzcd}
\end{equation*}
such that $w \mapsto \psi^{-1}_i \circ w(\psi_i)$ induces an element of $Z^1(F, G_{1, \ad})$ for $i=1,3$ and an element of $Z^1(F, G_{2, \ad})$ for $i=2$.
Let $z_1, z_3 \in Z^1(F, G_{1, \ad})$ and $z_2 \in Z^1(F, G_{2, \ad})$ be the $1$-cocycles defined by $\Int(z_{i}(w)) = \psi^{-1}_i \circ w(\psi_i)$. Then we have 
\begin{equation*}
    \psi^{-1}_1(z_2)z_1 = z_3.
\end{equation*}
\end{lem}
\begin{proof}
Indeed, for $w \in \Gamma_{F}$, we have
\begin{equation}{\label{eq: triplecocycles}}
    \psi^{-1}_3 \circ w(\psi_3) = \psi^{-1}_1 \circ \psi^{-1}_2 \circ w( \psi_2 \circ \psi_1 ) = \psi^{-1}_1 \circ \Int(z_2(w)) \circ w(\psi_1) = \Int(\psi^{-1}_1(z_2(w))z_1(w)).
\end{equation}
\end{proof}
\begin{lem}{\label{rigcocycletwistlem}}
Let $G_1$ and $G_2$ be reductive groups defined over a non-archimedean field $F$. Suppose that we have a $\ov{F}$-isomorphism
\begin{equation*}
    G_{1, \ol F} \xrightarrow{\psi} G_{2, \ol F},
\end{equation*}
and $1$-cocycles $z_i \in Z^1(u \to \cE^{\Rig}_{F}, Z(G_i) \to G_i)$ for $i=1, 2$ respectively such that $z_1|_u = z_2|_u$ and $w \mapsto \psi^{-1} \circ w(\psi) = \Int(z_1)$. Then
\begin{equation*}
    \TN(\psi^{-1}(z_2)z_1) =  \psi^{-1}(\TN(z_2))+ \TN(z_1).
\end{equation*}
\end{lem}
\begin{proof}
Let $S_1 \subset G_1$ be a fundamental torus. Since we are working over a non-archimedean local field, this means that $S_1$ is an elliptic maximal torus. By \cite[Cor.~3.7]{KalethaLocalRigid}, we have a surjection 
$$H^1(u \to \cE^{\Rig}_{F}, Z(G_1) \to S_1) \twoheadrightarrow H^1(u \to \cE^{\Rig}_{F}, Z(G_1) \to G_1).$$
Hence, we may modify $z_1$ by a coboundary to get $z'_1$ such that $z'_1(w) \in S_1$ for each $w \in \cE^{\Rig}_{F}$. We may similarly modify $\psi$ to get $G_{1, \ol F} \xrightarrow{\psi'} G_{2, \ol F}$ such that for each $w \in \cE^{\Rig}_{F}$, we have ${\psi'}^{-1} \circ w(\psi') = \Int(z'_1(w))$.

Now we define $S_2 := \psi'(S_1)$ and note that $S_2$ is defined over $F$. Indeed, for each $\sigma \in \Gamma_{F}$, we can lift $\sigma$ to some $w \in \cE^{\Rig}_{F}$ and get for $s \in S_1(\ov{F})$,
\begin{equation*}
    \sigma(\psi'(s)) = \sigma(\psi')(\sigma(s)) = (\psi' \circ \Int(z'_1(w)) )(\sigma(s)) = \psi'(\sigma(s)).
\end{equation*}
Hence, $S_2$ is a fundamental torus of $G_2$ and we can choose $z'_2$ that is cohomologous to $z_2$ and factors through $Z^1(u \to \cE^{\Rig}_{F}, Z(G_2) \to S_2) \to Z^1(u \to \cE^{\Rig}_{F}, Z(G_2) \to G_2)$. By construction, we have ${\psi'}^{-1}(z'_2), z'_1 \in Z^1(u \to \cE^{\Rig}_{F}, Z(G_1) \to S_1)$. The Tate--Nakayama map respects the group structure on cocycles valued in $S_2$. Moreover, $\psi'|_{S_1}$ is defined over $F$ and therefore induces an isomorphism $H^1(u \to \cE^{\Rig}_{F}, Z(G_1) \to S_1) \to H^1(u \to \cE^{\Rig}_{F}, Z(G_2) \to S_2)$ compatible with the Tate--Nakayama maps. We therefore have
\begin{equation*}
    \TN_{S_1}( {\psi'}^{-1}(z'_2)z'_1) = {\psi'}^{-1}\TN_{S_2}(z'_2) + \TN_{S_1}(z'_1).
\end{equation*}
Since the equation holds in $\pi_0(\hat{\bar{S_1}}^+)^D$, it is still true upon projection to $\pi_0(Z(\hat{\bar{G_1}})^+)^D$. This proves the lemma.
\end{proof}

\subsection{Isocrystals vs local rigid inner twists}{\label{ss:isocrystals}}

In this subsection, we explain how local rigid inner twists are related to isocrystals at $p$, cf. \cite{KalethaRigidvsIsoc}.

\subsubsection{Review of the Kottwitz set $B(G)$}{\label{BGreview}}

We recall some facts about the Kottwitz set \gls{BG} for $G$ a connected reductive group over a finite extension $F$ of $\Q_p$. Let $\breve F$ denote the completion of a maximal unramified extension of $F$, equipped with the arithmetric Frobenius automorphism $\sigma$. The set $B(G)$ is defined as the quotient of $G(\breve{F})$ by the equivalence relation that $b \sim b'$ if there exists $g \in G(\breve{F})$ such that $b' = g b\sigma(g)^{-1}$.

The set $B(G)$ is determined by two important invariants. The first attaches to each $b \in G(\breve{F})$, a \emph{slope cocharacter} $\gls{nub} \in \Hom_{\breve{F}}( \D, G)$, where $\D$ is the Kottwitz pro-torus with $X^*(\D) = \Q$. This induces a \emph{slope morphism}
\begin{equation*}
    \gls{[nu]} : B(G) \to (\Hom_{\breve{F}}( \D, G)/ G(\breve{F}) )^{\langle \sigma \rangle}.
\end{equation*}

We also have the \emph{Kottwitz morphism}
\begin{equation*}
    \gls{kappa} : B(G) \to \pi_1(G)_{\Gamma_F},
\end{equation*}
where we note that $\pi_1(G)_{\Gamma_F}$ is canonically isomorphic to $X^*(Z(\hat{G})^{\Gamma_F})$.

These morphisms fit into a commutative diagram as follows, which is functorial in $G$:
\begin{equation*}
    \begin{tikzcd}
    B(G) \arrow[r, "{[\nu]}"] \arrow[d, swap, "\kappa"] & (\Hom_{\breve{F}}( \D, G)/ G(\breve{F}) )^{\langle \sigma \rangle} \arrow[d]  \\
    \pi_1(G)_{\Gamma_F} \arrow[r] & \pi_1(G)^{\Gamma_F}_{\Q},
    \end{tikzcd}
\end{equation*}
where the bottom arrow is induced by $x\mapsto \frac{1}{|\Gamma_F \cdot x|} \sum\limits_{y \in \Gamma_F \cdot x} y$ from $\pi_1(G)$ to $\pi_1(G)^{\Gamma_F}_{\Q}$.

Given $b$, we define a reductive group \gls{J} over $F$ by
\begin{equation*}
    J_b(R) = \{ g \in G(\breve{F} \otimes_F R) \mid g= b \sigma(g)b^{-1}\} = \{ g \in G(\breve{F} \otimes_F R) \mid \nu_b= g \nu_b g^{-1}\},
\end{equation*}
for any $F$-algebra $R$. When $R$ is an $\ov{\breve{F}}$-algebra, the canonical $\ov{\breve{F}}$-algebra map $ \ov{\breve{F}} \otimes R  \to R$ induces an inclusion $\iota_b: J_b(R) \to G(R)$. When $R= \ov{\breve{F}}$, this inclusion is equivariant for the standard action of $W_F$ on $J_b(\ov{\breve{F}})$ and the ``$b$-twisted'' action on $G(\ov{\breve{F}})$ in the sense of \cite[(3.3.3)]{KottwitzIsocrystal2}.

Now fix a quasi-split group $G^*$ that is isomorphic to $G$ over $\ov{F}$ and an inner twist $(G, \psi, z_{\ad})$ of $G^*$. Fix also a maximal split torus $A^*$ of $G^*$ and let $T^*$ denote the maximal torus of $G^*$ equal to the centralizer of $A^*$ in $G^*$. Fix $B^*$ a Borel subgroup of $G^*$ containing $T^*$. Let \gls{CQ} denote the closed dominant Weyl chamber in $X_*(A^*)_{\Q}$ relative to $B$.

Given $b \in G(\breve{F})$, we can take the morphism
\begin{equation*}
    \psi^{-1} \circ \nu_b: \gls{D} \to G^*,
\end{equation*}
defined over $\ov{\breve{F}}$.
As in \cite[\S4.2]{KottwitzIsocrystal2}, we can take the $G^*(\ov{\breve{F}})$ conjugacy class of $\psi^{-1} \circ \nu_b$ and, since $G^*$ is quasi-split, this gives an element of $\ov{C}_{\Q}$ which we denote by $\ov{\nu}_b$. One can check this element does not depend on $\psi$ or the choice of $b$ in its class $[b] \in B(G)$. We call the element $\ov{\nu}_b$ the \emph{Newton point} of $b$ (or $[b]$). We have thus constructed a \emph{Newton map}
\begin{equation*}
    \gls{ovnu}: B(G) \to \ov{C}_{\Q}.
\end{equation*}
The centralizer $M^*_{\ov{\nu}_b}$ of $\ov{\nu}_b$ in $G^*$ is a standard $F$-rational Levi subgroup of $G^*$ and an inner form of $J_b$. To simplify notation, we henceforth denote $M^*_{\ov{\nu}_b}$ by \gls{Mb}. 
(Beware that $M_b$ need not transfer to an $F$-rational Levi subgroup of $G$, although this will be the case when $G$ itself is quasi-split.)
Following \emph{loc.~cit.}~we can choose $\psi'$ equivalent to $\psi$ such that $\psi' \circ \ov{\nu}_b = \nu_b$ and the restriction gives an inner twist
\begin{equation}\label{eq:M*Jb}
    \psi'_{M_b}: M_{b, \ol F} \to J_{b, \ol F},
\end{equation}
using the inclusion $J_b(R) \hookrightarrow G(R)$ discussed above.

Following \cite[Def.~1.8]{RapoportZink}, for $n \in \Z_{\geq 1}$, we say an element $b \in G(\breve{F})$ is $n$-\emph{decent} if $n\nu_b$ is a cocharacter of $G_{\breve{F}}$ and 
\begin{equation*}
    b \sigma(b) \dots \sigma^{n-1}(b) = (n\nu_b)(p).
\end{equation*}
One can always choose a decent representative of a class $[b] \in B(G)$ (\cite[\S4.3]{KottwitzIsocrystal1}). If $b$ is $n$-decent then it is $m$-decent for $n|m$. When $G$ is quasi-split, one can require that $\nu_b$ is defined over $F$ (\cite[page 219]{KottwitzIsocrystal1}).

There is another description of $B(G)$ using Galois gerbes, which we now discuss. The local Kottwitz gerbe $\cE^{\Iso}_F$ is a pro-Galois gerbe
\begin{equation*}
    1 \to \D(\ol{F}) \to \gls{cEiso} \to \Gamma_F \to 1.
\end{equation*}
We can then consider the set $H^1_{\alg}(\cE^{\Iso}_F, G)$ of continuous cocycles of $\cE^{\Iso}_F$ whose restriction to $\D(\ol{F})$ is induced by an algebraic morphism $\nu: \D \to G$. We define the subset $H^1_{\bas}(\cE^{\Iso}_F, G) \subset H^1_{\alg}(\cE^{\Iso}_F, G)$ to consist of those classes where the image of $\nu$ is contained in $Z(G)$.
We have a canonical bijection
\begin{equation}{\label{eq: BGHalgbij}}
    B(G) = H^1_{\alg}(\cE^{\Iso}_F, G)  ~ ; \quad \bb \mapsto z^{\Iso}_{\bb}
\end{equation}
given in \cite[Appendix B]{KottwitzIsocrystal2} and described explicitly in \cite[\S2.3]{HansenKalethaWeinstein} for decent $b$. There is a construction of the Kottwitz morphism for $H^1_{\alg}(\cE^{\Iso}_F, G)$ (see \cite[\S 11]{KottwitzglobalB(G)}), again denoted by $\kappa$, and equation \eqref{eq: BGHalgbij} preserves the Kottwitz maps: $\kappa(\bb) = \kappa(z^{\Iso}_{\bb})$. Indeed, one can check this via the usual strategy of Kottwitz of reducing successively to the $G_{\der}=G_{\tu{sc}}$ case, torus case, and $\G_m$ case (cf. \cite[\S7.6]{KottwitzIsocrystal2}).

\begin{defn}
An \emph{extended pure inner twist} of $G$ over $F$ consists of a triple $(G', \psi, z)$, where $z \in Z^1_{\alg}(\cE^{\Iso}_F, G)$ and $\psi: G_{\ol{F}} \to G'_{\ol{F}}$ such that $\psi^{-1} \circ \sigma(\psi) = \Int(z_{\ad}(\sigma))$ for all $\sigma \in \Gamma_F$.
\end{defn}

\subsubsection{Rigid inner twists vs extended pure inner twists}\label{sss:rigid-vs-epit}

Kaletha (\cite[\S3.3]{KalethaRigidvsIsoc}) constructs a morphism of Galois gerbes
\begin{equation*}
    \begin{tikzcd}
    1 \arrow[r] & u \arrow[r]  \arrow[d, "\phi"] & \cE^{\Rig}_F \arrow[r] \arrow[d, "\tilde{\phi}"] & \Gamma_F \arrow[r] \arrow[d, equals] & 1 \\
    1 \arrow[r] & \D  \arrow[r] & \cE^{\Iso}_F \arrow[r] & \Gamma_F \arrow[r] & 1,
    \end{tikzcd}
\end{equation*}
which induces a map 
\begin{equation*}
    \tilde{\phi}^*: H^1_{\bas}(\cE^{\Iso}_F, G) \to H^1(u \to \cE^{\Rig}_F, Z(G) \to G),
\end{equation*}
which is well defined on the level of cocycles. In particular, given an extended pure inner twist $(G', \psi, z_{\Iso})$, we get a rigid inner twist $(G', \psi, \gls{zrig})$ via pullback by $\tilde{\phi}$.

\subsection{Endoscopic data}\label{ss:endoscopic-data}

We review the various forms of endoscopy we use in this paper, primarily summarizing from \cite{BMaveraging}. In this subsection, $F$ is a local or global field.

\subsubsection{Endoscopic data}
We introduce three versions of endoscopic data for $G$, necessitated by different notions of inner twist: inner twists classified by $H^1(F, G_{\ad}),$ pure inner twists classified by $H^1(F,G)$ or extended pure inner twists classified by $B(G)_{\bas}$, and rigid inner twists classified by $H^1( u \to \cE^{\Rig}_F, Z(G) \to G)$.
\begin{defn}
A \emph{standard endoscopic datum} for $G$ is a tuple $(H, \cH, s,\eta)$ which consists of
\begin{itemize}
    \item a quasi-split group $H$ over $F$,
    \item an extension $\cH$ of $W_F$ by $\hat{H}$ such that the map $W_F \to \Out(\hat{H})$ coincides with $\rho_H$,
    \item an element $s \in Z(\widehat{H})$,
    \item an $L$-homomorphism $\eta: \cH \to \LL G$,
\end{itemize}
satisfying the conditions:
\begin{enumerate}
    \item we have $\eta(\hat{H}) = Z_{\hat{G}}(\eta(s))^{\circ}$, and
    \item $\Int(s) \circ \eta = a \cdot \eta$, where $a: W_F \to Z(\hat{G})$ is a trivial (resp.~locally trivial) $1$-cocycle when $F$ is local (resp.~global).
\end{enumerate}
\end{defn}
\begin{defn}
An isomorphism of endoscopic data from $(H, \cH, s, \eta)$ to $(H', \cH', s', \eta')$ is an element $g \in \hat{G}$ such that
\begin{enumerate}
    \item we have $(\Int(g) \circ \eta)(\cH) = \eta'(\cH')$, and 
    \item $\Int(g)(\eta(s)) = \eta'(s')$ modulo $Z(\hat{G})$.
\end{enumerate}
\end{defn}
As in \cite[p.19]{KottwitzShelstad}, an element $g \in \hat{G}$ giving an automorphism of endoscopic data determines an automorphism $\alpha_g \in \Aut_F(H)$. The automorphism of endoscopic data is said to be inner if $\alpha_g\in H_{\ad}(F)$. We write $\Out_F(H) := \Aut_F(H)/H_{\ad}(F)$. 
We denote the subgroup of $\Out_F(H)$ arising from automorphisms of endoscopic data by $\Out_F(H, \cH, s, \eta)$.

\begin{defn}{\label{def: refined endoscopic datum}}
    A \emph{refined endoscopic datum} is a standard endoscopic datum $(H,\cH,s,\eta)$ satisfying the further requirement that $s \in Z(\hat{H})^{\Gamma_F}$. This implies that condition (ii) is automatically satisfied. For an isomorphism of refined endoscopic data, we require the equality $\Int(g)(\eta(s)) = \eta'(s')$ in $\eta'(Z(\hat{H})^{\Gamma})$ rather than $\eta'(Z(\hat{H}))/Z(\hat{G})$. 
\end{defn}
    We will now define a notion of rigid endoscopy after some preliminaries. Given a standard endoscopic datum $(H, \cH, s, \eta)$, we define a group $\gls{hatbarH} = \varprojlim \hat{H_n}$ where $H_n = H/Z_n$ and \gls{Zn} is the pre-image in $Z(G)$ of $(Z(G)/Z(G_{\der}))[n]$. Note that we identify $Z(G)$ with a central subgroup of $H$ via $\eta$. There is a natural map $\hat{\bar{H}} \to \hat{H}$ and we define \gls{ZhatbarH+} to be the pre-image of $Z(\hat{H})^{\Gamma_F}$ in $Z(\hat{\bar{H}})$. 
    
    We claim that $\eta$ induces a natural map $\bar{\eta}: \hat{\bar{H}} \to \hat{\bar{G}}$ such that the following diagram commutes:
    \begin{equation*}
        \begin{tikzcd}
                \hat{\bar{H}} \arrow[d] \arrow[r, "{\bar{\eta}}"] & \hat{\bar{G}} \arrow[d] \\
                \hat{H} \arrow[r, "\eta"] & \hat{G}.
        \end{tikzcd}
    \end{equation*}
    By universal property, it suffices to show we have a system of maps $\hat{H_n} \to \hat{G_n}$ for each $n$ that are compatible with $\eta$. Let $\hat{H_n}'$ be the pre-image of $\eta(\hat{H})$ under the map $\hat{G_n} \to \hat{G}$. Then $\hat{H_n}'$ and $\hat{H_n}$ are both central extensions of $\hat{H}$ and so it suffices to show they are isomorphic as central extensions. Choose maximal tori $\hat{T} \subset \hat{G}$ and $\hat{T_H} \subset \hat{H}$ such that $\eta(\hat{T_H})= \hat{T}$. Let $\hat{T_n}$ and $\hat{T_{H,n}}$ be the pre-images of $\hat{T}$ in $\hat{H_n}'$ and $\hat{T_H}$ in $\hat{H_n}$, respectively. 
    We have natural maps $X_*(\hat{T_n}) \to X_*(\hat{T})$ and $X_*(\hat{T_{H,n}}) \to X_*(\hat{T_H}) \to X_*(\hat{T})$, and it suffices to show the images of these maps coincide. This is the case since they are both equal to $\ker( X_*(\hat{T}) \to X_*(\hat{Z_n}))$, where the map $X_*(\hat{T}) \to X_*(\hat{Z_n})$ comes from $G$ via the isomorphism of the pinned root datum of $\hat{G}$ with the dual of the canonical based root datum of $G$.
    
   \begin{defn} 
    A \emph{rigid} endoscopic datum $(H, \cH, \dot s,\eta)$ over a local field $F$ is defined in the same way as a standard datum (including conditions (i) and (ii)) except the requirement that $\dot s \in Z(\hat{\bar{H}})^+$. An isomorphism of rigid endoscopic data from $(H_1, \cH_1, \dot{s}_1, \eta_1)$ to $(H_2, \cH_2, \dot{s}_2, \eta_2)$ is an element $g \in \hat{G}$ such that 
    \begin{enumerate}
        \item $(\Int(g) \circ \eta_1)(\cH_1) = \eta_2(\cH_2)$, and 
        \item the images of $\dot{s}_i$ in $\pi_0(Z(\hat{\bar{H_i}})^+)$ coincide under the isomorphism
        \begin{equation*}
            \pi_0(Z(\hat{\bar{H_1}})^+) \xrightarrow{} \pi_0(Z(\hat{\bar{H_2}})^+),
        \end{equation*}
        induced by ${\ov{\eta}_2}^{-1} \circ \Int(g) \circ \ov{\eta}_1$.
    \end{enumerate}
\end{defn}

  We denote the set of standard (resp.~refined, resp.~rigid) endoscopic data by \gls{sE(G)}, (resp.~$\sE^{\Iso}(G)$, resp.~$\sE^{\Rig}(G)$). The corresponding sets of isomorphism classes of such data are denoted by \gls{mE(G)}, $\mE^{\Iso}(G)$, and $\mE^{\Rig}(G)$.

An endoscopic datum (of any type) is said to be \emph{elliptic} if $(Z(\hat{H})^{\Gamma_F})^\circ \subset Z(\hat{G})$. This condition determines the subsets $\sE^?_{\el}(G)\subset \sE^?(G)$ and $\mE^?_{\el}(G)\subset \mE^?(G)$ for $?\in \{\emptyset,\Iso,\Rig\}$.

\subsubsection{Endoscopic data and Levi subgroups}{\label{endoscopyandLevis}}
We will need to study the relation between the endoscopy of $G$ and its Levi subgroups. Fix a minimal parabolic subgroup $P_0 \subset G$ and suppose now that $M \subset G$ is a standard Levi subgroup of $G$. Fix a Borel subgroup $B \subset P_{0,\ol F} \subset G_{\ol{F}}$. We have a Levi subgroup $\hat{M} \subset \hat{G}$ determined by the set of simple roots of $\hat{G}$ corresponding to the simple absolute coroots determining $M$. We have the following notion of endoscopy for $M$ relative to $G$.
\begin{defn}\label{def:emb-end-datum}
    An \emph{embedded endoscopic datum} for $G$ is a tuple $(H_M, \cH_M, H, \cH, s, \eta)$, where 
\begin{itemize}
    \item $(H, \cH, s, \eta)$ is a refined endoscopic datum of $G$ with a fixed $F$-pinning $(T_H, B_H, \{X_{H,\alpha}\})$ of $H$,
    \item $H_M$ is a standard Levi subgroup of $H$,
    \item $\cH_M$ is a Levi subgroup of $\cH$, namely $\cH_M$ surjects onto $W_F$ and its intersection with $\hat H$ is a Levi subgroup of $\hat H$,
\end{itemize}    
    such that $\hat{H_M} = \cH_M \cap \hat{H}$ 
    and $(H_M, \cH_M, s, \eta|_{\cH_M})$ is a refined endoscopic datum of $M$. 
    
    An isomorphism of embedded data from $(H_M, \cH_M, H, \cH, s, \eta)$ to $(H'_M, \cH'_M, H', \cH', s', \eta')$ is a $g \in \hat{M}$, which simultaneously produces isomorphisms of refined endoscopic data
    $$(H_M, \cH_M, s, \eta) \isom (H'_M, \cH'_M, s', \eta')\quad \mbox{and} \quad (H, \cH, s, \eta) \isom (H', \cH', s', \eta').$$
    We denote the set of embedded endoscopic data by \gls{sEemb(M,G)} and the set of isomorphism classes by \gls{mEemb(M,G)}. An automorphism of embedded endoscopic data is said to be \emph{inner} if the associated automorphism of $(H_M, \cH_M, s, \eta)$ is an inner automorphism of endoscopic data.
    \end{defn}

We have the natural restrictions $X: \sE^{\emb}(M,G) \to \sE^{\Iso}(M)$ and $Y^{\emb}: \sE^{\emb}(M,G) \to \sE^{\Iso}(G)$. These induce maps of isomorphism classes, and the map induced by $X$ is a bijection 
by \cite[Prop.~2.20]{BMaveraging}. We recall from \cite[Construction 2.15]{BMaveraging} that there is a natural map $Y: \mE^{\Iso}(M) \to \mE^{\Iso}(G)$ such that the following diagram commutes
\begin{equation}
\begin{tikzcd}
&\mE^{\Iso}(G)&\\
\mE^{\emb}(M,G) \arrow[ur, " Y^{\emb}"] \arrow[rr, "X"] && \mE^{\Iso}(M) \arrow[ul, swap, "Y"].
\end{tikzcd}
\end{equation}

\begin{defn}
Fix a refined endoscopic datum $(H, \cH, s, \eta)$ of $G$. We define the set \gls{mEemb(M,G;H)} to be the set of isomorphism classes of embedded endoscopic data whose image under 
\begin{equation}
Y^{\emb}: \mE^{\emb}(M,G) \to \mE^{\Iso}(G)
\end{equation}
is the isomorphism class of $(H,\cH, s,\eta)$. Note that by \cite[Prop. 2.25]{BMaveraging}, every class in $\mE^{\emb}(M,G;H)$ has a (not necessarily unique) representative of the form $(H_M, \cH_M, H, \cH, s, \Int(n) \circ \eta)$ for $n \in N_{\hat{G}}(\hat{T})$.
We define the set of \emph{inner classes of embedded endoscopic data} relative to $H$, denoted by \gls{mEi(M,G;H)}, to be the set of equivalence classes of embedded data of the form $(H_M, \cH_M, H, \cH, s, \Int(n) \circ \eta)$. Two such data are considered equivalent if they are isomorphic by an inner isomorphism $\alpha$ of the group $H$ inducing an isomorphism of embedded endoscopic data. Note that the $\sE^{\emb}(M,G)$-isomorphism class of an element of $\mE^i(M,G;H)$ lies in $\mE^{\emb}(M,G;H)$ and that $\alpha$ need not induce an inner isomorphism of embedded endoscopic data.
\end{defn}

\subsubsection{Endoscopy and semi-simple conjugacy classes}\label{sss:endoscopy-and-ss-conj}

In this section we recall a number of results that will be of use to us. We largely follow \cite{BMaveraging}, whose results are based on \cite{KottwitzEllipticSingular} and \cite{Shin10}. Our notation is chosen to be similar to that of \cite{KSZ}. 
We assume throughout that $F$ is either a local or global field and $G$ is such that  $G_{\der} = G_{\scusp}$. If $M \subset G$ is a Levi subgroup and $T \subset M$ is a maximal torus, we say $\gamma \in T(\ol F)$ is $(G,M)$-\emph{regular} if for each root $\alpha \in R(G,T) \setminus R(M,T)$, we have $\alpha(\gamma) \neq 1$. Suppose $(H, \cH, s,  \eta)$ is an endoscopic datum. Let $\gamma_H\in H(\ol F)_{\tu{ss}}$. Choose a maximal torus $T_H \subset H_{\ol F}$ containing $\gamma_H$. We have a transfer $T_H\cong T\subset G_{\ol F}$ for a maximal torus $T$ of $G$. Write $\gamma\in T(\ol F)$ for the image of $\gamma_H$. Then $\gamma_H$ is said to be \emph{$(G,H)$-regular} if for each $\alpha \in R(G,T) \setminus R(H, T_H)$, we have $\alpha(\gamma) \neq 1$; if $\gamma$ is a fortiori strongly regular in $G(\ol F)$ then $\gamma_H$ is said to be \emph{$G$-strongly regular}. The definition of $(G,M)$-regular, $(G,H)$-regular, and $G$-strongly regular elements is independent of all auxiliary choices. We use the subscripts $(G,M)\tu{-reg}$, $(G,H)\tu{-reg}$, and $G\tu{-sr}$ to denote the subsets of such elements. By definition, $\gls{H(F)G-sr}\subset \gls{H(F)(G,H)-reg} \subset \gls{H(F)ss}$.

\begin{defn}\label{def:EK(G)}
We define \gls{EKiso(G)} to be the set of equivalence classes of pairs $(\gamma, \lambda)$ consisting of $\gamma \in G(F)_{\semis}$ and $\lambda \in Z(\widehat{I_{\gamma}})^{\Gamma_F}$, with $I_{\gamma} := Z_{G}(\gamma)^{\circ}$, where two pairs  $(\gamma, \lambda), (\gamma', \lambda')$ are equivalent if $[\gamma] =[\gamma'] \in \Sigma(G(F))$ and $\lambda$ corresponds to $\lambda'$ under the canonical isomorphism $Z(\widehat{I_{\gamma}})^{\Gamma_F} \cong Z(\widehat{I_{\gamma'}})^{\Gamma_F}$ \cite[\S 3]{KottwitzEllipticSingular}. We define the subset $\EK^{\Iso}_{\el}(G)\subset \EK^{\Iso}(G)$ by the condition that $\gamma \in G(F)_{\el}$.

For $M$ a Levi subgroup of $G$, we let \gls{EKiso(M,G)} denote the set of equivalence classes of pairs $(\gamma, \lambda)$ consisting of $\gamma \in M(F)_{(G,M)\tu{-reg}}$ and $\lambda \in Z(\widehat{I^M_{\gamma}})^{\Gamma_F}$, with $I^M_{\gamma}:=Z_M(\gamma)^\circ$, where two pairs $(\gamma, \lambda), (\gamma', \lambda')$ are equivalent if $[\gamma] = [\gamma'] \in \Sigma(M(F))$ and if $\lambda$ is sent to $\lambda'$ under the canonical isomorphism $Z(\widehat{I^M_{\gamma}})^{\Gamma_F} \cong Z(\widehat{I^M_{\gamma'}})^{\Gamma_F}$. In particular, $\EK^{\Iso}(M,G) \subset \EK^{\Iso}(M)$.
\end{defn}
\begin{defn}
We define $\ES^{\Iso}(G)$ to be the set of equivalence classes of tuples $(H, \cH,  s, \eta, \gamma_H)$ such that $(H, \cH, s, \eta) \in \sE^{\Iso}(G)$, $\gamma_H \in H(F)_{(G,H)\tu{-reg}}$, and $\gamma_H$ transfers to $G(F)$. Two tuples $(H, \cH, s, \eta, \gamma_H)$ and $(H', \cH', s', \eta', \gamma_{H'})$ are equivalent if there is an isomorphism $g \in \hat{G}$ of the refined data $(H, \cH, s, \eta)$ and $(H',\cH', s', \eta')$ and if the isomorphism $\alpha: H \to H'$ induced by $g$ in the sense of \cite[p.19]{KottwitzShelstad}  satisfies that $[\alpha(\gamma_H)]=[\gamma'_H] \in \Sigma(H'(F))$. We remark that $\alpha$ is only well-defined up to an inner automorphism of $H$, but that this does not matter for our purposes since we are only interested in stable conjugacy. We define the subset $\ES^{\Iso}_{\el}(G)\subset \ES^{\Iso}(G)$ by requiring that $(H, \cH, s, \eta) \in \sE^{\Iso}_{\el}(G)$ and $\gamma_H \in H(F)_{\el}$. 

Write \gls{ESiso(M,G)} for the set of equivalence classes of tuples $(H_M, \cH_M, s_M, \eta_M, \gamma_{H_M})$ such that $(H_M, \cH_M, s_M, \eta_M) \in \sE^{\Iso}(M)$ and $\gamma_{H_M}$ is an element of $H_M(F)_{(M,H_M)\tu{-reg}}$ that transfers to an element of $M(F)_{(G,M)\tu{-reg}}$. Two tuples $(H_M, \cH_M, s_M, \eta_M, \gamma_{H_M})$ and $(H'_M, \cH'_M, s'_M, \eta'_M, \gamma_{H'_M})$ are considered equivalent if there exists an isomorphism $m \in \hat{M}$ of the refined endoscopic data $(H_M, \cH_M, s_M, \eta_M)$ and $(H'_M, \cH'_M, s'_M, \eta'_M)$ such that the associated isomorphism $\alpha: H_M \to H'_M$ satisfies that $[\alpha(\gamma_{H_M})] = [\gamma_{H'_M}] \in \Sigma(H'_M(F))$. Note that $\ES^{\Iso}(M,G) \subset \ES^{\Iso}(M)$. 

The analogue of \gls{ESiso(M,G)} defined for embedded endoscopic data is denoted \gls{ESemb(M,G)}. By proof of \cite[Prop.~2.20]{BMaveraging}, we have  a natural identification $\ES^{\emb}(M,G)=\ES^{\Iso}(M,G)$.
\end{defn}

Define a natural map 
\begin{equation}{\label{EKESbij}}
 \ES^{\Iso}(G) \to \EK^{\Iso}(G),   
\end{equation}
as follows. Given $(H, \cH, s, \eta, \gamma_H) \in \ES^{\Iso}(G)$, we let $\gamma \in G(F)$ be a transfer of $\gamma_H$. 
(Recall that the existence of the transfer is built into the definition of $\ES^{\Iso}(G)$.) Since $\gamma_H$ is $(G,H)$-regular, we have a canonical $\Gamma_F$-equivariant isomorphism $Z(\widehat{I_{\gamma_H}})\cong Z(\widehat{I_{\gamma}})$. In particular, since $s \in Z(\widehat{H})^{\Gamma_F}\subset Z(\widehat{I_{\gamma_H}})^{\Gamma_F}$, it gives an element $\lambda\in Z(\widehat{I_{\gamma}})^{\Gamma_F}$. Thereby we obtain $(\gamma,\lambda)\in \EK^{\Iso}(G)$. If $(H', s', \eta', \gamma_{H'})$ is equivalent to $(H, s, \eta, \gamma_H)$, then the construction assigns the same element $(\gamma,\lambda)\in\EK^{\Iso}(G)$. This map \eqref{EKESbij} is a bijection (\cite[Lem.~2.30]{BMaveraging}, cf. \cite[Lem.~9.7]{KottwitzEllipticSingular}). Analogously we have a bijection
$\ES^{\Iso}_{\el}(G) \xrightarrow{\sim} \EK^{\Iso}_{\el}(G)$.
    
Now, since the $\gamma$ for any $(\gamma, \lambda) \in \EK^{\Iso}(M,G)$ is assumed to be $(G,M)$-regular, we have an equality $I^M_{\gamma}=I^G_{\gamma}$. Hence we get an obvious map
\begin{equation}
    \EK^{\Iso}(M,G) \to \EK^{\Iso}(G).
\end{equation}

The various objects we have defined fit into a commutative diagram (\cite[Lem.~2.36]{BMaveraging}):

\begin{equation}{\label{EQSSdiagram}}
    \begin{tikzcd}
    \EK^{\Iso}(G) \arrow[rr, "\sim"] & & \ES^{\Iso}(G)  \arrow[r] & \mE^{\Iso}(G)\\
    \EK^{\Iso}(M,G) \arrow[u] \arrow[r, "\sim"] & \ES^{\Iso}(M,G) \arrow[r, equal] & \ES^{\emb}(M,G) \arrow[r] \arrow[u] & \mE^{\emb}(M, G) \arrow[u, swap, "Y^e"],\\
    \end{tikzcd}
\end{equation}
   where the middle vertical map takes $( H_M, \cH_M, H, \cH, s, \eta, \gamma_{H_M})$ to $(H, \cH, s, \eta, \gamma_{H_M})$, where $\gamma_{H_M}$ is realized as an element of $H(F)$ via the inclusion $H_M \subset H$.

\subsection{Local transfer factors}\label{ss:local transfer factors}

We review the transfer factors normalized by a rigid endoscopic datum. In this section we assume that $\cH$ can be taken to be $\LL H$. This is satisfied, for instance, when $G_{\der}$ is simply connected. In general, one can reduce to this case using $z$-extensions.

Fix $G^*$ a quasi-split reductive group over a local field $F$ of characteristic $0$ and $(G, \psi, z^{\Rig})$ a rigid inner twist of $G^*$ over $F$. Fix also a Whittaker datum $\fkw$ for $G^*$, namely a $G^*(F)$-conjugacy class of pairs $(B, \theta)$, where $B$ is a Borel subgroup of $G^*$ defined over $F$ with unipotent radical $U$, and $\theta: U(F) \to \C^{\times}$ is a non-degenerate character. 
Let $(H, \cH, \dot{s}, \eta)$ be a rigid endoscopic datum. As in  \cite{KottwitzShelstad}, there is a Whittaker normalized transfer factor 
\begin{equation*}
    \gls{Deltaw}: H(F)_{G\tu{-}\sr} \times G^*(F)_{\sr} \to \C.
\end{equation*}
Specifically, we are using the $\Delta^{\lambda}_D$ normalization of \cite[\S5.5]{KottwitzShelstadSigns}. 

\subsubsection{Local rigid transfer factors}{\label{sss:transfer factors construction}}
Following \cite[\S 5.3]{KalethaLocalRigid}, we construct a transfer factor
\begin{equation*}
    \gls{deltawrig}: H(F)_{G\tu{-}\sr} \times G(F)_{\sr} \to \C,
\end{equation*}
 by the formula
\begin{equation}{\label{eq: rigtransferinnerforms}}
    \Delta[\fkw, z^{\Rig}](\gamma_H, \gamma) = \Delta[\fkw](\gamma_H, \gamma^*)\langle \gls{inv[zrig]}(\gamma^*, \gamma), \dot{s} \rangle^{-1}.
\end{equation}
We explain the notation in the above equation. Choose $\gamma^* \in G^*(F)$ such that $\psi(g\gamma^*g^{-1}) = \gamma$ for some $g \in G^*(\ol{F})$. Then one checks that if $T = Z_{G^*}(\gamma^*)$, then $w \mapsto g^{-1}z^{\Rig}(w)w(g)$ gives a cocycle in $Z^1(u \to \cE^{\Rig}_F, Z(G^*) \to  T)$, whose cohomology class is independent of the choice and denoted $\inv[z^{\Rig}](\gamma^*, \gamma)$. We have a canonical $\Gamma_F$-equivariant embedding $Z(\hat{H}) \to \hat{T}$ which induces a map $Z(\hat{\bar{H}})^+ \to \hat{\bar{T}}^+$. 
Then the pairing $\langle \inv[z^{\Rig}](\gamma^*, \gamma), \dot{s} \rangle$ is the Tate--Nakayama pairing of Proposition \ref{TateNakayama}, where we identify $\dot{s}$ with its image in $\hat{\bar{T}}^+$.

\subsubsection{The $(G,H)$-regular case}{\label{ss:GHregsection}}

We described our normalization of transfer factors in \S\ref{sss:transfer factors construction}.
In order to use these factors in the stabilization of the trace formula for the cohomology of Igusa varieties, we need to extend the transfer factors to functions
\begin{equation}\label{eq:transfer-factor-(G,H)-reg}
    \Delta[\fkw, z^{\Rig}]: H(F)_{(G,H)\tu{-}\reg} \times G(F)_{\semis} \to \C.
\end{equation}
In \cite{BMadelic} these factors were defined for refined endoscopic data via an explicit construction of the invariant $\inv[z](\gamma^*, \gamma)$ in the non-strongly regular case. This construction could in theory be carried out in the rigid setting but to do so compatibly at each localization of a global field would involve the cohomology of a global gerb that is an ``adelic form'' of $\cE^{\Rig}_{\dot{V}}$. Unfortunately, this would require a substantial amount of preparation and so we have instead opted for a different approach.

Given a transfer factor defined on $H(F)_{G-\sr} \times G(F)_{\sr}$, Langlands and Shelstad \cite[\S2.4]{LanglandsShelstadDescent} showed that one can extend the transfer factor to the domain $H(F)_{(G,H)\tu{-}\reg} \times G(F)_{\semis}$ via a limiting operation. In particular, given $(\gamma_H, \gamma) \in H(F)_{(G,H)\tu{-}\reg} \times G(F)_{\semis}$ such that $\gamma_H$ transfers to $\gamma$, one chooses a sequence of pairs $(\gamma_{H,i}, \gamma_i) \in H(F)_{G-\sr}\times G(F)_{\sr}$ such that $\gamma_{H_i}$ transfers to $\gamma_i$ and $(\gamma_{H_i}, \gamma_i)$ converges to $(\gamma_H, \gamma)$. One then defines
\begin{equation*}
    \Delta[\fkw, z^{\Rig}](\gamma_H, \gamma)= \lim\limits_{i \to \infty} \Delta[\fkw, z^{\Rig}](\gamma_{H,i}, \gamma_i),
\end{equation*}
and shows that this limit does not depend on the choice of convergent sequence. The factor $\Delta[\fkw]$ can be extended analogously. We then define $\langle \inv[z^{\Rig}](\gamma^*, \gamma), \dot{s} \rangle := \Delta[\fkw](\gamma_H, \gamma^*) / \Delta[\fkw, z^{\Rig}](\gamma_H, \gamma) $. It follows from the continuity of $\Delta[\fkw]$ and $\Delta[\fkw, z^{\Rig}]$ and the definition of $\inv[z^{\Rig}]$ on the strongly regular locus that $\langle \inv[z^{\Rig}](\gamma^*, \gamma), \dot{s} \rangle$ does not depend on $\gamma_H$.
Of course we set $\Delta[\fkw, z^{\Rig}](\gamma_H, \gamma)=0$ if $\gamma_H$ does not transfer to $\gamma$.

In the case where $G_{\der}=G_{\scusp}$, we can describe $\inv[z^{\Rig}](\gamma^*, \gamma)$ for $\gamma \in G(F)_{\semis}$ more explicitly. Then the centralizer of $\gamma^*$ in $G^*$ is connected. So if $g \in G^*(\ov{F})$ is such that $\psi(g \gamma^*g^{-1}) = \gamma$, then $w \mapsto g^{-1}z^{\Rig}(w)w(g)$ gives a cocycle $[z^{\Rig}]_{\gamma^*, \gamma} \in Z^1(u \to \cE^{\Rig}_F, Z(G^*) \to I_{\gamma^*})$. By a continuity argument analogous to \cite[Lem.~3.8]{BMadelic} but for rigid inner twists, we have that
\begin{equation}{\label{eq: explicit inv}}
    \langle \inv[z^{\Rig}](\gamma^*, \gamma), \dot{s} \rangle = \langle [z^{\Rig}]_{\gamma^*, \gamma}, \dot{s} \rangle.
\end{equation}

We also record an equivariance property of (any normalization of) transfer factors.

\begin{lem}\label{lem:equivariance}
There exists a smooth character $\gls{lambdaH}:Z_G(F)\ra \C^\times$ such that for all $x\in Z_G(F)$, $\gamma_H\in H(F)_{(G,H)\tu{-}\reg}$, and $\gamma\in G(F)_{\semis}$,
$$\Delta[\fkw, z^{\Rig}]( x \gamma_H,x\gamma) = \lambda_H(x) \Delta[\fkw, z^{\Rig}](  \gamma_H,\gamma).$$
\end{lem}
\begin{proof}
This is \cite[Lem.~3.5.A]{LanglandsShelstadDescent}. (The proof is given for strongly regular elements, but obviously extends to the $(G,H)$-regular case by the limit formula above.)
\end{proof}

\begin{rem}\label{rem:lambdaH}
By \cite[Lem.~7.4.6]{KSZ}, the restriction of $\lambda_H$ to $Z^{\circ}_G(F)$ corresponds via class field theory to the $L$-morphism
$$W_F \ra {}^L H \stackrel{\eta}{\ra} {}^L G \ra {}^L Z^{\circ}_G,$$
where the last map is dual to the inclusion $Z^{\circ}_G \ra G$.
\end{rem}

\subsubsection{Comparison of local transfer factors}
We describe the relationship between isocrystal and rigid transfer factors for local $F$. To do so, we fix $\gls{ziso} \in Z^1_{\bas}(\cE^{\Iso}_F, G)$  and denote by $z^{\Rig} \in Z^1(u \to \cE^{\Rig}_F, Z(G) \to G)$ the pullback (see ~\S\ref{sss:rigid-vs-epit}). In analogy with \S\ref{sss:transfer factors construction}, we can define for $\gamma^* \in G^*(F)_{\sr}$ and $\gamma \in G(F)_{\sr}$ an invariant $\gls{inv[ziso]}(\gamma^*, \gamma) \in Z^1_{\bas}(\cE^{\Iso}_F, T)$. Following \S\ref{ss:GHregsection}, this naturally extends to an invariant map defined on $\gamma^* \in G^*(F)_{\semis}$ and $\gamma \in G(F)_{\semis}$. 
\begin{lem}{\label{rig vs iso transfer factors}}
Choose $\dot{s} \in \hat{\bar{T}}^+$ and let $s \in \hat{T}^{\Gamma_F}$ be the projection of $\dot{s}$. Then we have
\begin{equation*}
\langle \inv[z^{\Iso}](\gamma^*, \gamma), s \rangle = \langle \inv[z^{\Rig}](\gamma^*, \gamma), \dot{s} \rangle,
\end{equation*}
which implies
\begin{equation*}
    \Delta[\fkw, z^{\Iso}](\gamma_H, \gamma) = \Delta[\fkw, z^{\Rig}](\gamma_H, \gamma).
\end{equation*}
\end{lem}
\begin{proof}
This is proven in \cite[p.15]{KalethaRigidvsIsoc} for strongly regular elements. This immediately implies equality of the transfer factors on all of $H(F)_{(G,H)\tu{-}\reg} \times G(F)_{\semis}$.
\end{proof}
\subsection{Global transfer factors}{\label{ss:global transfer factors construction}}

We recall how a global rigid inner twist gives a decomposition of the global transfer factor as the product of local transfer factors.

\subsubsection{The strongly regular case}
\label{sss:global-strongly-regular}
Now suppose that $F$ is a number field and that $G^*$ is a quasi-split connected reductive group over $F$. Let $\psi: G^*_{\ol F} \to G_{\ol F}$ be an inner twist and $(H, \cH, s, \eta)$ a standard endoscopic datum for~$G$. As in the last section, we assume that $\cH={} ^L H$. Let $\fkw$ be a global Whittaker datum. There is the canonical adelic transfer factor of \cite[\S6]{LanglandsShelstad} and \cite[\S7.3]{KottwitzShelstad}
\begin{equation*}
    \Delta_{\A}: H(\A_F)_{G\tu{-}\sr} \times G(\A_F)_{\sr} \to \C,
\end{equation*}
characterized as follows.
Whenever $\gamma_H\in H(F)_{G\tu{-}\sr}$ transfers everywhere locally to an element $\gamma\in G(\A_F)$, we have $\Delta_{\A}(\gamma_H, \gamma) = \langle \tu{obs}(\gamma)^{-1},\kappa \rangle$ in the notation of \cite[\S7.3]{KottwitzShelstad} (Cor.~7.3.B in \emph{loc.~cit.}) except that our $\tu{obs}$ is the inverse of theirs and hence is consistent with \cite{LanglandsShelstad} and \cite{KalethaTaibi} (see \cite[Remark 4.2.2 ]{KalethaTaibi}). If no $\gamma_H\in H(F)_{G\tu{-}\sr}$ transfers to $G(F)$, then $(H,\cH,s,\eta)$ is irrelevant to stabilization. (The same assumption is made in \cite[\S6.3]{LanglandsShelstad}.) More precisely, in this case, the stable orbital integrals on $H(F)$ will be identically zero in the proof of Lemma \ref{lem:SO-adelic} below by the local transfer identity of orbital integrals. Thus there is no need to consider the global normalization or canonical adelic transfer factor. Henceforth we assume that some $\gamma_H\in H(F)_{G\tu{-}\sr}$ transfers $G(F)$.

For each place $v \in V$, we want to normalize the local transfer factor at $v$ such that the product over all places of the local transfer factors is $\Delta_{\A}$. We follow the process described in \cite[\S4.2]{KalethaGlobalRigid} with the only difference being that we work with the set $H^1(u \to \cE^{\Rig}_F, Z(G) \to G)$ instead of $H^1(u \to \cE^{\Rig}_F, Z(G_{\der}) \to G)$.

In particular, we let $s_{\ad} \in \hat{G}_{\ad}$ be the image of $\eta(s)$ in $\hat{G}_{\ad}$ and let $s_{\scusp} \in \hat{G}_{\scusp}$ be a lift of $s_{\ad}$ and $s_{\der}$ be the image of $s_{\scusp}$ in $\hat{G}_{\der}$.  Then for each $v \in \dot{V}$, we can view $s_{\der}$ as an element of $(\hat{G_v})_{\der}$ and there is a $y_v \in Z(\hat{G_v})$ such that $s_{\der} y_v \in \eta(Z(\hat{H_v})^{\Gamma_{F_v}})$.
Then we can write $y_v = y'_v \cdot y''_v$ for $y'_v \in Z((\hat{G_v})_{\der})$ and $y''_v \in Z(\hat{G_v})^{\circ}$. We lift $y'_v$ to $\gls{yv} \in Z((\hat{G_v})_{\scusp})$ and lift $y''_v$ to $\dot{y}''_v \in \hat{C_{\infty}}$ via the map $\hat{C_{\infty}} \to \hat{Z(G_{v,1})} = Z(\hat{G_v})^{\circ}$.  Then we define $\gls{s} \in Z(\hat{\bar{H_v}})$ to be the pre-image under $\bar{\eta}$ of  $(s_{\scusp}  \dot{y}'_v, \dot{y}''_v) \in  (\hat{G_v})_{\scusp} \times \hat{C_{\infty}} = \hat{\bar{G_v}}$ and observe that $\dot{s}_v \in Z(\hat{\bar{H_v}})^+$. This gives rise to a local rigid endoscopic datum $(H_v, \cH_v, \dot{s}_v, \eta_v)$, where $\cH_v$ is the pre-image of $W_{F_v}$ under the projection $\cH \to W_F$. Note that we have a natural map 
$$\sE^{\Rig}(G) \to \sE(G),\qquad (H_v, \cH_v, \dot{s}_v, \eta_v) \mapsto (H_v, \cH_v, s_v, \eta_v),$$
where $s_v$ is the projection of $\dot{s}_v$ to $Z(\hat{H})^{\Gamma_{F_v}}$. The data $(H_v, \cH_v, s_v, \eta_v)$ and $(H_v, \cH_v, s, \eta_v)$ are isomorphic in $\sE(G)$ since $s_v$ and $s$ differ by an element of $Z(\hat{G_v})$.

By Lemma \ref{H1surjectionlem} and after potentially replacing $\psi$ with an inner twist in the same equivalence class, we can lift the cocycle in $Z^1(F, G^*_{\ad})$ corresponding to $\psi$ to a cocycle $z^{\Rig} \in Z^1(P^{\Rig}_{\dot{V}} \to \cE^{\Rig}_{\dot{V}}, Z(G^*_{\scusp}) \to G^*_{\scusp})$. For each $v \in \dot{V}$, we then let $z^{\Rig}_{\scusp, v}$ be the localization of $z^{\Rig}$ as in \S \ref{sss:localization} and let $z^{\Rig}_v$ be the image in $Z^1(u_v \to \cE^{\Rig}_{F_v}, Z(G^*) \to G^*)$. Then $(G_{F_v}, \psi_v, z^{\Rig}_v)$ is a rigid inner twist of $G^*_{F_v}$.

Finally, we have local transfer factors $\Delta[\fkw_v, z^{\Rig}_v]: H(F_v)_{G\tu{-}\sr} \times G(F_v)_{\sr} \to \C$. Following the proof of \cite[Prop.~4.4.1]{KalethaGlobalRigid} we get that
\begin{equation}\label{eq:Delta-adelic}
    \Delta_{\A}(\gamma_H, \gamma) = \prod\limits_{v \in \dot{V}}  \langle z^{\Rig}_{\scusp, v}, \dot{y}'_v \rangle \Delta[\fkw_v, z^{\Rig}_v](\gamma_{H,v}, \gamma_v),
\end{equation}
for all $(\gamma_H, \gamma) \in H(F)_{\sr} \times G(F)_{\sr}$. 

\begin{rem}\label{rem:y'=1}
Suppose that $G_{\der}=G_{\scusp}$. (We always reduce to this case via $z$-extensions.) Then $Z(\hat G)$ is connected, so we can take $y_v=y''_v$, $y_v'=1$, and $\dot y_v'=1$ at every $v$ above. Then \eqref{eq:Delta-adelic} simplifies as $\langle z^{\Rig}_{\scusp, v}, \dot{y}'_v \rangle=1$ for all $v$. 
\end{rem}

\subsubsection{The $(G,H)$-regular case}\label{sss:global-GH-regular}

A semisimple element of $H(\A_F)$ is $(G,H)$-regular if it is at each place of $F$. As in the local case, the subscript ``$(G,H)$-reg'' means the subset of $(G,H)$-regular elements.
We define the adelic transfer factor 
$$\Delta[\fkw_v, z^{\Rig}_v]: H(\A_F)_{(G,H)\tu{-reg}} \times G(\A_F)_{\semis} \to \C$$
by extending  \eqref{eq:Delta-adelic} from the strongly regular case. The idea is to take \eqref{eq:Delta-adelic} as a definition when $\gamma_H$ is $(G,H)$-regular but not strongly $G$-regular:
$$  \Delta_{\A}(\gamma_H, \gamma) := \prod\limits_{v \in \dot{V}}  \langle z^{\Rig}_{\scusp, v}, \dot{y}'_v \rangle \Delta[\fkw_v, z^{\Rig}_v](\gamma_{H,v}, \gamma_v).$$
It follows from our definition that the factor $\Delta_{\A}$ defined on $H(\A_F)_{(G,H)\tu{-reg}} \times G(\A_F)_{\semis}$ equals the unique continuous extension of $\Delta_{\A}$ defined on $H(\A_F)_{G-\sr} \times G(\A_F)_{\sr}$. We have the product formula $\Delta_{\A}(\gamma_H,\gamma)=\langle \tu{obs}(\gamma)^{-1},\kappa \rangle$ on $H(F)_{(G,H)\tu{-reg}} \times G(\A_F)_{\semis}$ in the notation of \cite[(6.10.1)]{KottwitzEllipticSingular}, as follows from \cite[Lem.~4.1.(i)]{ArthurGlobalDescent} and the fact that the product formula holds in the strongly regular case. As $\tu{obs}(\gamma)$ vanishes on $\gamma\in G(F)$, we have
$$\Delta_{\A}(\gamma_H,\gamma)=1\quad\mbox{if}~\gamma_H\in H(\A_F)_{(G,H)\tu{-reg}}~\mbox{transfers to}~\gamma\in G(F).$$
The equivariance of the local transfer factors (Lemma \ref{lem:equivariance}) implies that there exists a continuous character $\lambda_H:Z_G(F)\backslash Z_G(\A_F)\ra \C^\times$ such that for $\gamma_H\in  H(\A_F)_{(G,H)\tu{-reg}}$ and $\gamma\in  G(\A_F)_{\semis}$,
$$\Delta_{\A}(x\gamma_H, x\gamma) = \lambda_H(x) \Delta_{\A}(\gamma_H, \gamma),
\qquad x\in Z_G(\A_F).
$$ 
More precisely, Lemma \ref{lem:equivariance} tells us that there exists a character $\lambda_H$ of $Z_G(\A_F)$ with the above property, and we should verify that $\lambda_H|_{Z_G(F)}$ is trivial. To check this, fix $\gamma_H\in H(F)_{G\tu{-sr}}$ and its transfer $\gamma\in G(F)$; such a pair exists by the running assumption from \S\ref{sss:global-strongly-regular}. Observe that \cite[\S\S3.4--3.5]{LanglandsShelstadDescent} can be adapted to the adelic setting by choosing global $a$-data and $\chi$-data and an admissible embedding $Z_H(\gamma_H)\ra Z_G(\gamma)$ over $F$. Then the desired triviality follows from the fact that a continuous cohomology class $ \mathbf{a}\in H^1(W_F,\hat T)$ in \emph{loc.~cit.}~gives rise to a character of $T(\A_F)$ that is trivial on $T(F)$, cf.~\cite[Thm.~2(b)]{LanglandsAbelian}. As the global analogue of Remark \ref{rem:lambdaH}, the restriction of $\lambda_H$ to $Z_G^{\circ}(F)\backslash Z^{\circ}_G(\A_F)$ corresponds to the composite $L$-morphism
$$W_F \ra {}^L H \stackrel{\eta}{\ra} {}^L G \ra {}^L Z^{\circ}_G.$$ 

\subsection{Local transfer}
We recall the endoscopic transfer of functions when $F$ is local. This readily allows us to transfer functions in the adelic setting, possibly away from finitely many places.

\begin{defn}
A \emph{(local) central character datum} is a pair \gls{(X,chi)}, where
\begin{itemize}
    \item $\fkX$ is a closed subgroup of $Z_G(F)$ equipped with a Haar measure,
    \item $\chi:\fkX\ra \C^\times$ is a smooth character.
\end{itemize}
\end{defn}

\subsubsection{}\label{sss:Hecke-alg-chi-inv}

Fix a Haar measure on $G(F)$ and a maximal compact subgroup $K$ of $G(F)$. Write $\cH(G,\chi^{-1})$ for the space of smooth bi-$K$-finite functions $f$ on $G(F)$ which are compactly supported modulo $\fkX$ and satisfy $f(xg)=\chi^{-1}(x)f(g)$ for $x\in \fkX$ and $g\in G(F)$. (If $F$ is nonarchimedean, the $K$-finiteness is automatic.)  In the case $\fkX=\{1\}$, we simply write $\cH(G)$; this is the usual Hecke algebra.

For each $f\in \cH(G,\chi^{-1})$ and $\gamma\in G(F)_{\semis}$, the orbital integral 
\begin{equation}\label{eq:def-orbital-integral}
    \gls{OGgamma(f)}=\int_{G_\gamma^{\circ}(F)\backslash G(F)} f(g^{-1}\gamma g) dg
\end{equation}
is defined verbatim as in the case $f\in\cH(G)$, where the quotient measure on $G_\gamma^{\circ}(F)\backslash G(F)$ is taken with respect to a Haar measure on $G_\gamma^{\circ}(F)$ (which is in practice chosen to be compatible with a Haar measure on another group). Similarly, the definition of stable orbital integrals \gls{SOGgamma(f)} is unchanged from the case $f\in \cH(G)$. When $G$ is clear from the context, we drop it from the notation.

For each $f\in \cH(G,\chi^{-1})$ and an irreducible admissible representation $\pi$ of $G(F)$ whose central character on $\fkX$ is $\chi$, we have the linear operator $v\mapsto \int_{G(F)/\fkX} f(g) (\pi(g) v) dg$ on the underlying space for $\pi$. 
Its trace has finite value and is to be denoted by
$$\tr_{\fkX} (f|\pi).$$

\subsubsection{The Langlands--Shelstad transfer}\label{sss:LS-transfer}

Let $\fke=(H,\cH,s,\eta)$ be an endoscopic datum for $G$ satisfying $\cH={}^L H$. Via the canonical embedding $Z_G\hra Z_H$, we identify $\fkX$ with a closed subgroup of $Z_H(F)$. Let $\lambda_H$ be as in Lemma \ref{lem:equivariance}.
Define a smooth character $\chi_H:=\chi \lambda_H^{-1}$ on $\fkX$. Then $(\fkX,\chi_H)$ is a central character datum for $H$.
Fix Haar measures on $G(F)$ and $H(F)$. 
Let $ \Delta: H(F)_{(G,H)\tu{-}\reg} \times G(F)_{\semis} \to \C$ be a transfer factor, which is necessarily a nonzero scalar multiple of the factor \eqref{eq:transfer-factor-(G,H)-reg}.
Write $G(F)_{\tu{ss}}/\sim$ for a set of representatives for semisimple conjugacy classes in $G(F)$.

\begin{defn}\label{def:Delta-transfer}
Let $f\in \cH(G,\chi^{-1})$ and $f^H\in \cH(H,\chi_H^{-1})$. We say that $f^H$ is a $\Delta$-\emph{transfer} of $f$ (with respect to $\fkw$, $z^{\Rig}$, and $\fke$) if the following holds:
$$ SO_{\gamma_H}(f^H) = \sum_{\gamma\in G(F)_{\tu{ss}}/\sim } \Delta(\gamma_{H}, \gamma) O_\gamma(f),\qquad \gamma_H\in H(F)_{(G,H)\tu{-}\reg}.$$
 where the Haar measures implicit in the (stable) orbital integrals are chosen as follows: if the transfer factor is nonzero, then  $H^{\circ}_{\gamma_H}$ and $G^{\circ}_{\gamma}$ are inner forms, and the Haar measures on $H^{\circ}_{\gamma_H}(F)$ and $G^{\circ}_{\gamma}(F)$ are chosen compatibly in the sense of \cite[p.631]{KottwitzTamagawa}. (Within the stable orbital integral, Haar measures are similarly normalized as in \cite[p.638]{KottwitzTamagawa}.)
\end{defn}

\begin{rem}
The definition is unchanged if the equality is restricted to $\gamma_H\in H(F)_{G\tu{-sr}}$, as shown by \cite[Lem.~2.4]{LanglandsShelstadDescent}. The existence of transfer is clearly independent of the normalization (i.e., rescaling) of $\Delta$.
Since the right hand side is $\chi_H^{-1}$-equivariant in view of Lemma \ref{lem:equivariance} and the $\chi^{-1}$-equivariance of $f$, it is natural to require $f^H$ to be $\chi_H^{-1}$-equivariant.
\end{rem}

\begin{prop}\label{prop:LS-transfer}
Every $f\in \cH(G,\chi^{-1})$ admits a $\Delta$-transfer in the above sense.
\end{prop}

\begin{proof}
The general case is obtained from the case when $\fkX=\{1\}$ by averaging (see the proof of \cite[Prop.~7.4.11]{KSZ} for details), so we may assume $\fkX=\{1\}$.
Then the archimedean case is proven by \cite{ShelstadRealEndoscopy}, in light of \cite[Thm.~2.6.A]{LanglandsShelstadDescent}. In the nonarchimedean case, this is reduced to the fundamental lemma (Proposition \ref{prop:FL} below) by \cite{WaldspurgerFLimpliesTC}. The proof of the fundamental lemma was completed by Ng\^o \cite{NgoFL}, based on earlier works \cite{CluckersLoeser,Hales,WaldspurgerChangement}.
\end{proof}

\subsubsection{The fundamental lemma}\label{sss:fundamental-lemma}

We retain the notation from \S\ref{sss:LS-transfer}. Suppose that $F$ is nonarchimedean and that $G$ and $H$ are unramified over $F$. We fix pinnings $(B,T,\{X_\alpha\})$ for $G$ and $(B_H,T_H,\{Y_{\alpha}\})$ for $H$ defined over $F$. This determines hyperspecial subgroups $K\subset G(F)$ and $K_H\subset H(F)$; see \cite[\S4.1]{Waldspurger-sitordue}. We normalize Haar measures on $G(F)$ and $H(F)$ such that $K$ and $K_H$ have volume 1.
Define \gls{cHur(G,chi)} to be the space of bi-$K$-invariant functions which have compact support modulo $\fkX$ and transform under $\fkX$ by $\chi^{-1}$. Likewise $\cH^{\ur}(H,\chi_H^{-1})$ is defined. We have a morphism of unramified Hecke algebras induced by $\eta:{}^L H\ra {}^L G$ with respect to $(\fkX,\chi)$ (cf.~\cite[\S7.4.10]{KSZ})
$$\eta^*: \cH^{\ur}(G,\chi^{-1})\ra \cH^{\ur}(H,\chi_H^{-1}).$$
Let $\gls{Delta0}$ denote the canonical transfer factor for quasi-split groups \cite[\S3.7]{LanglandsShelstad} (extended to the $(G,H)$-regular case by \cite{LanglandsShelstadDescent}). 

\begin{prop}\label{prop:FL}
For each $f\in \cH^{\ur}(G,\chi^{-1})$, the function $\eta^*(f)$ is a $\Delta_0$-transfer of $f$.
\end{prop}

\begin{proof}
This follows from the case $\fkX=\{1\}$ by averaging with respect to central character data. The proof when $\fkX=\{1\}$ is already explained in the proof of Proposition \ref{prop:LS-transfer}.
\end{proof}

The $\Delta_0$ transfer factor is somewhat incompatible with the transfer factors we use because it is formulated in terms of the arithmetic normalization of the Langlands pairing for tori, whereas we use the geometric normalization (see \cite[\S4]{KottwitzShelstadSigns}). Let $\Delta_D$ denote the version of $\Delta_0$ defined with respect to the geometric normalization of the Langlands pairing and let $\Delta'_0$ denote the version of $\Delta_0$ using $s^{-1}$ as opposed to $s$ (see \cite[\S5.1]{KottwitzShelstadSigns}).  Let $V$ be the virtual representation of $\Gamma$ given by $X^*(T)_{\C} - X^*(T_H)_{\C}$.

\begin{cor}{\label{cor:FLD}}
For each $f\in \cH^{\ur}(G,\chi^{-1})$, the function $\eta^*(f)$ is a $\Delta_D$-transfer of $f$.
\end{cor}
\begin{proof}
As in the last proposition, we only need to check this for the case $\fkX=\{1\}$. 

By Proposition \ref{prop:FL}, we have that $\eta^*(f)$ is a $\Delta_0$-transfer of $f$. Temporarily fix an additive character $\theta_F: F \to \C^{\times}$ and define $\Delta_0[\fkw] :=\epsilon(V, \theta_F)\Delta_0$ as in \cite[\S5.3]{KottwitzShelstad} and analogously, $\Delta'_0[\fkw] := \epsilon(V, \theta_F)\Delta'_0$. Then $\eta^*(f)\epsilon(V, \theta_F)$ is a $\Delta_0[\fkw]$-transfer of $f$ and equivalently $\eta^*(f)\epsilon(V, \theta_F)$  is a $\Delta'_0[\fkw]$-transfer of $f$ relative to the endoscopic datum $(H, \cH, s^{-1}, \eta)$. (We are ``flipping the sign of $s$'' twice here: once by changing the endoscopic datum and once by changing $\Delta_0[\fkw]$ to $\Delta'_0[\fkw]$.)

Now, by \cite[Lem.~3.5]{BMNKottwitzConj}, we have that $\eta^*(f)\epsilon(V, \theta_F) \circ i_H$ is a $\Delta_D[\fkw^{-1}]$-transfer of $f \circ i$, where $i$ is inversion on $G(F)$ and $i_H$ is inversion on $H(F)$ and $\fkw^{-1}$ corresponds to replacing $\theta_F$ with $\theta^{-1}_F$. One checks that $i$ commutes with $\eta^*$ and hence that $\eta^*(f)\epsilon(V, \psi_F)$ is $\Delta_D[\fkw^{-1}]$-transfer of $f$. Thus $\eta^*(f)\epsilon(V, \theta_F)/\epsilon(V, \theta^{-1}_F)$ is a $\Delta_D$-transfer of $f$. Note that $\epsilon(V, \theta^{-1}_F) = \epsilon(V, \theta_F) \det(V)(-1)$, where $\det(V)$ is a representation of $F^{\times}$ via local class field theory. Since $G$ and $H$ are unramified, $V$ is an unramified virtual representation. Hence $\det(V)(-1)=1$, which implies the proposition. 
\end{proof}
\subsubsection{Lefschetz functions on real groups}\label{sss:Lefschetz}

This paragraph will be useful for  stabilizing the archimedean terms in the point-counting formula.

Let $G$ be a connected reductive group over $\R$ which contains an elliptic maximal torus. Fix Haar measures on $G(\R)$ and $Z_G(\R)$, thus also on $G(\R)/Z_G(\R)$. 

Let $\varphi: W_{\R}\ra {}^L G$ be a discrete $L$-parameter. Write $\Pi_\infty(\varphi)$ for the discrete $L$-packet associated with $\varphi$ by \cite{LanglandsRealClassification}. Let $\omega_\varphi:Z_G(\R)\ra \C^\times$ be the common central character for members of $\Pi_\infty(\varphi)$.
There is a group-theoretic recipe for $\omega_\varphi$ as in \S2 of \emph{loc.~cit.}~(where it is denoted by $\chi_\varphi$). We will work with the central character datum $(Z_G(\R),\omega_\varphi)$; this will go well with the main global central character datum of this paper described in Example \ref{ex:fkX-K} below.

 We introduce an \emph{averaged Lefschetz function}
\begin{equation}\label{eq:Lefschetz}
    \gls{favgphi}:=\frac{1}{|\Pi_\infty(\varphi)|} \sum_{\pi\in \Pi_\infty(\varphi)} f_\pi \in \cH(G(\R),\omega_\varphi^{-1}),
\end{equation}
where $f_\pi$ denotes a pseudo-coefficient for $\pi$ \`a la Clozel--Delorme \cite{ClozelDelorme:PaleyWiener2}. Even though $f^{\tu{avg}}_\varphi$ is not uniquely defined, invariant distributions on $G(\R)$ have well-defined values at $f^{\tu{avg}}_\varphi$ thanks to the following characterizing property: the trace of $f^{\tu{avg}}_\varphi$ against an irreducible tempered representation $\pi$ of $G(\R)$ with central character $\omega_\phi$ is $0$ unless $\pi\in \Pi_\infty(\varphi)$, in which case the trace equals $1$.

An important example arises from an irreducible algebraic representation $\xi$ of $G_{\C}$. By restriction, $\xi$ induces a continuous central character $\omega_\xi:Z_G(\R)\ra \C^\times$. Moreover, $\xi$ determines a discrete $L$-parameter $\varphi_\xi: W_{\R}\ra {}^L G$ whose $L$-packet $\Pi_\infty(\varphi_\xi)$ consists of irreducible discrete series representations whose central and infinitesimal characters are the same as the contragredient of $\xi$. Note that $\omega_{\varphi_\xi}=\omega_\xi^{-1}$.
In this case, we set
$$f^{\tu{avg}}_\xi:=f^{\tu{avg}}_\varphi, \qquad 
\Pi_\infty(\xi):=\Pi_\infty(\varphi_\xi).$$

\subsection{Acceptable elements and acceptable functions}\label{ss:acceptable}

Let $F$ be a non-archimedean local field with a uniformizer $\varpi$, and $\nu:\D\ra G$ a fractional cocharacter over $F$.  Write $M_\nu$ for the centralizer of $\nu$ in $G$, which is a Levi subgroup of $G$. We have a unique $F$-rational parabolic subgroup $P_\nu$ (resp.~$P_{\nu}^\op$) of $G$ with \gls{Mnu} as a Levi factor such that every root $\alpha$ of $A_{M_\nu}$ with root group in the unipotent radical $N_{\nu}$ of $P_{\nu}$ (resp.~$P_{\nu}^\op$) satisfies $\langle \alpha,\nu\rangle<0$ (resp.~$\langle \alpha,\nu\rangle>0$). 
Let $(\fkX,\chi)$ be a central character datum of $G$, which can also be viewed as a central character datum of $M_\nu$.

\begin{defn}\label{def:acceptable}
An element $\gamma\in M_\nu(\ol F)$ is said to $\nu$-\emph{acceptable} (or \emph{acceptable} when $\nu$ is clear from the context) if the adjoint action of $\gamma$ on $\Lie N_{\nu}(\ol F)$ is dilating, namely if $|\lambda|>1$ for every eigenvalue $\lambda$ of the action. By $\cH_{\tu{acc}}(M_\nu)\subset \cH(M_{\nu})$ and $\gls{Hacc(M)} \subset \cH(M_{\nu}, \chi^{-1})$ we denote the subspaces of functions supported on acceptable elements.
\end{defn}
Let $J$ be an inner form of $M_\nu$ over $F$, equipped with an $M_\nu(\ol F)$-conjugacy class of isomorphisms $i:J_{\ol F}\simeq M_{\nu,\ol F}$.
Via the canonical isomorphism $Z(J)\cong Z(M_\nu)$ over $F$, we can view $(\fkX,\chi)$ as a central character datum of $J$.
Since the acceptability of $\gamma\in M_\nu(\ol F)$ depends only on its $M_\nu(\ol F)$-conjugacy class, the following definition depends on $i$ only though its $M_{\nu}(\ol F)$-conjugacy class.

\begin{defn}\label{def:acceptable-J}
An element $\delta\in J(\ol F)$ is said to \emph{acceptable} if $i(\delta)\in M_{\nu}(\ol F)$ is acceptable. 
The spaces $\cH_{\tu{acc}}(J)$ and $\cH_{\tu{acc}}(J,\chi^{-1})$ are defined as in Definition \ref{def:acceptable}.
\end{defn}

\begin{ex}\label{ex:M*Jb}
In the setting of \S\ref{BGreview}, we can take $\nu=\ol \nu_b$ in $G^*$. Then $M_\nu=M_b$, and $J=J_b$ is an inner twist of $M_b$ via \eqref{eq:M*Jb}.
\end{ex}

\subsubsection{}
Choose $r\in \Z_{\ge 1}$  such that $r\nu$ is a cocharacter of $G$ (i.e., factors through the projection $\D\ra \G_m$).
For $\phi\in \cH(M_{\nu})$ and $j\in r\Z$, we define the following translates of $\phi$:
$$\gls{phi(j)}\in \cH(M_{\nu})\quad \mbox{by}\quad
\phi^{(j)}(\gamma)= \phi((j\nu)(\varpi)^{-1}\gamma).
$$
For $\phi\in \cH(J)$ and $j\in r\Z$, the same formula defines $\phi^{(j)}\in \cH(J)$.

\begin{lem}
Given $\phi\in \cH(M_{\nu})$, there exists $j_0\in \Z$ such that $\phi^{(j)}\in \cH_{\tu{acc}}(M_{\nu})$ for every $j\in r\Z$ with $j\ge j_0$. The same holds true if $M_{\nu}$ is replaced with $J$. 
\end{lem}

\begin{proof}
This follows from the facts that $r\nu(\varpi)$ is acceptable and that $\phi$ has compact support.
\end{proof}

When $P\subset G$ is an $F$-rational parabolic subgroup with a Levi factor $M$, write $\delta_P:M(F)\ra \C^\times$ for the modulus character. Let $J_P$ denote the normalized Jacquet module functor from smooth representations of $G(F)$ to those of $M(F)$. 

\begin{defn}\label{def:ascent}
Let $\phi\in \cH_{\tu{acc}}(M_{\nu})$. We say that $f\in \cH(G)$ is a $\nu$-\emph{ascent} of $\phi$ if 
\begin{enumerate}
    \item for every $g\in G(F)_{\semis}$, we have $O_g(f)=0$ unless $g$ is conjugate in $G(F)$ to a $\nu$-acceptable element $m\in M_{\nu}(F)$, in which case
    $$O^G_g(f)=\gls{deltaP}^{-1/2} O^{M_\nu}_m(\phi).$$
    \item $\tr (f|\pi)=\tr (\phi|\gls{JP(pi)})$ for admissible representations $\pi$ of $G(F)$.
\end{enumerate}
\end{defn}

\begin{lem}{\label{lem:ascentexistance}}
Every $\phi\in \cH_{\tu{acc}}(M_{\nu})$ admits a $\nu$-ascent $f\in \cH(G)$. Moreover, statement (i) of Definition \ref{def:ascent} holds with stable orbital integrals in place of orbital integrals.
\end{lem}

\begin{proof}
This is \cite[Lem.~3.1.2, Cor.~3.1.3]{KretShin}.
\end{proof}

\begin{rem}
When extracting spectral information from the stabilized trace formula of this paper, e.g., as in \cite{ShinGalois,KretShin}, it is important to understand the interaction between $\nu$-ascent and constant terms/endoscopy. This is studied in \cite[\S3]{KretShin}.
\end{rem}

\subsubsection{Local endoscopic refinements}{\label{effectiveendoscopy}}

In this subsection, we use the notion of acceptability to define refinements of $\ES^{\emb}(M,G)$. We continue to assume $F$ is a non-archimedean local field and we further assume that $G=G^*$ is quasi-split and that $G^*_{\der}$ is simply connected.

Fix $[\bb] \in B(G)$. Since $G$ is quasi-split, we can choose a decent representative $\bb \in G(\breve{F})$ such that $\ov{\nu}_{\bb}$ is defined over $F$. (See the proof of \cite[Prop.~6.2]{KottwitzIsocrystal1}.)

\begin{defn}\label{def:EKisoeff}
We let \gls{EKisoeff(J,G)} denote the set of equivalence classes of pairs $(\delta, \lambda)$ such that $\delta \in J_{\bb}(F)$ is $\nu_{\bb}$-acceptable and transfers to a $(G, M_{\bb})$-regular semisimple element of $G(F)$ and $\lambda \in Z(\widehat{I_{\delta}})^{\Gamma_F}$. Via endoscopic transfer of conjugacy classes from $J_{\bb}$ to $M_{\bb}$, we can identify $\EK^{\Iso}_{\eff}(J_{\bb}, G)$ with a subset of $\EK^{\Iso}(M_{\bb}, G)$ via the map $(\delta, \lambda) \mapsto (\gamma, \lambda')$, where $\lambda'$ is the image of $\lambda$ under the canonical isomorphism $Z(\widehat{I_{\delta}}) \cong Z(\widehat{I^{M_{\bb}}_{\gamma}})$. 
\end{defn}

Via the bijections in Diagram \eqref{EQSSdiagram}, we get a  subset $\gls{ESisoeff(J,G)} \subset \ES^{\Iso}(M_{\bb}, G)$ in bijection with $\EK^{\Iso}_{\eff}(J_{\bb},G)$. We define \gls{ESisoeff(G)} as the image of $\ES^{\Iso}_{\eff}(J_{\bb}, G)$ under the map
\begin{equation*}
    \ES^{\Iso}(M_{\bb}, G) \to \ES^{\Iso}(G).
\end{equation*}

We have a natural projection 
\begin{equation*}
    \ES^{\Iso}(M_{\bb}, G) \to \mE^{\Iso}(M_{\bb}),
\end{equation*}
and we denote the image of $\ES^{\Iso}_{\eff}(J_{\bb}, G)$ by \gls{mEisoeff(J,G)}. 

Analogously, we can define \gls{ESembeff(J,G)} and \gls{mEembeff(J,G)} and we have
\begin{equation*}
   \ES^{\emb}_{\eff}(J_{\bb},G) \cong \ES^{\Iso}_{\eff}(J_{\bb},G), 
\end{equation*}
and
\begin{equation*}
    \mE^{\emb}_{\eff}(J_{\bb},G) \cong \mE^{\Iso}_{\eff}(J_{\bb},G).
\end{equation*}

\begin{defn}
Fix an equivalence class $\fke = (H, \cH, s, \eta) \in \mE^{\Iso}(G)$. Relative to this fixed class, we define the sets
\begin{eqnarray*}
    \gls{mEisoeff(J,G;H)} &:=& Y^{-1}(\fke) \cap \mE^{\Iso}_{\eff}(J_{\bb}, G),\\
    \gls{mEembeff(J,G;H)} &:=& \gls{mEembeff(J,G)}\cap \mE^{\emb}(M_{\bb}, G; H).
\end{eqnarray*}
We also define $\ES^{\emb}_{\eff}(J_{\bb},G;H)$ to be the pre-image of $\mE^{\emb}_{\eff}(J_{\bb}, G; H)$ under the projection $\ES^{\emb}(M_{\bb}, G; H) \to \mE^{\emb}(M_{\bb}, G; H)$. We define $\mE^i_{\eff}(J_{\bb}, G;H) \subset \mE^i_{\eff}(M_{\bb},G;H)$ to consist of those equivalence classes whose $\mE^{\emb}(M_{\bb}, G;H)$-isomorphism class lies in $\mE^{\emb}_{\eff}(J_{\bb}, G; H)$.
\end{defn}

We now reinterpret the set $\mE^{\Iso}_{\eff}(J_{\bb}, G) $ in terms of transfer of maximal tori instead of transfer of semisimple conjugacy classes, as the former is in practice easier to work with. 
\begin{lem}
An element $(H_{\bb}, \cH_{\bb}, s_{\bb}, \eta_{\bb}) \in \mE^{\Iso}(M_{\bb})$ is contained in $\mE^{\Iso}_{\eff}(J_{\bb}, G)$ if and only if there exist maximal tori $T_{H_{\bb}}, T_{M_{\bb}}, T_{J_{\bb}}$ defined over $F$ of $H_{\bb}, M_{\bb}, J_{\bb}$ respectively such that each torus transfers to the others.
\end{lem}
\begin{proof}
This is \cite[Lem.~2.41]{BMaveraging}.
\end{proof}

We now fix a representative $(H, \cH, s, \eta) \in \sE^{\Iso}(G)$ of the class $\fke$ and a set of representatives $X^{\fke}_{\bb}(H)$ of $\mE^{i}_{\eff}(J_{\bb}, G; H)$.
We then have the following lemma.
\begin{lem}{\label{lem: uniquenesslem}}
Suppose that $[(H_{\bb}, \cH_{\bb}, H, \cH, s, \eta_1, \gamma_{H_{\bb}})] \in \ES^{\emb}_{\eff}(M_{\bb}, G; H)$ and $(H, \cH, s, \eta, \gamma_H)$ project to the same class in $\ES^{\Iso}_{\eff}(G)$. Then there exist
\begin{itemize}
    \item a unique $(H'_{\bb}, \cH'_{\bb}, H, \cH, s, \eta_2) \in X^{\fke}_{\bb}(H)$ and
    \item a  $(G, H_{\bb})$-regular, $\nu_{\bb}$-acceptable, semisimple $\gamma_{H'_{\bb}} \in \Sigma(H'_{\bb}(F))$
\end{itemize}
 such that
 \begin{itemize}
     \item  $(H'_{\bb}, \cH'_{\bb}, H, \cH, s, \eta_2, \gamma_{H'_{\bb}})$ projects to $[(H_{\bb}, \cH_{\bb}, H, \cH, s, \eta_1, \gamma_{H_{\bb}})]$ and
     \item $\gamma_{H'_{\bb}}$ and $\gamma_H$ are stably conjugate in $H(F)$ under the natural inclusion $H'_{\bb}(F) \subset H(F)$.
 \end{itemize}
\end{lem}
\begin{proof}
This is \cite[Lem.~2.42]{BMaveraging} (cf. \cite[Lem.~6.2]{Shin10}).
\end{proof}

\subsection{\texorpdfstring{$z$}{z}-extensions}\label{ss:z-extensions}

Let $F$ be local or global in this subsection. The content here allows us to reduce questions about general endoscopic data to the case where $\cH={}^L H$.

\begin{defn}\label{def:z-ext}
A \emph{$z$-extension} of $G$ over $F$ is a central extension of reductive groups over $F$
\begin{equation}\label{eq:z-ext}
    1\ra Z_1 \ra G_1 \ra G \ra 1
\end{equation}
with the property that $G_1$ has simply connected derived subgroup, and that $Z_1\cong \prod_i \Res_{F_i/\Q}\G_m$ for finite extensions $F_i$ over $F$.
\end{defn}

\subsubsection{$z$-extensions and endoscopic data}\label{sss:z-ext}
A fixed pinning for $G$ induces a pinning for $G_1$ via pullback. Dual to $G_1\ra G$ is a morphism of dual groups $\hat{G} \ra \hat{G}_1$. We can arrange that the $\Gamma_F$-pinnings for $\hat{G}$ and $\hat{G}_1$ are compatible with the latter map. This map uniquely extends to an $L$-morphism $\zeta_{G_1}: {}^L G \ra {}^L G_1$ such that $\zeta_{G_1}|_{W_F}$ is the identity map.

A $z$-extension of $G$ as in \eqref{eq:z-ext} and an endoscopic datum $\fke=(H,\cH,s,\eta)$ determine a central extension over $F$
\begin{equation}\label{eq:z-ext-H}
1\ra Z_1 \ra H_1 \ra H \ra 1
\end{equation}
and an endoscopic datum $\fke_1=(H_1,{}^L H_1,s_1,\eta_1)$ for $G_1$ with $s_1=\zeta_{G_1}(s)$ with the following properties (see \cite[Lem.~7.2.6, 7.2.9]{KSZ}):
\begin{enumerate}
\item The canonical embeddings $Z_G\ra Z_H$ and $Z_{G_1}\ra Z_{H_1}$ fit in a row-exact commutative diagram of $F$-groups, where the rows come from \eqref{eq:z-ext} and \eqref{eq:z-ext-H}:
\begin{equation}\label{eq:Z1-ZG-ZH}
    \xymatrix{ 1 \ar[r] & Z_{1} \ar[r] \ar@{=}[d] & Z_{G_1} \ar[r] \ar[d] & Z_G \ar[d] \ar[r] & 1
		\\ 1 \ar[r] & Z_{1} \ar[r] & Z_{H_1} \ar[r] & Z_H \ar[r] & 1. }
\end{equation} 
    \item There exists an $L$-morphism $\zeta_{H_1}$ which extends the embedding $\hat H \ra \hat H_1$ dual to $H_1 \ra H$ and makes the following diagram commute:
$$ \xymatrix{ \cH \ar[r]^-{\zeta_{H_1}} \ar[d]^-{\eta} & {}^L H_1 \ar[d]^-{\eta_1} \\ ^L G \ar[r]^-{\zeta_{G_1}} & ^L G_1.}$$
\item If $\fke$ is a refined datum then so is $\fke_1$.
\end{enumerate}
For each $H$ as above, we are fixing a pinning $(T_H, B_H, \{X_{\alpha_H}\})$ defined over $F$. This induces a pinning $(T_{H_1}, B_{H_1}, \{X_{\alpha_H, 1}\})$ for $H_1$ over $F$ by pullback.

\subsubsection{$z$-extensions and embedded endoscopic data}\label{sss:z-ext-embedded}

Let $M\subset G$ be a standard Levi subgroup.
Given a $z$-extension \eqref{eq:z-ext} and an embedded endoscopic datum $(H_M,\cH_M,H,\cH,s,\eta)$ for $G$ (Definition \ref{def:emb-end-datum}), we use $\fke_1$ to construct an embedded endoscopic datum for $G_1$ as follows. Take $H_{M,1}$ to be the preimage of $H_M$ in $H_1$ in \eqref{eq:z-ext-H}. Define $\cH_{M,1}$ to be the subgroup of $^L H_1$ generated by $\cH_M$ and $Z(\hat H_1)$. Then $(H_{M,1},\cH_{M,1},H_1,{}^L H_1,s_1,\eta_1)\in \sE^{\emb}(M_1,G_1)$.

\section{Igusa varieties}\label{s:Igusa}

We are going to formulate a conjectural trace formula for Igusa varieties in the case of general Shimura data with parahoric level at $p$, assuming some hypothetical construction of integral models and Igusa varieties. We are in favor of this somewhat axiomatic approach because, as long as the conjectural trace formula is taken for granted, we do not need further assumptions (e.g., Shimura data can be arbitrary).
Thus the results of this paper will become unconditional in more cases as we make further progress on the hypotheses. The state of the art is reviewed in \S\ref{sss:known} below.

\subsection{The general setup}\label{ss:general-setup}

\begin{defn}\label{def:integral-Shimura-datum}
An \emph{integral Shimura datum} is a quadruple $(G,X,p,\cG)$, where
\begin{itemize}
    \item $(G,X)$ is a Shimura datum satisfying Deligne's axioms \cite[(2.1.1.1)--(2.1.1.3)]{DeligneCorvallis},
    \item $p$ is a prime number,
    \item $\cG$ is a parahoric model of $G_{\Q_p}$.
\end{itemize}
In particular $G$ is a connected reductive group over $\Q$ such that $G_{\R}$ contains an elliptic maximal torus. We write $K_p:=\cG(\Z_p)$, which is a parahoric subgroup of $G(\Q_p)$. The datum $(G,X)$ determines a conjugacy class $[\mu_X]$ of Hodge cocharacters $\G_m\ra G_{\C}$, cf.~\cite[1.1.1, 1.1.11]{DeligneCorvallis}. Denote by $E:=E(G,X)\subset \C$ the field of definition of $[\mu_X]$, namely the reflex field. 
By abuse of notation, we still write $[\mu_X]$ for the element of $\pi_1(G)$ determined by $[\mu_X]$.
\end{defn}

\subsubsection{} 

Henceforth we fix a datum as above. Fix a prime $\ell$ not equal to $p$. We also fix isomorphisms $\iota_p:\Qpbar\simeq \C$, $\iota:\Qlbar\simeq \C$, and an embedding $\Qbar\hra \C$. We have induced embeddings $\Qbar\hra \Qpbar$, $\Qbar \hra \Qlbar$, and $E\hra \Qpbar$. This determines a prime $\fkp$ of $E$ above $p$. Write $k:=k(\fkp)$ for the residue field at $\fkp$. Its algebraic closure $\ol k$ can be identified with the residue field of $\Qpbar$.
Fix a cocharacter $\mu_p:\G_m\ra G_{\Qpbar}$ in the conjugacy class $\iota_p^{-1}[\mu_X]$.

We have a projective system of varieties $\gls{ShKp}=\{\Sh_{K^pK_p}\}$ over $E$ in which the transition maps are finite and \'etale (so $\gls{ShKp}$ can be viewed as a scheme over $E$), where $K^p$ runs over neat open compact subgroups of $G(\A^{\infty,p})$. The Shimura variety $\Sh_{K_p}$ is equipped with a natural Hecke action of $G(\A^{\infty,p})$. The following existence of integral models is a prerequisite for the definition of Igusa varieties. (Compare with the more precise statement \cite[Conj.~3.4]{PappasICM}.)

\begin{hypothesis}\label{hypo:integral-model-Shimura}
The $E$-scheme $\Sh_{K_p}$ extends to a natural $\cO_{E_{\fkp}}$-scheme $\gls{mSKp}=\{\mS_{K^pK_p}\}$ whose transition maps are finite and \'etale, equipped with a right $G(\A^{\infty,p})$-action 
(in the sense of \cite[2.7.1]{DeligneCorvallis})
such that the $G(\A^{\infty,p})$-action on $\gls{mSKp}$ extends that on $\Sh_{K_p}$. 
\end{hypothesis}

\begin{rem}\label{rem:natural}
Often the meaning of ``natural'' can be made precise on the one hand by a unique characterization and on the other hand by a natural construction (e.g., taking closure in a Siegel modular variety and normalization when possible). See \S\ref{sss:known} below for known results and references.
However a precise notion is unnecessary for our purpose. All we want is \emph{some} construction of integral models $\mS_{K_p}$ over which Igusa varieties may be defined such that Hypothesis \ref{hypo:Igusa-existence} and Conjecture \ref{conj:Igusa-general} below are satisfied.
\end{rem}

\subsubsection{} 

We assume Hypothesis \ref{hypo:integral-model-Shimura} in the sense above.
Set $\mS_{K_p,\ol k}:=\mS_{K_p}\times_{\cO_{E_{\fkp}}} \ol k$.
Let $\xi$ be an irreducible algebraic representation of $G$ over $\C$. 
 Write \gls{Zac} for the anti-cuspidal part of $Z_G^{\circ}$, which is a $\Q$-subtorus of $Z^{\circ}_G$. (See \cite[Def.~1.5.4]{KSZ}.) Henceforth we assume that
\begin{equation}\label{eq:trivial-on-Zac}
    \xi ~\mbox{factors through the projection}~ G\ra G/Z_{\tu{ac}}.
\end{equation}
Then a lisse $\ell$-adic sheaf \gls{cLxi} can be constructed on $\mS_{K_p}=\{\mS_{K^pK_p}\}$ following \cite[\S1.5]{KSZ}.

Let $\bb\in G(\breve\Q_p)$ and write $[\bb]\in B(G_{\Q_p})$ for the image of $\bb$. 
Caraiani--Scholze \cite{CS17,CS19} and related works (e.g.,~\cite{HamacherKim}) suggest the following for Igusa varieties.

\begin{hypothesis}\label{hypo:Igusa-existence}
Assume that $[\bb]\in B(G_{\Q_p},\mu_p^{-1})$. 
\begin{enumerate}
    \item There exist
\begin{itemize}
    \item a scheme \gls{IGb} over $\ol k$ equipped with a right $G(\A^{\infty,p})\times J_{\bb}(\Q_p)$-action and
    \item a $G(\A^{\infty,p})$-equivariant morphism $\gls{sigmab}:\IG_{\bb} \ra \mS_{K_p,\ol{k}}$. 
\end{itemize}
For every $\bb'\in G(\breve\Q_p)$ with $[\bb]=[\bb']$, choose an element $g\in G(\breve\Q_p)$ such that $\bb'=g^{-1}\bb \sigma(g)$, which induces a $\Q_p$-isomorphism $i_g:J_{\bb}\simeq J_{\bb'}$. Then there exists a $G(\A^{\infty,p})\times J_{\bb}(\Q_p)$-equivariant isomorphism $\IG_{\bb}\simeq \IG_{\bb'}$ carrying $\varsigma_{\bb}^* \cL_\xi$ to $\varsigma_{\bb'}^* \cL_\xi$, where $J_{\bb}(\Q_p)$ acts on $\IG_{\bb'}$ by pulling back the $J_{\bb'}(\Q_p)$-action along $i_g$.
\item We have the following data after possibly changing $\bb$ to another element (which is in particular decent) in its $\sigma$-conjugacy class inside $G(\breve \Q_p)$. There exists a projective system $$\Ig_{\bb}=(\Ig_{\bb,m,K^p})_{m,K^p}, \qquad m\in \Z_{\ge0}, \quad K^p\subset G(\A^{\infty,p}) ~\mbox{neat compact open},$$
consisting of smooth finite-type varieties over $\ol k$ (definable over a common finite extension of~$k$) with finite \'etale transition maps, equipped with a right $G(\A^{\infty,p})\times J_{\bb}(\Q_p)$-action on $\Ig_{\bb}$
(in the sense of \cite[2.7.1]{DeligneCorvallis}). Moreover there are $\ol k$-morphisms $\varsigma_{\bb,m,K^p}:\Ig_{\bb,m,K^p}\ra \mS_{K_pK^p}$ which are compatible with the transition maps as $m$ and $K^p$ vary, as well as a $G(\A^{\infty,p})\times J_{\bb}(\Q_p)$-equivariant isomorphism 
\begin{equation}\label{eq:IG=Ig}
    \IG_{\bb}\simeq (\varprojlim \Ig_{\bb,m,K^p})^{\tu{perf}},
\end{equation}
such that $\varsigma_{\bb}$ coincides with the limit of $\varsigma_{\bb,m,K^p}$ over $m,K^p$ after perfection. 
\end{enumerate}
\end{hypothesis}
\subsubsection{} \label{sss:H(IG)}

We want to understand the cohomology of $H^i(\IG_{\bb},\varsigma_{\bb}^* \cL_\xi)$ as a $G(\A^{\infty,p})\times J_{\bb}(\Q_p)$-module, which is an admissible module thanks to the preceding hypothesis. By the hypothesis, we can choose $\bb$ as in (ii) and consider the cohomology of 
$\Ig_{\bb,m,K^p}$ instead. By Poincar\'e duality, we may switch to cohomology with compact support, as the latter is more directly related to the fixed-point formula for varieties. Therefore we define
$$[H_c]_{\bb,\xi}:=\sum_{i\ge 0} (-1)^i \left( \varinjlim \iota H^i_c(\Ig_{\bb,m,K^p},\varsigma_{\bb,m,K^p}^*\cL_\xi) \right) \in \Groth(G(\A^{\infty,p})\times J_{\bb}(\Q_p)),$$
where we applied $\iota$ to switch from $\ell$-adic coefficients to $\C$-coefficients.  A priori only $G(\A^{\infty,p})\times S_{\bb}$ acts on the right hand side, where $S_{\bb}\subset J_{\bb}(\Q_p)$ is a certain submonoid which generates $J_{\bb}(\Q_p)$ as a group (see \cite[2.4.5]{MackCraneShin} for the definition of $S_{\bb}$, which is not needed in this paper), but it uniquely extends to a $G(\A^{\infty,p})\times J_{\bb}(\Q_p)$-action thanks to \eqref{eq:IG=Ig} and the fact that $\IG_{\bb}$ already admits a $G(\A^{\infty,p})\times J_{\bb}(\Q_p)$-action.

We would like to describe a conjectural formula for the trace
$$\tr (\phi^{\infty,p}\phi_p |[H_c]_{\bb,\xi})\in \C, \qquad \phi^{\infty,p}\phi_p\in C^\infty_c(G(\A^{\infty,p})\times J_{\bb}(\Q_p)).$$

\subsection{Kottwitz parameters and Kottwitz invariants}

To state a conjectural trace formula, we introduce some more notions and notation. We make the following technical hypothesis on $G$ from here throughout the paper, as we anticipate extra subtlety (e.g., regarding the formalism of Kottwitz parameters) without the condition.

\begin{hypothesis}\label{hypo:quasi-split-at-p}
The group $G_{\Q_p}$ is quasi-split.
\end{hypothesis}

\subsubsection{}

According to \S\ref{ss:isocrystals}, we can choose $\bb$ to be decent without affecting $[\bb]$, and such that $\nu_{\bb}$ is defined over $\Q_p$ thanks to Hypothesis \ref{hypo:quasi-split-at-p}. Then for a sufficiently divisible $r\in \Z_{>0}$, we have that $\bb \in G(\Q_{p^r})$ and that $r\nu_{\bb}$ is a cocharacter over $\Q_p$.
Fix such a $\bb\in G(\breve\Q_p)$ from now on.

\subsubsection{Galois cohomology}\label{sss:Galois-cohomology}
Let $I\ra G$ be a morphism of reductive groups over $\Q$. Let $S$ be a finite set of places of $\Q$. We define
\begin{align*}
    \gls{DIGA}&:=\ker(H^1(\A^S,I)\ra H^1(\A^S,G)),\\
    \gls{kerHab}&:=\ker(H^1_{\ab}(\A^S,I)\ra H^1_{\ab}(\A^S,G)).
\end{align*}
Note that $\fkD(I,G;\A^S)$ is a pointed set and $\fkE(I,G;\A^S)$ is an abelian group. 
When $S_0\subset S$, there is an obvious map $\fkD(I,G;\A^S)\ra \fkD(I,G;\A^{S_0})$ and likewise for $\fkE$.
Replacing $\A^S$ with $\Q_v$ for a place $v$ of $\Q$, we define $\fkD(I,G;\Q_v)$ and $\fkE(I,G;\Q_v)$.
The canonical and functorial isomorphism $H^1_{\ab}(\Q_v,G)\cong \pi_1(G)_{\Gamma_v,\tu{tor}}$ induces a canonical isomorphism
$$\fkE(I,G;\Q_v) \cong \ker(\pi_1(I)_{\Gamma_v,\tu{tor}}\ra \pi_1(G)_{\Gamma_v,\tu{tor}}).$$
We have natural restricted product decompositions
$$\fkD(I,G;\A^S)=\textstyle\prod'_{v\notin S} \fkD(I,G;\Q_v), \qquad \fkE(I,G;\A^S)=\textstyle\prod'_{v\notin S} \fkE(I,G;\Q_v),$$
coming from similar decompositions for $H^1$ and $H^1_{\ab}$, see ~\cite[(1.1.6.1),(1.1.6.2)]{KSZ}, where $'$ means that we only consider elements in the product whose components are trivial for all but finitely many $v$.
We have the abelianization map
$$\ab^1:\fkD(I,G;\A^S) \ra \fkE(I,G;\A^S),$$
which is compatible with the analogous local abelianization map over $\Q_v$ via the product decomposition. 
The map $\ab^1$ is a bijection if $S$ contains the infinite place $\infty$.

Now assume that $I$ is a reductive subgroup of $G$ over $\Q$ containing a maximal torus of $G$. The two abelian groups $\fkE(I,G;\A/\Q)$ and \gls{K(I/Q)} are defined in \cite[\S1.8]{LabesseStabilization} and \cite[\S4]{KottwitzEllipticSingular}. There is a natural map
$$\fkE(I,G;\A) \ra \fkE(I,G;\A/\Q).$$
By \cite[Prop.~1.7.3, Cor.~1.7.4]{KSZ}, $\fkE(I,G;\A/\Q)$ and \gls{K(I/Q)} are finite and in Pontryagin duality with each other via a natural pairing
$$\langle ~,~\rangle: \fkE(I,G;\A/\Q)\times \fkK(I/\Q) \ra \C^\times.$$
Now let $I\ra G$ be a morphism of connected reductive groups over $\Q_p$, inducing the map $B(I)\ra B(G)$. We have $\bb\in G(\breve\Q_p)$ as before. Define the subset 
$$\gls{DIGb} \subset B(I)$$
to be the preimage of $[\bb]\in B(G)$.

\begin{defn}\label{def:Kottwitz-parameter}
A \emph{$\bb$-admissible Kottwitz parameter} is a triple \gls{c} $=(\gamma_0,a,[b])$, where
\begin{itemize}
    \item $\gamma_0\in G(\Q)_{\R\tu{-ell}}$; write $I_0:=(G_{\gamma_0})^{\circ}$,
    \item $a\in \fkD(I_0,G;\A^{\infty,p})$,
    \item $[b]\in \fkD_p(I_0,G;\bb)$.
\end{itemize}
Write \gls{KPbb} for the set of $\bb$-admissible Kottwitz parameters.
Given $\fkc\in \KP_{\bb}$, if we choose a representative $b\in I_0(\breve \Q_p)$ of $[b]$ then
$\gamma_0^{-1} b \sigma(\gamma_0)=\gamma_0^{-1} b \gamma_0=b$ in $G(\breve \Q_p)$. Therefore 
$\gamma_0$ lies in the subgroup $J_b(\Q_p)$ of $G(\breve \Q_p)$. We say that $\fkc$ is \emph{acceptable} if $\gamma_0$ is a $\nu_b$-acceptable element of $J_b(\Q_p)$ (Definition \ref{def:acceptable-J}, Example \ref{ex:M*Jb}). This notion is independent of the choice of $b$ (\cite[4.2.10]{MackCraneShin}).
\end{defn}

\begin{rem}
Our $\bb$-admissible Kottwitz parameter is a Kottwitz parameter in the sense of \cite[\S1.6]{KSZ}. The point is that the condition KP0 therein is satisfied since $[\bb]\in B(G_{\Q_p},\mu_p^{-1})$, which implies that $\kappa_G([\bb])=-[\mu_X]$.
\end{rem}

\subsubsection{From Kottwitz parameters to $\bb$-classical Kottwitz parameters}\label{sss:KP-to-CKP}

By a $\bb$-\emph{classical Kottwitz parameter} (which was called a ``Kottwitz triple'' in \cite{Shin09,Shin10}), we mean a triple $(\gamma_0;\gamma,\delta)$, where
\begin{itemize}
    \item $\gamma_0\in G(\Q)_{\R\tu{-ell}}$,
    \item $\gamma\in \Gamma(G(\A^{\infty,p}))$ which is conjugate to $\gamma_0$ in $G(\ol \A^{\infty,p})$,
    \item $\delta\in \Gamma(J_{\bb}(\Q_p))$ which is conjugate to $\gamma_0$ in $G(\breve\Q_p)$
    via $J_{\bb}(\Q_p)\subset G(\breve\Q_p)$.
\end{itemize}
Denote by \gls{CKP} the set of $\bb$-classical Kottwitz parameters. We construct a natural map
\begin{equation}\label{eq:KP-to-CKP}
\KP_{\bb} \ra \mathcal{CKP}_{\bb},\qquad
(\gamma_0,a,[b])\mapsto (\gamma_0,\gamma_a,\delta_{[b]})
\end{equation}
as follows. Let $(\gamma_0,a,[b])\in \KP_{\bb}$.
We already remarked that $\gamma_0\in J_b(\Q_p)$ for a representative $b\in I_0(\breve\Q_p)$ of $[b]$. Since $[\bb]=[b]$ in $B(G)$, we have a well-defined $J_{\bb}(\Q_p)$-conjugacy class of $\Q_p$-isomorphisms  $J_b\simeq J_{\bb}$. Therefore $\gamma_0\in J_b(\Q_p)$ determines a conjugacy class of $\delta_{[b]}\in J_{\bb}(\Q_p)$. Now since $a$ has trivial image in $H^1(\A^{\infty,p},G)$, it is represented by a cocycle $\tau \mapsto g^{-1} \tau(g)$ for some $g\in G(\ol \A^{\infty,p})$. Then $\gamma_a:=g \gamma_0 g^{-1}\in G(\A^{\infty,p})$. By construction $(\gamma_0,\gamma_a,\delta_{[b]})\in \mathcal{CKP}_{\bb}$, which is easily checked to be independent of the choice of $b$ and $g$.

\begin{lem}\label{lem:KP-to-CKP}
If $G_{\tu{der}}=G_{\tu{sc}}$ then the map \eqref{eq:KP-to-CKP} is a bijection.
\end{lem}
\begin{proof}
We construct the inverse map. Let $(\gamma_0,\gamma,\delta)\in \mathcal{CKP}_{\bb}$. Choose $g\in G(\ol \A^{\infty,p})$ such that $g\gamma_0 g^{-1}=\gamma$. Then the cocycle $\tau\mapsto g^{-1} \tau(g)\in G_{\gamma_0}(\ol \A^{\infty,p})$ defines an element $a\in \fkD(I_0,G,\A^{\infty,p})$; the point is that $G_{\gamma_0}=I_0$, namely $G_{\gamma_0}$ is connected, since $G_{\tu{der}}=G_{\tu{sc}}$.
Now write $\delta=c\gamma_0 c^{-1}$ for some $c\in G(\breve \Q_p)$, and put $b:=c^{-1} \bb \sigma(c)$. It follows from $\delta \bb = \bb \sigma(\delta)$ that $b$ centralizes $\gamma_0$ in $G$, so $b\in I_0(\breve \Q_p)$.
It is routine to check that $(\gamma_0,a,[b])$ is well defined and belongs to $\KP_{\bb}$, and that the map $(\gamma_0,\gamma,\delta)\mapsto (\gamma_0,a,[b])$ is converse to \eqref{eq:KP-to-CKP}.
\end{proof}

\begin{lem}\label{lem:b-is-basic}
If $(\gamma_0,a,[b])\in \KP_{\bb}$ is acceptable then $[b]$ is basic in $B(I_0)$. 
\end{lem}

\begin{proof}
It is enough to verify the equivalent assertion that $\nu_b$ is central in $I_0$, or that $I_0$ is contained in the centralizer of $\nu_b$ in~$G$. This is \cite[Lem.~2.2.10]{MackCraneShin}. 
\end{proof}

\subsubsection{Kottwitz invariant}
\label{sss:Kottwitz-invariant}

To each Kottwitz parameter $\fkc$, we assign a Kottwitz invariant \gls{alphac} following \cite[\S1.7]{KSZ}, cf.~\cite[\S4.2.12]{MackCraneShin}). The facts in \S\ref{sss:Galois-cohomology} will be used freely.

Write $(\beta_v(\fkc))_{v\neq p,\infty}$ for the image of $a$ under the composite map
$$\fkD(I_0,G;\A^{\infty,p})\simeq \fkE(I_0,G;\A^{\infty,p})\simeq \bigoplus_{v\neq p,\infty}\ker(\pi_1(I_0)_{\Gamma_v,\tu{tor}}\ra \pi_1(G)_{\Gamma_v,\tu{tor}}).$$
Choose a lift \gls{tildebetav} $\in \ker(\pi_1(I_0)\ra \pi_1(G))$ of $\beta_v(\fkc)$ for each $v\neq p,\infty$; if $\beta_v(\fkc)=0$ then simply put $\tilde \beta_v(\fkc):=0$. At $v=p$, set $\beta_p(\fkc):=\kappa_{I_0}([b])\in \pi_1(I_0)_{\Gamma_p}$. Then $\beta_p(\fkc)$ maps to $-[\mu_X]\in \pi_1(G)_{\Gamma_p}$ under the natural map $\pi_1(I_0)_{\Gamma_p} \ra \pi_1(G)_{\Gamma_p}$. So we can pick a lift $\tilde \beta_p(\fkc)\in \pi_1(I_0)$ whose image in $\pi_1(G)$ is $-[\mu_X]$. At $v=\infty$, since $\gamma_0$ is $\R$-elliptic, we can choose an elliptic maximal torus $T_\infty$ of $G_{\R}$ containing $\gamma_0$. Then $T_\infty\subset I_{0,\R}$, and there exists $h:\Res_{\C/\R}\G_m \ra G_{\R}$ in $X$ factoring through $T_\infty$. The latter gives rise to a cocharacter $\mu_h:\G_m\ra T_{\infty, \C}$. Define $\beta_\infty(\fkc)\in \pi_1(I_0)_{\Gamma_\infty}$ to be the image of $\mu_h$ under 
$$X_*(T)=\pi_1(T)\ra \pi_1(I_0)\ra \pi_1(I_0)_{\Gamma_\infty}.$$
Then $\beta_\infty(\fkc)$ is independent of the choice of $h\in X$, and has the same image in $\pi_1(G)_{\Gamma_\infty}$ as $[\mu_X]\in \pi_1(G)$. 
Thus we can choose a lift $\tilde\beta_\infty(\fkc)\in \pi_1(I_0)$ mapping to $[\mu_X]\in \pi_1(G)$. 

In \S\ref{s:stabilization} below, we will often write
$$\tilde\beta^{\infty,p}(\gamma_0,a),~\tilde\beta_p(\gamma_0,[b]) ,~\tilde\beta_\infty(\gamma_0)\quad \mbox{for}\quad \tilde\beta^{\infty,p}(\fkc),~ \tilde\beta_\infty(\fkc) ,~\tilde\beta_\infty(\fkc),\quad\mbox{respectively},
$$
and similarly for their analogues without tildes.
For instance, this indicates that $[b]$ does not enter the construction of $\tilde\beta^{\infty,p}(\gamma_0,a)$. In fact, the definitions of $\tilde\beta^{\infty,p}(\gamma_0,a),~\tilde\beta_p(\gamma_0,[b]) ,~\tilde\beta_\infty(\gamma_0)$ are of a local nature and extend in an obvious manner to $\gamma_0$ in $G(\A^{\infty,p})$, $G(\Q_p)$, and $G(\R)_{\el}$, respectively. (We need not require $\gamma_0\in G(\Q)$.)

Put $K(\fkc):=\ker(\pi_1(I_0)\ra \pi_1(G))$ and $\fkE(\fkc):=\fkE(I_0,G; \A/\Q)$.
Define
$$\tilde \beta(\fkc):=\sum_v \tilde\beta_v (\fkc)\in K(\fkc).$$
In our situation, $K(\fkc)_{\Gamma}$ is a torsion abelian group and equipped with a surjection (see \cite[Prop.~1.7.3, \S1.7.5]{KSZ})
$$K(\fkc)_{\Gamma}=K(\fkc)_{\Gamma,\tu{tor}}\ra \fkE(\fkc).$$
The \emph{Kottwitz invariant} $\alpha(\fkc)\in \fkE(\fkc)$ is defined to be the image of $\tilde \beta(\fkc)$, which is independent of the choice of the lifts $\tilde \beta_v$ in the construction.

\subsubsection{Local inner twists of $I_0$}\label{sss:inner-twist-I0}
Starting from acceptable $\fkc=(\gamma_0,a,[b])\in \KP_{\bb}$, let us construct local and global inner twists of $I_0$ as in \cite[\S1.7.11]{KSZ}. 
 The image of $a\in H^1(\A^{\infty,p},I_0)$ in $H^1(\A^{\infty,p},I_0^{\ad})$ determines an inner twist $I_v$ of $I_0$ over $\Q_v$ at each place $v\neq p,\infty$ of $\Q$, which is none other than the identity component of the centralizer of $\gamma_a$ in $G_{\Q_v}$ (up to inner automorphisms). 
 By Lemma \ref{lem:b-is-basic}, $[b]\in B(I_{0,\Q_p})$ maps into the basic subset $B(I^{\ad}_{0,\Q_p})$, which is in bijection with $H^1(\Q_p,I^{\ad}_0)$. Hence $[b]$ gives rise to an inner twist $I_p$ of $I_{0,\Q_p}$. At $\infty$, the choice of $(T_\infty,h)$ as in \S\ref{sss:Kottwitz-invariant} gives a Cartan involution $\Int(h(i))$ on $(I_0/Z_G)_{\R}$. Thus we obtain an inner twist $I_\infty$ of $I_{0,\R}$ such that $I_\infty/Z_{G,\R}$ is $\R$-anisotropic.
 
 If $\alpha(\fkc)$ is trivial, then by \cite[Prop.~1.7.12]{KSZ}, there exists an inner twist $I_\fkc$ of $I_0$ over $\Q$ such that $I_{\fkc,\Q_v}$ is isomorphic to $I_v$ as an inner twist of $I_{0,\Q_v}$ at every place $v$.

\subsubsection{Summary}{\label{sss:summary-Kottwitz-parameters}}
We have assigned the following data to each $\fkc=(\gamma_0,a,[b])\in \KP_{\bb}$:
\begin{itemize}
    \item semisimple conjugacy classes $\gamma_a\in \Gamma( G(\A^{\infty,p}))$ and $\delta_{[b]}\in \Gamma(J_{\bb}(\Q_p) )$,
    \item a finite abelian group $\fkE(\fkc)=\fkE(I_0,G;\A/\Q)$,
    \item the Kottwitz invariant $\alpha(\fkc)\in \fkE(\fkc)$ and local invariants \gls{betav} $\in \pi_1(I_0)_{\Gamma_v}$,
    \item local inner twists $I_v$ of $I_{0,\Q_v}$ if $\fkc$ is acceptable; also an inner twist $I_{\fkc}$ of $I_0$ over $\Q$ localizing to $I_v$ if $\alpha(\fkc)$ is trivial.
\end{itemize}

\subsection{The trace formula for Igusa varieties}

 Continuing from the preceding subsection, we maintain Hypotheses \ref{hypo:integral-model-Shimura}, \ref{hypo:Igusa-existence}, and \ref{hypo:quasi-split-at-p}.

\begin{defn}
A \emph{central character datum} for $G$ is a pair $(\fkX,\chi)$, where
\begin{itemize}
    \item $\fkX=\fkX^\infty \fkX_\infty$ is a closed subgroup of $Z(\A^\infty)\times Z(\R)$ equipped with a Haar measure on $\fkX$, such that $\fkX_\infty\supset A_{G,\infty}$ and that $Z(\Q)\fkX$ is closed in $Z(\A)$,
    \item $\chi:\fkX\ra\C^\times$ is a continuous character which is trivial on $\fkX_{\Q}:=\fkX\cap Z(\Q)$. 
\end{itemize}

\end{defn}

\begin{ex}\label{ex:fkX0}
An irreducible algebraic representation $\xi$ of $G_{\C}$ determines a central character $\omega_\xi:Z(\R)\ra \C^\times$. The following pair is a central character datum:
\begin{equation}\label{eq:fkX0}
    (\fkX_0,\chi_0):=(A_{G,\infty},~ \omega^{-1}_\xi|_{A_{G,\infty}})
\end{equation}
\end{ex}

\begin{ex}\label{ex:fkX-K}
 Let $\xi,\omega_\xi$ be as above. Assume \eqref{eq:trivial-on-Zac}.
 Let $K\subset G(\A^\infty)$ be a neat open compact subgroup. Set
\begin{equation}\label{eq:fkX-K}
\fkX^\infty:=Z(\A^\infty)\cap K, \quad \fkX_\infty:=Z(\R), \qquad
\chi^\infty:=\textbf{1},\quad \chi_\infty:=\omega^{-1}_\xi.
\end{equation}
Then $(\fkX,\chi):=(\fkX^\infty\fkX_\infty,\chi^\infty \chi_\infty)$ is a central character datum. The only nontrivial point is that $\chi$ is trivial on $\fkX_\Q$, namely that $\omega_\xi$ is trivial on $Z(\Q)\cap K$; this follows from \cite[Lem.~1.5.7]{KSZ}.
\end{ex}

\subsubsection{Notation for the trace formula}\label{sss:notation-for-TF}
We introduce the following notation.
\begin{itemize}
\item $(\fkX,\chi)$ is the central character datum of Example \ref{ex:fkX-K}.
    \item $\Sigma_{\R\tu{-ell}}(G)$ is the set of semisimple conjugacy classes in $G(\Q)$ which are elliptic in $G(\R)$.
    \item $\Sigma_{\R\tu{-ell}}(G)/\fkX_{\Q}$ is the set of $\fkX_\Q$-orbits in $\Sigma_{\R\tu{-ell}}(G)$ under the multiplication action.
    \item  Fix a sufficiently divisible $r\in \Z_{>0}$ such that $\bb$ is $r$-decent.
     \item $\cH_{\tu{acc}}(J_{\bb}(\Q_p))$ is defined as in \S\ref{ss:acceptable} by using $\nu=\nu_{\bb}$ and viewing $J_{\bb}$ as an inner form of $M_{\bb}$.  \item For $\phi\in \cH(J_{\bb}(\Q_p))$ and $j\in r\Z$, define the translates $\phi^{(j)}$ as in \S\ref{ss:acceptable}. 
\end{itemize}

For each $\gamma_0\in G(\Q)_{\R\tu{-ell}}$, 
\begin{itemize}
    \item $\KP^{\Fr}_{\bb}(\gamma_0)$ is the set of Kottwitz parameters $\fkc$ whose first entry is equal to $\gamma_0$ (not just stably conjugate) and such that $\alpha(\fkc)$ is trivial,
    \item $\ol{\iota}_G(\gamma_0):=|(G_{\gamma_0}/G_{\gamma_0}^{\circ})(\Q)|\in \Z_{>0}$,
    \item $\gls{c2gamma}:=|\ker(\ker^1(\Q,I_0)\ra \ker^1(\Q,G))|\in \Z_{>0}$.
\end{itemize}

We fix Haar measures on $G(\A^{\infty,p})$, $J_{\bb}(\Q_p)$, $I_0(\A^{\infty,p})$, and $I_0(\Q_p)$.
Suppose that $\fkc\in \KP^{\Fr}_{\bb}(\gamma_0)$ is acceptable. Recall the $\Q$-group $I_{\fkc}$ from \S\ref{sss:inner-twist-I0}.
Equip $I_{\fkc}(\A^\infty)$ with the Haar measure compatible with that on $I_0(\A^\infty)$, and $I_{\fkc}(\Q)$ with the counting measure. Define
$$\gls{c1c}:= \vol(I_\fkc(\Q)\backslash I_{\fkc}(\A^\infty)/\fkX^\infty).$$
Let $(\gamma,\delta)\in \Gamma(G(\A^{\infty,p}))\times \Gamma(J_{\bb}(\Q_p))$ be the image of $\fkc$ under \eqref{eq:KP-to-CKP}. Since $\gamma$ is stably conjugate to $\gamma_0$, we see that $G_{\gamma}^{\circ}$ is an inner form of $I_0$ over $\Q_v$ for $v\neq p,\infty$. So the Haar measure on $I_0(\A^{\infty,p})$ determines a Haar measure on $G_{\gamma}^{\circ}(\A^{\infty,p})$.
Similarly $J_{\bb,\delta}^{\circ}(\Q_p)$ is equipped with a unique Haar measure compatibly with $I_0(\Q_p)$, provided that the following claim is true: that $J_{\bb,\delta}^{\circ}$ is an inner form of $I_0$ over $\Q_p$ if $\delta$ is acceptable.
Let us verify the claim. By taking a $z$-extension $G_1$ of $G$ over $\Q$, we reduce to the case where $G_{\der}=G_{\tu{sc}}$. (Lift $\gamma_0$ to $\gamma_{0,1}\in G_1(\Q)$, $\bb$ to $\bb_1\in G_1(\breve \Q_p)$, and then $b$ to $b_1\in \fkD_p(I_{0,1},G_1;\bb_1)$ via Lemma \ref{liftingatplem} below; just like $\gamma_0$ and $b$ determine $\delta$, we use $\gamma_{0,1}$ and $b_1$ to give $\delta_1\in J_{\bb_1}(\Q_p)$. With these data, it suffices to check the claim on the level of $G_1$.)
The property of having simply connected derived subgroup is inherited by the Levi subgroup $M_{\bb}$ of $G_{\Q_p}$, thus also by the inner form $J_{\bb}$ of $M_{\bb}$. In particular, centralizers of semisimple elements in $M_{\bb}$ and $J_{\bb}$ are connected. 
Since $M_{\bb}$ is a quasi-split inner form of $J_{\bb}$, we can transfer  $\delta\in J_{\bb}(\Q_p)_{\semis}$ to some $\delta^*\in M_{\bb}(\Q_p)_{\semis}$.
Then $J_{\bb,\delta}$ is an inner form of $M_{\bb,\delta^*}$. On the other hand, $\delta^*$ is $\nu_{\bb}$-acceptable since $\delta$ is $\nu_{\bb}$-acceptable, so $M_{\bb,\delta^*}=G_{\delta^*}$. Now $\delta^*$ is conjugate to $\gamma_0$ in $G(\Qpbar)$, so $G_{\delta^*}$ is an inner form of $I_{0,\Q_p}$. The claim is proved.

We define orbital integrals using the measures above, cf.~\eqref{eq:def-orbital-integral},
\begin{align*}
   O^G_{\gamma}(\phi^{\infty,p})
= &\int_{G_\gamma^{\circ}(\A^{\infty,p})\backslash G(\A^{\infty,p})} \phi^{\infty,p}(x^{-1} \gamma x) dx, 
 & & \phi^{\infty,p}\in \cH(G(\A^{\infty,p})),\\
O^{J_{\bb}}_{\delta}(\phi_p)
= & \int_{J_{\bb,\delta}^{\circ}(\Q_p)\backslash J_{\bb}(\Q_p)}
\phi_p(y^{-1} \delta y) dy,
 & & \phi_p\in \cH(J_{\bb}(\Q_p)).
\end{align*}

We are ready to state the conjectural trace formula for Igusa varieties. The condition on $\phi^{\infty,p}$ and $\phi_p$ below comes from compatibility with the central character datum $(\fkX,\chi)$ fixed above.

\begin{conj}\label{conj:Igusa-general}
Given $\phi^{\infty,p}\in \cH(G(\A^{\infty,p}))$ and $\phi_p\in \cH(J_{\bb}(\Q_p))$ which are invariant under the translation by $\fkX^{\infty,p}$ and $\fkX_p$ respectively, there exists $j_0=j_0(\phi^{\infty,p},\phi_p)\in \Z_{>0}$ such that $\phi_p^{(j_0)}\in \cH_{\tu{acc}}(J_{\bb}(\Q_p))$ and the following holds:
  for every integer $j\ge j_0$ divisible by $r$,
  \begin{equation}\label{eq:MackCrane}
  \tr_{\fkX^\infty}\big(\phi^{\infty,p}\phi^{(j)}_p \left|[H_c]_{\bb,\xi} \right.\big)
 = \!
   \sum_{\gamma_0\in \Sigma_{\R\tu{-ell}}(G)/\fkX_{\Q}}  \!\! \frac{c_2(\gamma_0)  \tr \xi (\gamma_0)}{\bar \iota_G (\gamma_0)}  \! \sum_{\fkc\in \KP^{\Fr}_{\bb}(\gamma_0)} c_1 (\mathfrak c) O^G_{\gamma_a}(\phi^{\infty,p})O^{J_{\bb}}_{\delta_{[b]}}(\phi^{(j)}_p),
  \end{equation}
  where $\gamma_a$ and $\delta_{[b]}$ denote the image of $\fkc$ under the map \eqref{eq:KP-to-CKP}.
\end{conj}

\subsubsection{\textit{First simplification}}\label{sss:1st-simp}
Let us point out how Conjecture \ref{conj:Igusa-general} simplifies when $Z_G^{\circ}$ is a cuspidal torus. In this case, $Z_{\tu{ac}}=\{1\}$ so assumption \eqref{eq:trivial-on-Zac} is vacuous; moreover, $Z^{\circ}(\Q)$ is discrete in $Z^{\circ}(\A^\infty)$. (See \cite[Def.~1.6.4, Lem.~1.5.5]{KSZ}.) Therefore $\fkX_{\Q}=Z(\Q)\cap K$ is a finite subgroup of $K$, which must be trivial since $K$ is neat. Hence the first sum in \eqref{eq:MackCrane} is simply over $\gamma_0\in \Sigma_{\R\tu{-ell}}(G)$.

In this case, it is convenient to view formula \eqref{eq:MackCrane} with respect to the central character datum $(\fkX_0,\chi_0)=(A_{G,\infty},~ \omega_\xi|_{A_{G,\infty}})$ from Example \ref{ex:fkX0}, which does not depend on $K$. To see that \eqref{eq:MackCrane} relative to $(\fkX,\chi)$ is equivalent to the same formula relative to $(\fkX_0,\chi_0)$, it is enough to observe that $\fkX_{\Q}=\fkX_{0,\Q}=\{1\}$ and that the volume of $\fkX$ appearing on both sides gets canceled out. (The volume is implicit in the constant $c_1(\fkc)$ as well as the trace on the left hand side.)

\subsubsection{Second simplification}\label{sss:2nd-simp}

We still assume that $Z_G$ is a cuspidal torus, and in addition that  $G_{\tu{der}}=G_{\tu{sc}}$. The simplifications in \S\ref{sss:1st-simp} remain valid.
Moreover the second sum of \eqref{eq:MackCrane} can be taken over classical $\bb$-admissible Kottwitz parameters via Lemma \ref{lem:KP-to-CKP}.

\subsubsection{Known results}\label{sss:known}

Let $(G,X,p,\cG)$ be an integral Shimura datum such that
\begin{itemize}
    \item $(G,X)$ is of abelian type,
    \item $\cG$ is a reductive model of $G_{\Q_p}$ (so $K_p=\cG(\Z_p)$ is hyperspecial).
\end{itemize}
The construction of canonical integral models (Hypothesis \ref{hypo:integral-model-Shimura}) has been carried out by \cite{KisinModels,KimMadapusiPera} for $p>2$ and $p=2$, respectively. Now we restrict $(G,X)$ to a Shimura datum of \emph{Hodge type}.
Then $Z_G$ is a cuspidal torus, so we are in the simple case of \S\ref{sss:1st-simp}. Assumption \eqref{eq:trivial-on-Zac} is vacuous as mentioned in \S\ref{sss:1st-simp}. 
The Igusa varieties satisfying Hypothesis \ref{hypo:Igusa-existence} have been constructed by \cite{CS17,HamacherKim} (also see \cite[\S6]{KretShin}). 
In this situation, \cite{MackCraneShin} proves Conjecture \ref{conj:Igusa-general}, generalizing from the case of PEL-type A or C (\cite{Shin09}):

\begin{thm}\label{thm:MackCrane}
Conjecture \ref{conj:Igusa-general} holds true when $(G,X)$ is of Hodge type and $\cG$~is a reductive model. 
\end{thm}

\begin{proof}
This follows from Lemmas 2.2.7, 4.1.9, and Theorem 4.5.11 of \cite{MackCraneShin}.
\end{proof}

We go back to consider $(G,X)$ of abelian type and allow $\cG$ to be a general parahoric model. Here is progress towards Hypothesis \ref{hypo:integral-model-Shimura} in this case: The integral models are constructed in \cite{KisinPappas,PappasRapoportModels} under certain technical assumptions. See \cite{PappasModels,PappasRapoportModels} for results on a unique characterization. Hypothesis \ref{hypo:Igusa-existence} is known by \cite{HamacherKim} over the integral models of \cite{KisinPappas} if $(G,X)$ is of Hodge type. Conjecture \ref{conj:Igusa-general} is open in this generality.

\begin{rem}\label{rem:LR-for-Igusa}
We cautiously expect that future progress on Conjecture \ref{conj:Igusa-general} would follow the model of \cite{MackCraneShin}, namely by proving the ``Langlands--Rapoport (LR) conjecture for Igusa varieties''. In the notation and terminology of \emph{loc.~cit.}, this conjecture, proven by \cite{MackCraneShin} under the hypothesis of Theorem \ref{thm:MackCrane}, asserts that there exists a $G(\A^{\infty,p})\times J_{\bb}(\Q_p)$-equivariant bijection
\begin{equation}\label{eq:LR-Igusa}
  \Ig_{\bb}(\ol k)\cong \coprod_{\phi} I_{\phi}(\Q)\backslash_{\tau} (G(\A^{\infty,p})\times J_{\bb}(\Q_p)),  
\end{equation}
where the disjoint union is over a set of representatives for isomorphism classes of $\bb$-admissible morphisms; the symbol $\backslash_{\tau}$ means that the left quotient is twisted by an element $\tau(\phi)\in I_{\phi}^{\ad}(\A^{\infty})$.
Moreover \cite{MackCraneShin} establishes this bijection for a family of $\tau(\phi)$ satisfying further conditions, in a way exactly analogous to the Langlands--Rapoport-$\tau$ conjecture of \cite{KSZ}. (See \cite[Thm.~3.6.1]{MackCraneShin} for the precise statement.) In light of recent progress \cite{KSZ,ZhouParahoric,vanHoftenLR} on the LR conjecture for Shimura varieties, it may be possible to extend the LR conjecture for Igusa varieties and Theorem \ref{thm:MackCrane} to the case where $(G,X)$ is of abelian type and $\cG$ is  a reductive model, and perhaps some cases of parahoric $\cG$ as well.
\end{rem}

\section{Stabilization}\label{s:stabilization}

Assuming Conjecture \ref{conj:Igusa-general} on the trace formula for Igusa varieties (cf.~Theorem \ref{thm:MackCrane}), we stabilize the trace formula. The main theorem is Theorem \ref{thm:stable-Igusa}.

\subsection{Initial steps}\label{ss:initial-steps}

\subsubsection{Haar measures}\label{sss:Haar}

For every connected reductive group $G'$ over $\Q$ appearing in this section (e.g., $G$, $G_1$, $H$, $H_1$ in \S\ref{sss:setup-G} and \S\ref{sss:setup-endoscopic-data}), unless it is said otherwise, we equip $G'(\A)$ with the canonical measure \cite[\S2]{OnoTamagawaNumbers} and $G'(\Q)$ with the counting measure. Denote by $A_{G'}$ the maximal $\Q$-split torus in $Z_{G'}$. Upon choosing an isomorphism $A_{G'}\cong \G_m^{\tu{rank}A_{G'}}$, we have an induced isomorphism $A_{G',\infty}:=A_{G'}(\R)^0 \cong (\R^\times_{>0})^{\tu{rank}A_{G'}}$.
The multiplicative Lebesgue measure on $(\R^\times_{>0})^{\tu{rank}A_{G'}}$ is transported to a Haar measure on $A_{G',\infty}$, which is independent of choice of the isomorphism $A_{G'}\cong \G_m^{\tu{rank}A_{G'}}$.

Let $(\fkX',\chi')$ be a central character datum for $G'$. This comes with a Haar measure on $\fkX'$. Set
$$\gls{tauX'(G')}:=\vol(G'(\Q)\backslash G'(\A)/\fkX'),$$
for the unique invariant measure on the double coset determined by the measures on $G'(\A)$, $G'(\Q)$, and $\fkX'$.
When $\fkX'=A_{G',\infty}$, the number $\tau_{\fkX'}(G')$ is the usual Tamagawa number.

\subsubsection{Setup}\label{sss:setup-G}
We assume Hypotheses \ref{hypo:integral-model-Shimura} and \ref{hypo:Igusa-existence} as well as \eqref{eq:trivial-on-Zac}. We retain the notation of \S\ref{s:Igusa}, and most importantly, assume Conjecture \ref{conj:Igusa-general}.

Choose a quasi-split group $G^*$ over $\Q$ that is an inner form of $G$ and fix an inner twist $\psi: G^*_{\ol \Q} \to G_{\ol \Q}$ which gives a cocycle $z \in Z^1(\Q, G^*_{\ad})$. We lift $z$ to $z^{\Rig} \in Z^1(P^{\Rig}_{\dot{V}} \to \cE^{\Rig}_{\dot{V}}, Z(G^*_{\scusp}) \to G^*_{\scusp})$. 

Fix a pinning $(T^*,B^*,\{X_\alpha\})$ for $G^*$ defined over~$\Q$. Write $U^*$ for the unipotent radical of $B^*$ and fix a continuous nondegenerate character $\theta^*:U^*(\Q)\backslash U^*(\A)\ra \C^\times$. In particular, $(B^*,\theta^*)$ determines a global Whittaker datum $\fkw$.

Fix a $z$-extension $ 1 \ra Z_1 \ra G_1 \ra G \ra 1$ over~$\Q$.
Applying \S\ref{sss:z-ext} to $H=G^*$, we obtain a $z$-extension $ 1 \ra Z_1 \ra G^*_1 \ra G^* \ra 1.$
We pull back the pinning for $G^*$ to obtain a pinning $(T^*_1,B^*_1,\{X_{\alpha,1}\})$ for $G^*_1$ over $\Q$. The unipotent radical of $B^*_1$ is $\Q$-isomorphic to $U^*$, so we have a Whittaker datum $\fkw_1$ for $G_1^*$ coming from $(B^*_1,\theta^*)$. Note that $G^*_{\scusp} = G^*_{1, \scusp} = G^*_{1,\der}$. Following the discussion in \S \ref{ss:global transfer factors construction}, $z^{\Rig}$ gives a family of elements  $z^{\Rig}_{\tu{sc},1,v}$ for each place $v$ of $\Q$ as well as local rigid inner twists $(G_{1,v}, \psi_{1,v}, z^{\Rig}_{1,v})$ of $G^*_{1}$. 
Moreover we can lift the Shimura datum $(G,X)$ to $(G_1,X_1)$ such that the map $G_1\ra G$ induces a morphism of Shimura data $(G_1,X_1)\ra (G,X)$. Thus the conjugacy class of cocharacters $[\mu_{X_1}]$ is carried to $[\mu_X]$ under $G_1\ra G$.

We have a central character datum $(\fkX,\chi)$ from Example \ref{ex:fkX-K}. Take $\fkX_1\subset Z_{G_1}(\A)$ to be the preimage of $\fkX$, and define $\chi_1:\fkX_{1,\Q}\backslash \fkX_1\ra \C^\times $ by pulling back~$\chi$. We choose a unique Haar measure on $\fkX_1$ such that $\tau_{\fkX}(G)=\tau_{\fkX_1}(G_1)$.
Then $(\fkX_1,\chi_1)$ is a central character datum for $G_1$.

\subsubsection{Setup for endoscopic data}\label{sss:setup-endoscopic-data}

Choose a subset $\gls{sEheartel(G)}\subset \sE_{\el}(G)$, which is a set of representatives for $\mE_{\el}(G)$. For each $\fke=(H, \cH, s, \eta) \in \sE^\heartsuit_{\el}(G)$, we fix some choices in preparation for the stabilization. We apply \S\ref{sss:z-ext} to obtain a $z$-extension
$$1\ra Z_1 \ra H_1 \ra H\ra 1$$
and an endoscopic datum $\fke_1=(H_1, \LL H_1, s_1, \eta_1) \in \sE(G_1)$. We denote the set of such data by $\sE^\heartsuit_{\el}(G_1)$. Following the procedure in \S\ref{ss:global transfer factors construction}, in light of Remark \ref{rem:y'=1}, we construct a local rigid endoscopic datum $\dot \fke_1=(H_{1,v}, \LL H_{1,v}, \dot{s}_{1,v}, \eta_{1,v})$ for $G_{1,v}$ while taking $\dot y'_{1,v}=1$ (and fixing a choice of $y''_{1,v}$) at each place $v$ of $\Q$. Along the way, we choose $s_{1,\tu{sc}}$, thus also $s_{1,\der}\in \hat G_{1,\der}$. We denote the set of $\dot \fke_1$ obtained in this way by \gls{sEheartrig(G1)}.

Define $\fkX_{H_1}$ to be the image of $\fkX_1$ under the canonical embedding $Z_{G_1}(\A)\ra Z_{H_1}(\A)$. So $\fkX_1= \fkX_{H_1}$ canonically by construction, and $\fkX_{H_1}$ is the preimage of $\fkX\subset Z_G(\A)\subset Z_H(\A)$ in $Z_{H_1}(\A)$ via the diagram \eqref{eq:Z1-ZG-ZH}. There is also a decomposition $\fkX_{H_1}=\fkX^\infty_{H_1} \fkX_{H_1,\infty}$. 
Concretely $\fkX_{H_1,\infty}=Z_{H_1}(\R)$, and $\fkX^\infty_{H_1}$ is the preimage of $Z_G(\A^{\infty})\cap K$. 

The pullback of $\chi:\fkX_{\Q}\backslash\fkX\ra \C^\times$ via $\fkX_{H_1}\ra \fkX$ will be still denoted by $\chi$.
Write $$\lambda_{H_1}:Z_{G_1}(\Q)\backslash Z_{G_1}(\A)\ra \C^\times$$
for the character as in \S\ref{sss:global-GH-regular}, with $H_1,G_1$ playing the roles of $H,G$. We transport the Haar measure from $\fkX_1$ to $\fkX_{H_1}$ via $\fkX_1 = \fkX_{H_1}$. Thereby consider the central character datum $(\fkX_{H_1},\chi_{H_1})$ for $H_1$, where
$\chi_{H_1}:\fkX_{H_1,\Q}\backslash \fkX_{H_1}\ra \C^\times$ is given by
$$\chi_{H_1}:=\chi\cdot \lambda_{H_1}|_{\fkX_{1}}^{-1} .$$
We can view $(\fkX,\chi)$ in \S\ref{sss:setup-G} as a central character datum for $H$ via the canonical embedding $Z_G\ra Z_H$. We have the equality \cite[Lem.~8.2.2]{KSZ}
\begin{equation}\label{eq:tau(H1)=tau(H)}
    \tau_{\fkX_{H_1}} (H_1) = \tau_{\fkX}(H).
\end{equation}

\subsubsection{Transfer factors}

For each $(H_1, \LL H_1, s_1, \eta_1)\in \sE^\heartsuit_{\el}(G_1)$ we choose a decomposition of the canonical global transfer factor $\gls{DeltaA}: H_1(\A_F)_{(G_1,H_1)\tu{-reg}} \times G_1(\A_F)_{\semis} \to \C$ as
$$\Delta_{\A}(\gamma_{H_1},\gamma_1)=\prod_v \Delta[\fkw_{1},z^{\Rig}_{1,v}](\gamma_{H_1,v},\gamma_{1,v}),$$
which is possible by \eqref{eq:Delta-adelic} and the choice $\dot y'_{1,v}=1$ in \S\ref{sss:setup-endoscopic-data}.
For simplicity, we will often write $\Delta_v=\Delta[\fkw_{1},z^{\Rig}_{1,v}]$, and $\Delta^{\infty,p}:=\prod_{v\neq \infty,p} \Delta_v$.

\subsubsection{The first step}

Let $\fkc=(\gamma_0,a,[b])\in \KP_{\bb}$. For each place $v$ of $\Q$, define the Kottwitz sign
$$e_v(\fkc):=e(I_v)\in \{\pm 1\},\qquad e(\fkc):=\prod_v e_v(\fkc),$$
where $I_v$ is the inner form of $I_{0,\Q_v}$ in \S\ref{sss:inner-twist-I0}. 
We also write 
$$e^{\infty,p}(\gamma_0,a):=\prod_{v\neq p,\infty} e(I_v),\qquad e_p(\gamma_0,[b]):=e(I_p),\qquad 
e_\infty(\gamma_0):=e(I_\infty)$$
to make it clear what the signs exactly depend on.

\begin{lem}{\label{lem:initial-steps}}
If $\alpha(\fkc)$ is trivial, then
$$e(\fkc)=1,\qquad c_1(\fkc) c_2(\gamma_0)=\tau_{\fkX}(G)\cdot |\fkK(I_0/\Q)| \cdot \vol(Z(\R)\backslash I_\infty(\R))^{-1}.$$
\end{lem}
\begin{proof}
If $\alpha(\fkc)$ is trivial then $e(\fkc)=\prod_v e(I_{\fkc,\Q_v})=1$ in light of \S\ref{sss:inner-twist-I0} and the product formula for Kottwitz signs. The latter equality is \cite[Lem.~8.1.3]{KSZ}.
\end{proof}

The first equality of the lemma implies that
\begin{equation}\label{eq:FT-on-K}
    \sum_{\kappa\in \fkK(I_0/\Q)} e(\fkc) \langle \alpha(\fkc),\kappa \rangle^{-1} = \begin{cases} 
    |\fkK(I_0/\Q)|, & \mbox{if}~\alpha(\fkc)~\mbox{is trivial},\\
    0, &\mbox{otherwise}.
    \end{cases}
\end{equation}
We are assuming Conjecture \ref{conj:Igusa-general}. Fix $\phi^{\infty,p}$, $\phi_p$, and $j$ as in the conjecture throughout this section. For each $\fkc=(\gamma_0,a,[b])\in \KP_{\bb}$ and $\kappa\in \fkK(I_0/\Q)$, set
\begin{equation}\label{eq:def-of-N}
   N(\gamma_0,\kappa,a,[b],j):=
\langle \alpha(\gamma_0,a,[b]),\kappa\rangle^{-1} e(\gamma_0,a,[b])
\frac{O^G_{\gamma_a}(\phi^{\infty,p}) O^{J_{\bb}}_{\delta_{[b]}}(\phi^{(j)}_p)\tr \xi(\gamma_0)}{\vol(Z(\R)\backslash I_\infty(\R))}. 
\end{equation}

\begin{lem}\label{lem:after-initial-steps}
In the setting of Conjecture \ref{conj:Igusa-general}, equation \eqref{eq:MackCrane} can be rewritten as
  \begin{equation}\label{eq:after-initial-steps}
  \tr_{\fkX^\infty}\big(\phi^{\infty,p}\phi^{(j)}_p \left|[H_c]_{\bb,\xi} \right.\big)
 = \tau_{\fkX}(G)
   \sum_{\gamma_0}
   \ol{\iota}_G(\gamma_0)^{-1}
    \sum_{\kappa} \sum_{(a,[b])} N(\gamma_0,\kappa,a,[b],j),
  \end{equation}
where the sums run over $\gamma_0\in \Sigma_{\R\tu{-ell}}(G)/\fkX_{\Q}$, $\kappa\in\fkK(I_0/\Q)$, and $(a,[b])\in \fkD(I_0,G;\A^{\infty,p})\times \fkD_p(I_0,G;\bb)$, respectively.
\end{lem}

\begin{proof}
We apply Lemma \ref{lem:initial-steps} and \eqref{eq:FT-on-K} to
\eqref{eq:MackCrane}. Then $\tr_{\fkX^\infty}\big(\phi^{\infty,p}\phi^{(j)}_p \left|[H_c]_{\bb,\xi} \right.\big)$ equals
\begin{align*}
&  \sum_{\gamma_0}  \frac{c_2(\gamma_0)  \tr \xi (\gamma_0)}{\bar \iota_G (\gamma_0)}  \! \sum_{\fkc\in \KP^{\Fr}_{\bb}(\gamma_0)} c_1 (\mathfrak c) O^G_{\gamma_a}(\phi^{\infty,p})O^{J_{\bb}}_{\delta_{[b]}}(\phi^{(j)}_p)
\\
= \quad &  \sum_{\gamma_0}
\frac{\tau_{\fkX}(G)}{\ol{\iota}_G(\gamma_0)}
 \! \sum_{\fkc\in \KP_{\bb}(\gamma_0)}
 \sum_{\kappa}
 e(\fkc)\langle \alpha(\fkc),\kappa \rangle^{-1} 
 \frac{O^G_{\gamma_a}(\phi^{\infty,p})O^{J_{\bb}}_{\delta_{[b]}}(\phi^{(j)}_p)\tr \xi(\gamma_0)}{\vol(Z(\R)\backslash I_\infty(\R))},
\end{align*}
where $\gamma_0\in\Sigma_{\R\tu{-ell}}(G)/\fkX_{\Q}$ and $\kappa\in\fkK(I_0/\Q)$.
Hence \eqref{eq:after-initial-steps} holds in view of the definition \eqref{eq:def-of-N}.
\end{proof}

\subsection{Stabilization away from \texorpdfstring{$p$}{p}}\label{ss:stab-away-from-p}

This and the following subsections are devoted to orbital integral identities, for one datum $(H, \cH, s, \eta) \in \sE^\heartsuit_{\el}(G)$ at a time. The choices made in \S\ref{sss:setup-endoscopic-data} will be used freely.

\subsubsection{Stabilization away from $p$ and $\infty$}

To transfer $\phi^{\infty,p}\in \cH(G(\A^{\infty,p}),(\chi^{\infty,p})^{-1})$ to a function on $H_1(\A^{\infty,p})$, we appeal to the usual Langlands--Shelstad transfer.

\begin{lem}\label{lem:SO-away-from-pinfty}
There exists $h_1^{\infty,p}\in \cH(H_1(\A^{\infty,p}),(\chi_{H_1}^{\infty,p})^{-1})$ with the following property. 
For each $\gamma_{H_1}\in H_1(\A^{\infty,p})_{(G_1,H_1)\tu{-reg}}$,
 if $\gamma_{H_1}$ has no image in $G_1(\A^{\infty,p})$ then $SO_{\gamma_{H_1}}(h_1^{\infty,p})=0$. If $\gamma_{H_1}$ has $\gamma_{0,1}\in G_1(\A^{\infty,p})_{\semis}$ as image, then writing $\gamma_0\in G(\A^{\infty,p})_{\semis}$ for the projection of $\gamma_{0,1}$, we have 
    $$SO^{H_1}_{\gamma_{H_1}}(h_1^{\infty,p}) =
\Delta^{\infty,p}(\gamma_{H_1},\gamma_{0,1})
\sum_{a\in \fkD(I_0,G;\A^{\infty,p})}   \langle \tilde \beta^{\infty,p}(\gamma_0,a),\kappa \rangle^{-1}
 e^{\infty,p}(\gamma_0,a) O^G_{\gamma_a}(\phi^{\infty,p}).$$
 \end{lem}

\begin{proof}
We may assume that $\phi^{\infty,p}=\prod_{v\neq \infty,p} \phi_{v}$ with $\phi_{v}\in \cH(G(\Q_v),\chi_v^{-1})$.
Choose a lift $\gamma_{0,1}\in G_1(\A^{\infty,p})$ of $\gamma_0$ and denote the centralizer of $\gamma_{0,1}$ in $G_1$ by $I_{0,1}$ (which is connected since $G_{1,\tu{sc}}=G_{1,\der}$).
Then the idea is to obtain the $\Delta^{\infty,p}$-transfer $h_1^{\infty,p}$ of $\phi^{\infty,p}$ by applying the Langlands--Shelstad transfer and the fundamental lemma from $G_1$ to $H_1$ at each $v\neq\infty,p$ (\S\ref{sss:LS-transfer}, \S\ref{sss:fundamental-lemma}), and then go between $G_1$ and $G$ via the canonical bijection 
$\fkD(I_{0,1},G_1;\A^{\infty,p})\cong \fkD(I_0,G;\A^{\infty,p})$.
See \cite[\S8.2.3]{KSZ} for the details, where the setup is identical (except that a normalization of transfer factors is not specified therein).
\end{proof}

\subsubsection{Stabilization at $\infty$}
Let us construct the function $h_{1,\infty}$ as a suitable linear combination of averaged Lefschetz functions following \cite[\S7]{KottwitzAnnArbor}, dividing into two cases.

\smallskip

\noindent\emph{Case 1. No elliptic maximal torus of $G_{\R}$ is a transfer of a torus of $H_{\R}$.} In this case, we simply set $h_{1,\infty}:=0$.

\smallskip

\noindent\emph{Case 2. An elliptic maximal torus $T_\infty$ of  $G_{\R}$ is a transfer of a torus $T^H_{\infty}$ of $H_{\R}$.} 
We fix such $T_\infty$ and $T^H_{\infty}$, thus also their preimages $T_{1,\infty}$ in $G_{1,\R}$ and $T^{H_1}_{\infty}$ in $H_{1,\R}$. (Then $T^{H_1}_{\infty}$ is an elliptic maximal torus in $H_{1,\R}$.) We fix an isomorphism $j_1:T^{H_1}_{\infty,\C}\isom T_{1,\infty,\C}$, which is canonical up to the action by the Weyl group $\Omega_1$ of $T_{1,\infty,\C}$ in $G_{1,\C}$. We also choose a Borel subgroup $B'_1$ of $G_{1,\C}$ containing $T_{1,\infty,\C}$. This determines a Borel subgroup $B'_{H_1}$ of $H_{1,\C}$ via $j_1$. 
Let 
$$\Delta_{j_1,B'_1}:T^{H_1}_{\infty}(\R)_{(G,H)\tu{-}\reg} \times T_{1,\infty}(\R)\ra \C$$
denote Shelstad's transfer factor \cite[\S3]{ShelstadRealEndoscopy}. (We are following the convention of \cite[p.184]{KottwitzAnnArbor}.) Write $C_\infty\in \C^\times$ for the unique constant such that on $(G_1,H_1)$-regular elements,
\begin{equation}\label{eq:Delta-infty}
    \Delta_\infty=C_\infty\cdot\Delta_{j_1,B'_1},
\end{equation}
where $\Delta_\infty=\Delta[\fkw_1,z_{1,\infty}^{\Rig}]$ by our earlier convention.
Write $\xi_1$ for the representation of $G_1$ that is pulled back from $\xi$ via the surjection $G_1\ra G$. Write $\varphi_{\xi_1}:W_{\R}\ra {}^L G_1$ for the associated $L$-parameter as in \S\ref{sss:Lefschetz}, and $\Phi_{H_1}(\xi_1)$ for the isomorphism classes of $L$-parameters $\varphi_{H_1}:W_{\R}\ra {}^L H_1$ such that $\eta_1\circ \varphi_{H_1}\cong \varphi_{\xi_1}$. We have an injective map \cite[p.185]{KottwitzAnnArbor}
$$\omega_*:\Phi_{H_1}(\xi_1)\ra \Omega_1,$$
determined by the choice of $j_1$ and $B'_1$ above.
Recall from \S\ref{sss:setup-G} that we fixed a Shimura datum $(G_1,X_1)$. Choose $h_1\in X_1$ factoring through $T_{1,\infty}$, which gives rise to a cocharacter $\mu_{h_1}:\G_m\ra T_{1,\infty,\C}$. On the other hand, we can transport $\dot s_{1,\infty}\in Z(\hat{\bar{H_1}})^+$ to $\hat T_{1,\infty}$ via 
$$ Z(\hat{\bar{H_1}})^+ \ra \hat T_{H_1,\infty} \stackrel{j_1}{\cong} \hat T_{1,\infty}. $$
Thus we can evaluate the pairing $\langle \mu_{h_1}, \dot s_{1,\infty} \rangle $.
Finally, define
\footnote{We have put an inverse over the pairings in the definition of $h_{1,\infty}$ and in the formula of Lemma \ref{lem:SO-at-infty}, while there is no inverse in \cite[\S8.2.5]{KSZ}. The difference stems from our use of the Deligne normalization of transfer factors unlike \emph{loc.~cit.}, cf.~\S\ref{ss:local transfer factors}.}
$$
h_{1,\infty}:=C_\infty(-1)^{q(G)} \langle \mu_{h_1}, \dot s_{1,\infty} \rangle^{-1} \sum_{\varphi_{H_1}\in \Phi_{H_1}(\xi_1)} \det(\omega_*(\varphi_{H_1})) f^{\tu{avg}}_{\varphi_{H_1}} ~\in ~ \cH(H_1(\R),\chi_{H_1,\infty}^{-1}),
$$
where the averaged Lefschetz function $f^{\tu{avg}}_{\varphi_{H_1}}$ is as in \S\ref{sss:Lefschetz}. The $\chi_{H_1,\infty}^{-1}$-equivariance is verified as in \cite[\S8.2.5]{KSZ}.

\begin{lem}\label{lem:SO-at-infty}
If $\gamma_{H_1}\in H_1(\R)_{\semis}$ is non-elliptic or non-$(G_1,H_1)$-regular then $SO^{H_1}_{\gamma_{H_1}}(h_{1,\infty})=0$.
For $\gamma_{H_1}\in H_1(\R)_{\el}$ which is $(G_1,H_1)$-regular, there exists a transfer $\gamma_{0,1}\in G_1(\R)_{\el}$; writing $\gamma_0\in G(\R)$ for the image of $\gamma_{0,1}$, we have
$$SO^{H_1}_{\gamma_{H,1}}(h_{1,\infty}) = 
\Delta_\infty(\gamma_{H_1},\gamma_{0,1})
\vol(Z(\R)\backslash I_\infty(\R))^{-1} \langle \tilde\beta_\infty(\gamma_0),\tilde\kappa \rangle^{-1} e_\infty(I_\infty) \tr \xi(\gamma_0).
$$
\end{lem}

\begin{proof}
This is shown in \cite[\S8.2.5]{KSZ} by extending from \cite[\S7]{KottwitzAnnArbor} via $z$-extensions, with $h_{1,\infty}$ defined without the constant $C_\infty$, when the transfer factor is normalized as $\Delta_{j_1,B'_1}$. Our case is immediate from it, in light of \eqref{eq:Delta-infty}.
\end{proof}

\begin{rem}
The constant $C_\infty$ can be pinned down by a spectral pairing. More precisely, fix $\varphi_{H_1}$ as above and choose $j_1$ and $B'_1$ such that $(j_1,B'_1,B'_{H_1})$ is aligned with $\varphi_{H_1}$ in the sense of \cite[p.184]{KottwitzAnnArbor}. Write $\pi'$ for the discrete series representation of $G_1(\R)$ corresponding to $B'_1$. Then the spectral transfer factor $\Delta(\varphi_H,\pi')$ equals 1 with respect to $\Delta_{j_1,B'_1}$, whereas $\Delta(\varphi_H,\pi')=e(G)\langle \dot s,\dot \pi'\rangle$ with respect to $\Delta[\fkw_1,z_{1,\infty}^{\Rig}]$ by \cite[Prop.~5.10]{KalethaLocalRigid} (in the notation thereof). Hence $C_\infty=e(G)\langle \dot s,\dot \pi'\rangle$ in this case.
\end{rem}

\subsection{Stabilization at \texorpdfstring{$p$}{p}}{\label{ss:stab-at-p}}
After some preliminaries, we begin to stabilize the terms at $p$ in \S\ref{sss:stab-at-p}. The reader may jump to \S\ref{sss:stab-at-p} and refer to the preceding materials as needed.

\subsubsection{Endoscopy at $p$}
\label{sss:endoscopy-at-p}
For each datum $(H_{1,p}, \cH_{1,p}, \dot{s}_{1,p}, \eta_{1,p}) \in \sE^{\heartsuit,\Rig}_{p}(G_1)$ we project $\dot{s}_{1,p}$ to $Z(\hat{H_{1,p}})^{\Gamma_{\Q_p}}$ to get a refined datum $(H_{1,p}, \cH_{1,p}, s_{1,p}, \eta_{1,p}) \in \sE^{\Iso}(G_{1,p})$. Denote the set of these refined data by $\sE^{\heartsuit,\Iso}_{p}(G_1)$. We caution the reader that $s_{1,p}$ is the projection of $\dot{s}_{1,p}$ to $Z(\widehat{H_{1,p}})^{\Gamma_{\Q_p}}$ which need not equal the image of $s_1$ in $Z(\widehat{H_{1,p}})$ under the identification $Z(\widehat{H_1})=Z(\widehat{H_{1,p}}) $. The inner twist $(G_{1,p}, \psi_{1,p}, z^{\Rig}_{1,p})$ induces a unique $L$-isomorphism $\LL G^*_{1,p} \cong \LL G_{1,p}$ preserving the fixed pinnings. 
Hence the above endoscopic data may be considered associated with either $G_{1,p}$ or $G^*_{1,p}$.

Following \cite[Lem.~5.3.8]{KretShin}, we choose a lift $\bb_1 \in G_1(\Q_{p^r})$ of $\bb$ such that $\nu_{\bb_1}$ is defined over $\Q_p$ and projects to $\nu_{\bb}$. We fix $A^*_1, T^*_1, B^*_1$ as in \S \ref{BGreview}, and then get a Levi subgroup $M_{\bb_1} \subset G^*_{1,p}$ that is an inner form of $J_{\bb_1}$. Our fixed pinning of $G^*_1$ and non-degenerate character $\theta^*$ restrict to give a Whittaker datum $\fkw_{\bb_1}$ of $M_{\bb_1}$. 

As in \S \ref{effectiveendoscopy}, consider the set $\mE^i_{\eff}(J_{\bb_1}, G^*_{1,p} ; H_{1,p})$. We will now choose a set \gls{sEheartb(H)} of representatives for $\mE^i_{\eff}(J_{\bb_1}, G^*_{1,p} ; H_{1,p})$. 
We require $H_{\bb_1}$ to be a standard Levi subgroup of $H_{1,p}$ relative to $B_{H_1}$.

Given this data, we produce for each element of $\sE^\heartsuit_{\bb_1}(H_{1,p})$ a $\Q_p$-homomorphism $\D \xrightarrow{\nu} H_{1,p}$. For $(H_{\bb_1}, \cH_{\bb_1}, H_{1,p}, \cH_{1,p}, s_{1,p}, \Int(n) \circ \eta_{\bb_1}) \in \sE^\heartsuit_{\bb_1}(H_{1,p})$, we can choose a maximal torus $T'_1 \subset M_{\bb_1}$ such that $T_{H_1}$ transfers to $T'_1$. Hence we have a canonical isomorphism $T_{H_1} \cong T'_1$
(up to our choice of pinnings). We then define $\nu$ as the composition
\begin{equation*}
    \D \xrightarrow{\ov{\nu}_{\bb_1}} A_{M_{\bb_1}} \subset T'_1 \cong T_{H_1} \subset H_{1,p}.
\end{equation*}
It is easy to check that this does not depend on our choice of $T'_1$.

By replacing $s_{1,p}$ everywhere with $\dot{s}_{1,p}$ and forgetting $H_{1,p}, \cH_{1,p}$, we get from $\sE^\heartsuit_{\bb_1}(H_{1,p})$ a set $\sE^{\heartsuit,\Rig}_{\bb_1}(H_{1,p}) \subset \sE^{\Rig}(M_{\bb_1})$.

\subsubsection{Comparison of transfer factors}{\label{sss:comparison-trans-factors}}

Recall that we have equipped $G^*_{1,p}$ with the data of a rigid inner twist $(G_{1,p}, \psi_{1,p}, z^{\Rig}_{1,p})$. As discussed in \S \ref{BGreview}, there exists $g \in G^*_1(\ol \Q_{p^r})=G^*_1(\ol \Q_p)$ such that the restriction of $\Int(g) \circ \psi_{1,p}$ to $M_{\bb_1}$ gives an inner twist $\psi_{\bb_1}: M_{\bb_1, \ol \Q_p} \to J_{\bb_1, \ol \Q_p}$. In particular, we have the following commutative diagram of maps defined over $\ov{\Q_p}$.
\begin{equation*}
\begin{tikzcd}
G^*_{1, \ol \Q_p} \arrow[rr, "\Int(g) \circ \psi_{1,p}"] & &  G_{1, \ol \Q_p} \\
M_{\bb_1, \ol \Q_p} \arrow[rr, "\psi_{\bb_1}"] \arrow[u, hook] & & J_{\bb_1, \ol \Q_p} \arrow[u, hook, "\iota_{\bb_1}"'].
\end{tikzcd}    
\end{equation*}
The map $\iota_{\bb_1}$ gives an isomorphism onto its image and satisfies $\iota_{\bb_1} \circ \sigma(\iota^{-1}_{\bb_1}) = \Int(\bb_1)$ since $\iota_{\bb_1}$ is equivariant for the standard action of $W_{\Q_p}$ on $J_{\bb_1}$ and the twisted action on $G_1$ (see \S \ref{BGreview}). Then $\bb_1$ lifts to a cocycle in $z^{\Iso}_{\bb_1} \in Z^1(\cE^{\Iso}_{\Q_p}, G_1)$ (as in \eqref{eq: BGHalgbij}) and by pullback we get a cocycle $z^{\Rig, \prime}_{\bb_1}$ of $\cE^{\Rig}_{\Q_p}$ valued in~$G_1$. For ease of notation, we denote $\Int(g) \circ \psi_{1,p}$ by $\psi_{1,p, (g)}$. Let
$$z^{\Rig}_{1,p, (g)} \in Z^1(\cE^{\Rig}_{\Q_p}, G^*_{1, \ol \Q_p})\quad \mbox{be given by}\quad w \mapsto \psi^{-1}_{1,p, (g)}z^{\Rig}_{1,p}(w)w(\psi_{1,p,(g)})$$
and note that $\psi^{-1}_{1,p, (g)} \circ w(\psi_{1,p, (g)})=\Int(z^{\Rig}_{1,p, (g)}(w))$ as automorphisms of $G^*_{1, \ol \Q_p}$.  Then define 
$$z^{\Rig}_{\bb_1}(w) := \psi^{-1}_{1,p, (g)}(z^{\Rig, \prime}_{\bb_1}(w))z^{\Rig}_{1,p, (g)}(w)\quad\mbox{for}\quad w \in \cE^{\Rig}_{\Q_p}.$$ 
We have the following lemma.
\begin{lem}
The triple $(J_{\bb_1}, \psi_{\bb_1}, z^{\Rig}_{\bb_1})$ is a rigid inner twist of $M_{\bb_1}$.
\end{lem}
\begin{proof}
It is routine to check that $z^{\Rig}_{\bb_1}$ is an algebraic $1$-cocycle of $\cE^{\Rig}_{\Q_p}$ valued in $G^*_{1,\Qpbar}$.
We have a commutative diagram:
\begin{equation*}
    \begin{tikzcd}
    M_{\bb_1, \ol \Q_p} \arrow[rrd, "\psi_{\bb_1}"'] \arrow[rr, "\psi_{1,p, (g)}"]  & & \im(\iota_{\bb_1}) \arrow[d, "\iota^{-1}_{\bb_1}"]\\
    && J_{\bb_1, \ol \Q_p}.
    \end{tikzcd}
\end{equation*}
By equation \eqref{eq: triplecocycles} of Lemma \ref{cocyclelem}, we have that $z^{\Rig}_{\bb_1}(w) = \psi^{-1}_{\bb_1} \circ w(\psi_{\bb_1})$ as automorphisms of $M_{\bb_1, \ol \Q_p}$. Since $\psi^{-1}_{\bb_1} \circ w(\psi_{\bb_1}) \in M_{\bb_1, \ad}(\ol \Q_p)$, it follows that $z^{\Rig}_{\bb_1}(w)$ commutes with $Z(M_{\bb_1})$ and hence factors through $M_{\bb_1}$.
\end{proof}
\begin{rem}
We note that in general, it will not be possible to express $J_{\bb_1}$ as an extended pure inner twist of $M_{\bb_1}$.
\end{rem}

We temporarily fix a $\nu_{\bb_1}$-acceptable, semisimple element $\gamma_{0,1} \in G_1(\Q_p)$ as well as an element $b_1 \in I_{0,1}(\Q_{p^r})$ that is a decent representative of a class $[b_1] \in B(I_{0,1})$ that maps to $[\bb_1] \in B(G_1)$.  
We suppose further that there exists $\gamma_1 \in M_{\bb_1}(\Q_p)$ such that $\psi_{1,p}(\gamma_1)$ is stably conjugate to $\gamma_{0,1}$. By \S \ref{sss:KP-to-CKP}, have an associated conjugacy class $\delta_{[b_1]} \in \Gamma(J_{\bb_1}(\Q_p))$ and  by assumption, $(\iota_{\bb_1} \circ \psi_{\bb_1})(\gamma_1)$ and $\iota_{\bb_1}(\delta_{[b_1]})$ are stably conjugate in $G_1$. It follows from \cite[Lem.~3.6]{Shin09} that $\gamma_1$ and $\psi_{\bb_1}^{-1}(\delta_{[b_1]})$ are stably conjugate in $M_{\bb_1}$.
\begin{lem}{\label{invariantcalclem}}   
We have the equality:
\begin{equation*}
    \langle \inv[z^{\Rig}_{1,p}](\gamma_1, \gamma_{0,1}), \dot{s}_{1,p} \rangle\langle \inv[z^{\Rig}_{\bb_1}](\gamma_1, \delta_{[b_1]}), \dot{s}_{1,p} \rangle^{-1} = \langle \beta_p(\gamma_{0,1}, [b_1]), s_{1,p} \rangle,
\end{equation*}
where $s_{1,p}$ is the image of $\dot{s}_{1,p}$ in $Z(\hat{H})^{\Gamma_{\Q_p}}$.
\end{lem}
\begin{proof}
We recall that $\beta_p(\gamma_{0,1}, [b_1])$ is defined to be $\kappa_{I_{0,1}}([b_1])$ and that there is $c \in G_1(\breve{\Q}_p)$ such that $b_1 = c^{-1}\bb_1 \sigma(c)$ and $\iota_{\bb_1}(\delta_{[b_1]}) = c \gamma_{0,1} c^{-1}$. In fact, since $b_1$ and $\bb_1$ are $n$-decent for some sufficiently divisible $n$, it follows from \cite[Corollary 1.10]{RapoportZink} that $c \in G_1(\Q_{p^n}) \subset G_1(\Q^{\tu{ur}}_p)$.  We fix $d_1 \in M_{\bb_1}(\ov{\Q_p})$ satisfying $\psi_{\bb_1}(d_1\gamma_1d^{-1}_1) = \delta_{[b_1]}$, noting that such a $d_1$ exists by the discussion immediately before this lemma. Putting $d_2 := \psi^{-1}_{1,p, (g)}(c^{-1})d_1$, we check that
\begin{equation*}
   \psi_{1,p, (g)}( d_2 \gamma_1 d^{-1}_2)
= c^{-1}\psi_{1,p, (g)}(d_1\gamma_1 d_1^{-1})c
= c^{-1} \iota_{\bb_1}(\psi_{\bb_1}(d_1 \gamma_1 d_1^{-1})) c
= c^{-1} \iota_{\bb_1}(\delta_{[b_1]}) c
=    \gamma_{0,1}.
\end{equation*}

The element $\bb_1$ corresponds to an algebraic $1$-cocycle $z^{\Iso}_{\bb_1} \in Z^1_{\alg}(\cE^{\Iso}_{\Q_p}, G_1)$ and as in \ref{BGreview}, we have $\kappa_{G_1}(\bb_1) = \kappa_{G_1}(z^{\Iso}_{\bb_1})$. Similarly, $b_1$ corresponds to a $z^{\Iso}_{b_1} \in Z^1_{\alg}(\cE^{\Iso}_{\Q_p}, I_{0,1})$ and this cocycle satisfies $z^{\Iso}_{b_1}(w) = c^{-1} z^{\Iso}_{\bb_1}(w) w(c)$ for each $w \in \cE^{\Iso}_{\Q_p}$. This pulls back to a cocycle $z^{\Rig}_{b_1} \in Z^1(u \to \cE^{\Rig}_{\Q_p}, Z(I_{0,1}) \to I_{0,1})$. By, the aforementioned preservation of Kottwitz maps, Lemma \ref{rig vs iso transfer factors}, and equation \eqref{eq: explicit inv} we have
\begin{equation*}
     \langle \beta_p(\gamma_{0,1}, [b_1]), s_{1,p} \rangle = \langle z^{\Iso}_{b_1}, s_{1,p} \rangle= \langle z^{\Rig}_{b_1}, \dot{s}_{1,p} \rangle = \langle \inv[z^{\Rig,\prime}_{\bb_1}](\gamma_{0,1}, \delta_{[b_1]}), \dot{s}_{1,p} \rangle.
\end{equation*}

Finally, observe that from $\nu_{\bb_1}$-acceptability, the various centralizers  $I_{\gamma_1}, I_{0,1},I_{\delta_1}$ are all inner forms  and we have a commutative diagram 
\begin{equation*}
\begin{tikzcd}
I_{\gamma_1, \ol \Q_p} \arrow[rr, "\psi_{1,p, (g)} \circ \Int(d_2)"] \arrow[rrd, swap, "\psi_{\bb_1} \circ \Int(d_1)"] && I_{0,1, \ol \Q_p} \arrow[d, "\iota^{-1}_{\bb_1} \circ \Int(c)"] \\
&& I_{\delta_{[b_1]}, \ol \Q_p}.
\end{tikzcd}    
\end{equation*}
By construction, the cocycles  $[z^{\Rig}_{1,p}]_{\gamma_1, \gamma_{0,1}}, [z^{\Rig}_{\bb_1}]_{\gamma_1, \delta_{[b_1]}}, z^{\Rig}_{b_1}$ (notation is as in \eqref{eq: explicit inv}) are those corresponding to the above maps via $w \mapsto \psi^{-1} \circ w(\psi)$, starting from the horizontal map and then counterclockwise. Moreover, we claim the following equality of these cocycles
\begin{equation}{\label{eq: transitivityeqn}}
  (\psi_{1,p, (g)} \circ \Int(d_2))^{-1}(z^{\Rig}_{b_1}) [z^{\Rig}_{1,p}]_{\gamma_1, \gamma_{0,1}}= [z^{\Rig}_{\bb_1}]_{\gamma_1, \delta_{[b_1]}}.
\end{equation}
Before verifying this claim, we first note that by definition of $d_2$,
\begin{equation}{\label{eq: d2trick}}
    \psi_{1,p, (g)} \circ \Int(d_2) = \Int(c^{-1}) \circ \psi_{1,p, (g)} \circ \Int(d_1).
\end{equation}

Now, expanding the right-hand side of \eqref{eq: transitivityeqn} gives
\begin{equation*}
   [z^{\Rig}_{\bb_1}]_{\gamma_1, \delta_{[b_1]}}(w) = d^{-1}_1 z^{\Rig}_{\bb_1}(w)w(d_1) = d^{-1}_1 \psi^{-1}_{1,p, (g)}(z^{\Rig, \prime}_{\bb_1}(w))z^{\Rig}_{1,p, (g)}(w)w(d_1).  
\end{equation*}
We then apply $\psi_{1,p, (g)} \circ \Int(d_2)$ using \eqref{eq: d2trick} to get
\begin{equation*}
    c^{-1}z^{\Rig, \prime}_{\bb_1}(w) \psi_{1,p, (g)}(z^{\Rig}_{1,p,(g)}(w) w(d_1)d^{-1}_1)c.
\end{equation*}

On the other hand, expanding the left-hand side of \eqref{eq: transitivityeqn} and applying $\psi_{1,p, (g)} \circ \Int(d_2)$ gives
\begin{equation}{\label{eq: left-hand side}}
        z^{\Rig}_{b_1}(w) (\psi_{1,p, (g)} \circ \Int(d_2))([z^{\Rig}_{1,p}]_{\gamma_1, \gamma_{0,1}}(w)) =  c^{-1}z^{\Rig, \prime }_{\bb_1}(w) w(c)\psi_{1,p, (g)}(z^{\Rig}_{1,p, (g)}(w) w(d_2)d^{-1}_2).
\end{equation}
Observe that we have 
\begin{equation*}
   \psi_{1,p, (g)}(d^{-1}_2)=\psi_{1,p, (g)}(d^{-1}_1)c, 
\end{equation*}
and
\begin{align*}
\psi_{1,p, (g)}(w(d_2)) &=(\psi_{1,p, (g)} \circ w(\psi^{-1}_{1,p, (g)}))(w(c^{-1}))\psi_{1,p, (g)}(w(d_1))\\
&=\psi_{1,p, (g)}(z^{\Rig}_{1,p, (g)}(w))^{-1}w(c^{-1})\psi_{1,p, (g)}(z^{\Rig}_{1,p, (g)}(w))\psi_{1,p, (g)}(w(d_1))
\end{align*}
so that after cancellation, the right-hand side of \eqref{eq: left-hand side} equals
\begin{equation*}
   c^{-1}z^{\Rig, \prime }_{\bb_1}(w) \psi_{1,p, (g)}(z^{\Rig}_{1,p, (g)}(w)w(d_1)d^{-1}_1)c.
\end{equation*}
Hence \eqref{eq: transitivityeqn} is proved.

From equations \eqref{eq: transitivityeqn} and \eqref{eq: explicit inv}, it follows that
\begin{equation}{\label{eq: refinedtransitivityeqn}}
   (\psi_{1,p, (g)} \circ \Int(d_2))^{-1}(\inv[z^{\Rig, \prime}_{\bb_1}](\gamma_{0,1}, \delta_{[b_1]})) \inv[z^{\Rig}_{1,p}](\gamma_1, \gamma_{0,1})= \inv[z^{\Rig}_{\bb_1}](\gamma_1, \delta_{[b_1]}).
\end{equation}
Thereby we obtain the desired equality using Lemma \ref{rigcocycletwistlem}.
\end{proof}

\begin{lem}{\label{transfactLevi}}
Let $(H_{\bb_1}, \cH_{\bb_1}, H_{1,p}, \cH_{1,p}, s_{1,p}, \Int(n) \circ \eta_{\bb_1}) \in \sE^\heartsuit_{\bb_1}(H_{1,p})$. Suppose $\gamma_{H_{\bb_1}} \in H_{\bb_1}(\Q_p)_{\semis}$ is $(G^*_1, H_{1,p})$-regular and $\gamma_1 \in M_{\bb_1}(\Q_p)_{\semis}$. 
We have the following equality of transfer factors
\begin{equation*}
  |D^{G^*_1}_{M_{\bb_1}}(\gamma_1)|^{\frac{1}{2}}_p |D^{H_{1,p}}_{H_{\bb_1}}(\gamma_{H_{\bb_1}})|^{-\frac{1}{2}}_p\Delta[\fkw_{\bb_1}](\gamma_{H_{\bb_1}}, \gamma_1) = \Delta[\fkw_1](\gamma_{H_{\bb_1}}, \gamma_1).
\end{equation*}
(See \S\ref{ss:notation} for the definition of Weyl discriminants appearing on the left hand side.)
\end{lem}
\begin{proof}
By \cite[Prop.~5.3]{BMaveraging}, cf.~\cite[Lem.~6.5]{WaldspurgerFLimpliesTC}, this equality would be true if $\Delta[\fkw](\gamma_{H_{\bb_1}}, \gamma_1)$ were the transfer factor relative to $(H_{1,p}, \cH_{1,p}, s_{1,p}, \Int(n) \circ \eta_{\bb_1})$. Hence we need only show that the Whittaker-normalized transfer factors for $(H_{1,p}, \cH_{1,p}, s_{1,p}, \Int(n) \circ \eta_{\bb_1})$ and $(H_{1,p}, \cH_{1,p}, s_{1,p}, \eta_{\bb_1})$ agree. The transfer factor depends on a fixed admissible isomorphism between maximal tori of $H_{1,p}$ and $G_{1,p}$ that maps $\gamma_{H_{\bb_1}}$ to $\gamma_1$. Because the two endoscopic data agree up to the conjugation of $\eta_{\bb_1}$ by $\Int(n)$, the same admissible isomorphism can be chosen for both. In constructing the transfer factor, we may also fix the same $\chi$-data and $a$-data. Hence it is readily apparent that all terms agree in the two transfer factors. We remark that for $\Delta_I = \langle \lambda, s^{\fke} \rangle$, where $\lambda$ is the so-called \emph{splitting invariant}, we have that $s^{\fke} \in \hat{G_{1,p}}$ is defined to be $s$ transported to $\hat{G_{1,p}}$ via the dual of the admissible isomorphism, and so this $s^{\fke}$ agrees for both endoscopic data.
\end{proof}
\begin{cor}{\label{transfactcomparison}}
Let $(H_{\bb_1}, \cH_{\bb_1}, H_{1,p}, \cH_{1,p}, \dot s_{1,p}, \Int(n) \circ \eta_{\bb_1}) \in \sE^{\heartsuit,\Rig}_{\bb_1}(H_{1,p})$. Using the same notation as Lemma \ref{invariantcalclem}, we have the following equality:
\begin{equation*}
    \Delta[\fkw_1, z^{\Rig}_{1,p}](\gamma_{H_{\bb_1}}, \gamma_{0,1})=\langle \beta_p(\gamma_{0,1},[b_1]), s_{1,p})\rangle|D^{G^*_1}_{M_{\bb_1}}(\gamma_1)|^{\frac{1}{2}}_p |D^{H_{1,p}}_{H_{\bb_1}}(\gamma_{H_{\bb_1}})|^{-\frac{1}{2}}_p\Delta[\fkw_{\bb_1}, z^{\Rig}_{\bb_1}](\gamma_{H_{\bb_1}}, \delta_{[b_1]}).
\end{equation*}
\end{cor}
\begin{proof}
This follows from Lemmas \ref{invariantcalclem} and \ref{transfactLevi} and equation \eqref{eq: rigtransferinnerforms}. 
\end{proof}

\subsubsection{Stabilization at $p$}\label{sss:stab-at-p}

We now return to the notation of 
\S \ref{ss:initial-steps}. For each datum $(H, \cH, s, \eta) \in E^{\heartsuit}_{\el}(G)$ we can localize at $p$ to get a standard endoscopic datum in $\sE(G_{\Q_p})$. We denote the set of data produced in this way by $\sE^{\heartsuit}_p(G)$. We then define the set $\ES^{\heartsuit}_p(G)$ 
as $\{ (\fke, \gamma_{H_p}) : \fke \in \sE^{\heartsuit}_p(G), \gamma_{H_p} \in \Sigma(H_p)_{(G,H_p)\tu{-}\reg}\}$. Following \eqref{EKESbij} (cf. \cite[Lem.~9.7]{KottwitzEllipticSingular}), we get a map $\ES^{\heartsuit}_p(G_{\Q_p}) \to \EK(G_{\Q_p})$. 
\begin{rem}
To orient the reader, we compare the set $\ES^{\heartsuit}_p(G_{\Q_p})$ to the sets appearing in \S \ref{sss:endoscopy-at-p} such as $\sE^\heartsuit_{\bb_1}(H_{1,p})$. The point is that given a standard endoscopic datum $(H, \cH, s, \eta)$ of $G$ over $\Q$, we can get a local endoscopic datum at $p$ in two ways. One way is to simply base-change to $\Q_p$ as we have done in this subsection. This produces a standard endoscopic datum over $\Q_p$ that may not be refined since $s$ need not lie in $Z(\widehat{H})^{\Gamma_{\Q_p}}$. This endoscopic datum is the one mostly closely related to the ``$p$-part'' of \eqref{eq:after-initial-steps} since our $s$ will yield a $\kappa \in \fkK(I_0/\Q)$.

On the other hand, one can localize $(H, \cH, s, \eta)$ as in \ref{sss:global-strongly-regular} to produce a rigid endoscopic datum $(H_p, \cH_p, \dot{s}_p, \eta_p)$ over $\Q_p$. This rigid endoscopic datum yields a refined endoscopic datum by projection of $\dot{s}_p$ to $Z(\widehat{H_p})^{\Gamma_{\Q_p}}$. This is essentially the approach followed in \S \ref{sss:endoscopy-at-p} and is the one that yields an appropriately normalized local transfer factor from the perspective of endoscopy for rigid inner forms. 

The goal  is then to compare these two notions of local endoscopy. This is accomplished by combining Corollary \ref{transfactcomparison} with equation \ref{eq: kappa vs s}.
(This is reminiscent of the stabilization at $p$ for Shimura varieties in \cite[\S7]{KottwitzAnnArbor}. The stabilization there is carried out at first when $s\in Z(\widehat{H_p})^{\Gamma_{\Q_p}}$; the general case is reduced to this case by translating $s$ by an element of $Z(\widehat{G})$, which introduces a factor analogous to our $\mu_{h_1}(y_{1,p}g^{-1})$ showing up below.)
\end{rem}

We now fix a datum $\fke_p^*=(H_p, \cH_p, s^*_p, \eta_p) \in \sE^{\heartsuit}_p(G)$ and  $\gamma_{H_p} \in \Sigma(H_p)_{(G,H_p)\tu{-}\reg}$. We use the $*$ superscript to distinguish $s^*_p$ from the elements $s_{1,p}$ of \S\ref{sss:endoscopy-at-p}. The image under this map is to be denoted by $(\gamma_0, \kappa)\in \EK(G_{\Q_p})$. Our fixed datum $\fke_p^*$ yields a lift of $\kappa$ denoted $\tilde{\kappa} \in Z(\hat{I_0})$ corresponding to $s^*_p$ 
via the construction of \cite[Lem.~9.7]{KottwitzEllipticSingular}. 
Choose a lift $\gamma_{0,1} \in G_1(\Q_p)$ of $\gamma_0$. We also have an endoscopic datum $(H_{1,p}, \cH_{1,p}, s^*_{1,p}, \eta_{1,p})$ of $G_{1, \Q_p}$, corresponding to $\fke_p^*$ via the construction in \S\ref{sss:z-ext}. This yields an element $\tilde{\kappa}_1 \in Z(\widehat{I_{0,1}})$ which equals the image of $\tilde{\kappa}$ under the map $Z(\widehat{I_0}) \to Z(\widehat{I_{0,1}})$.

Our goal in this subsection is to rewrite the expression
\begin{equation}{\label{atpequation}}
    \Delta_p(\gamma_{H_{1,p}}, \gamma_{0,1})\sum_{[b]\in \fkD_p(I_0,G,\bb)}  \langle \tilde\beta_p(\gamma_{0},[b]),\tilde{\kappa} \rangle^{-1} e_p(\gamma_{0},[b])O^{J_{\bb}}_{\delta_{[b]}}(\phi^{(j)}_p),
\end{equation}
where $\Delta_p=\Delta[\fkw_1, z^{\Rig}_{1,p}]$,
in terms of stable orbital integrals on endoscopic groups of $G_{1,p}$. Henceforth, we simplify the notation by dropping the superscript $(j)$ from $\phi_p$ and assume that $\phi_p \in \cH_{\tu{acc}}(J_{\bb}(\Q_p))$.
Our first task is to lift \eqref{atpequation} from $G$ to $G_{1}$. 
\begin{lem}{\label{liftingatplem}}
We have a canonical bijection
\begin{equation*}
    \fkD_p(I_{0,1}, G_1, \bb_1)=\fkD_p(I_0, G, \bb),
\end{equation*}
as well as a row-exact commutative diagram 
\begin{equation*}
\begin{tikzcd}
        1 \arrow[r]  &\pi_1(Z_1) \arrow[r, hook]\arrow[d, two heads] & \pi_1(I_{0,1}) \arrow[r, two heads] \arrow[d, two heads] & \pi_1(I_0) \arrow[r] \arrow[d, two heads] & 1\\
         & \pi_1(Z_1)_{\Gamma_{\Q_p}} \arrow[r, "\ov{\epsilon}"]& \pi_1(I_{0,1})_{\Gamma_{\Q_p}} \arrow[r, two heads] & \pi_1(I_0)_{\Gamma_{\Q_p}} \arrow[r] & 1
\end{tikzcd}
\end{equation*}
induced by the exact sequence $1\ra Z_1 \ra I_{0,1} \to I_0 \ra 1$. 
\end{lem}
\begin{proof}
The second fact is standard and follows from \cite[(1.8.1)]{KottwitzCuspidalTempered}. We now prove that the natural map $\fkD_p(I_{0,1}, G_1, \bb_1) \to \fkD_p(I_0, G, \bb)$ is a bijection. By \cite[\S4.5]{KottwitzIsocrystal2} and \cite[Prop.~10.4]{KottwitzglobalB(G)}, $B(Z_1)$ acts transitively on each fiber of the maps $B(I_{0,1}) \to B(I_0)$ and $B(G_1) \to B(G)$. The action is equivariant for the map $B(I_{0,1}) \to B(G)$. Moreover, the stabilizer of each fiber is given by $\ker(H^1(\Q_p, I_{0,1}) \to H^1(\Q_p, I_0))$ and $\ker(H^1(\Q_p, G_1) \to H^1(\Q_p, G))$ respectively. Since $H^1(\Q_p, Z_1) = \{1\}$, these stabilizers are trivial. This implies the desired result. 
\end{proof}

Each $[b] \in \fkD_p(I_{0}, G, \bb)$ gives an element $[b_1] \in \fkD_p(I_{0,1}, G_1, \bb_1)$ by Lemma \ref{liftingatplem}. Define $\beta_{1,p} = \kappa_{I_{0,1}}([b_1]) \in \pi_1(I_{0,1})_{\Gamma_{\Q_p}}$. 

It follows from the diagram in Lemma \ref{liftingatplem} that $\beta_{1,p} \in \pi_1(I_{0,1})_{\Gamma_{\Q_p}}$ lifts to an element  $\tilde{\beta}_{1,p} \in \pi_1(I_{0,1})$ which projects to $\tilde{\beta}_p \in \pi_1(I_0)$. By construction, we have $\langle \tilde{\beta}_{1,p}, \tilde{\kappa}_1 \rangle = \langle \tilde{\beta}_p, \tilde{\kappa} \rangle$. Since $I_p$ and $I_{1,p}$ have the same adjoint group, the Kottwitz signs match: $e_p(\gamma_{0},[b]) = e_p(\gamma_{0,1},[b_1])$ (see \cite{KottwitzSign}). Moreover $O^{J_{\bb}}_{\delta_{[b]}}(\phi_p) = O^{J_{\bb_1}}_{\delta_{[b_1]}}(\phi_{1,p})$, where $\phi_{1,p} \in \cH_{\tu{acc}}(J_{\bb_1}, \mathbf{1}_{\fkX_{1,p}})$ is the pullback of $\phi_p$ under the map $J_{\bb_1}(\Q_p) \to J_{\bb}(\Q_p)$.
Since $\phi_p$ is invariant under $\fkX_p$, the function $\phi_{1,p}$ is invariant under its preimage $\fkX_{1,p}$ in $J_{\bb_1}(\Q_p)$.

Hence, by the equalities of the previous paragraph, the expression in \eqref{atpequation} becomes
\begin{equation*}
      \Delta_p(\gamma_{H_{1,p}}, \gamma_{0,1})\sum_{[b_1]\in \fkD_p(I_{0,1},G_1,\bb_1)}  \langle \tilde{\beta}_{1,p},\tilde{\kappa}_1 \rangle^{-1} e_p(\gamma_{0,1},[b_1])O^{J_{\bb_1}}_{\delta_{[b_1]}}(\phi_{1,p}).
\end{equation*}

Recall from \S\ref{sss:setup-endoscopic-data} that we had $s_{1,\der}\in \hat G_{1,\der}$ and chose $y'_{1,p} = 1$ so that $y_{1,p} = y''_{1,p}$ in the definition of $\dot{s}_{1,p}$ as in \S \ref{ss:global transfer factors construction}. Let $g \in Z(\hat{G_1})$ be such that $\tilde{\kappa}_1=s_{1,\der}g$. 

Recall that $s_{1,p}$ is defined as the projection of $\dot{s}_{1,p}$ to $Z(\hat{H_1})^{\Gamma_{\Q_p}}$ and that $\bar{\eta}_1(\dot{s}_{1,p}) = (s_{\scusp}, \dot{y}_{1,p})$, where $\bar{\eta}_1: \hat{\bar{H_1}} \to \hat{\bar{G_1}}$ is as in Definition \ref{def: refined endoscopic datum}. It follows that we have $\eta_1(s_{1,p}) = s_{1,\der}y_{1,p}=\tilde{\kappa}_1g^{-1}y_{1,p}$. Moreover $g^{-1}y_{1,p} \in Z(\hat{G_1})$, so
\begin{equation}{\label{eq: kappa vs s}}
    \langle \tilde{\beta}_{1,p}, \tilde{\kappa}_1 \rangle = \langle \beta_{1,p}, \eta_1(s_{1,p}) \rangle \mu_{h_1}(y^{-1}_{1,p} g). 
\end{equation}

Now suppose that there exists $\fke_{\bb_1}=(H_{\bb_1}, \cH_{\bb_1}, H_{1,p}, \cH_{1,p}, \dot{s}_{1,p}, \eta_{\bb_1}) \in \sE^{\heartsuit,\Rig}_{\bb_1}(H_{1,p})$ as well as $\gamma_{H_{\bb_1}} \in \Sigma( H_{\bb_1}(\Q_p))$ that transfers to $\delta_1 \in \Sigma(J_{\bb_1}(\Q_p))$ and whose image in $\Sigma(H_{1,p}(\Q_p))$ is $\gamma_{H_{1,p}}$.
Write $\Delta_{\bb_1}$ for the transfer factor $\Delta[\fkw_{\bb_1}, z^{\Rig}_{\bb_1}]$. (This is different from $\Delta_p$.)
Then substituting the above equality and applying Corollary \ref{transfactcomparison}, the preceding formula for \eqref{atpequation} becomes
\begin{equation*}
         \sum_{[b_1]\in \fkD_p(I_{0,1},G_1,\bb_1)}  \mu_{h_1}(y_{1,p} g^{-1}) \Delta_{\bb_1}(\gamma_{H_{\bb_1}}, \delta_1)|D^{H_{1,p}}_{H_{\bb_1}}(\gamma_{H_{\bb_1}})|^{-\frac{1}{2}}_p|D^{G^*_1}_{M_{\bb_1}}(\gamma_1)|^{\frac{1}{2}}_p e_p(\gamma_{0,1},[b_1])O^{J_{\bb_1}}_{\delta_{[b_1]}}(\phi_{1,p}).
\end{equation*}
Following \cite[\S6.3]{Shin10}, we define a character $\bar{\delta}_{P(\bar{\nu}_{\bb_1})}: J_{\bb_1}(\Q_p) \to \C^{\times}$ such that $\bar{\delta}_{P(\bar{\nu}_{\bb_1})}(\delta'_1) = 
\delta_{P(\bar{\nu}_{\bb_1})}(\gamma'_1)$, where $\delta_{P(\bar{\nu}_{\bb_1})}: M_{\bb_1}(\Q_p) \to \C^{\times}$ is the modulus character for $P(\bar{\nu}_{\bb_1})$ and $\delta'_1 \in J_{\bb_1}(\Q_p)$ and $\gamma'_1 \in M_{\bb_1}(\Q_p)$ are semi-simple elements with matching stable conjugacy class. We then define 
$$\phi^0_{1,p} := \phi_{1,p} \cdot \bar{\delta}^{\frac{1}{2}}_{P(\bar{\nu}_{\bb_1})} \in \cH_{\tu{acc}}(J_{\bb_1}(\Q_p), \mathbf{1}_{\fkX_{1,p}}).$$
Let $\phi^{H_{\bb_1}}_p \in \cH(H_{\bb_1}(\Q_p), \chi^{-1}_{H_{\bb_1}})$ be a $\Delta_{\bb_1}$-transfer of $\phi^0_{1,p}$ to $H_{\bb_1}$ as in Proposition \ref{prop:LS-transfer}. In fact, since the set of $\nu$-acceptable elements is open, closed, and invariant under stable conjugacy, we can and do choose $\phi^{H_{\bb_1}}_p \in \cH_{\tu{acc}}(H_{\bb_1}(\Q_p), \chi^{-1}_{H_{\bb_1}})$ by multiplying with the characteristic function on the set of $\nu$-acceptable elements.
We have
\begin{align*}
& \sum_{[b_1]}  \mu_{h_1}(y_{1,p}g^{-1}) \Delta_{\bb_1}(\gamma_{H_{\bb_1}}, \delta_1)|D^{H_{1,p}}_{H_{\bb_1}}(\gamma_{H_{\bb_1}})|^{-\frac{1}{2}}_p|D^{G^*_1}_{M_{\bb_1}}(\gamma_1)|^{\frac{1}{2}}_p e_p(\gamma_{0,1},[b_1])O^{J_{\bb_1}}_{\delta_{[b_1]}}(\phi_{1,p})\\
=& \sum_{[b_1]}  \mu_{h_1}(y_{1,p}g^{-1}) \Delta_{\bb_1}(\gamma_{H_{\bb_1}}, \delta_1)|D^{H_{1,p}}_{H_{\bb_1}}(\gamma_{H_{\bb_1}})|^{-\frac{1}{2}}_p e_p(\gamma_{0,1},[b_1])O^{J_{\bb_1}}_{\delta_{[b_1]}}(\phi^0_{1,p})\\
=& ~ \mu_{h_1}(y_{1,p} g^{-1})|D^{H_{1,p}}_{H_{\bb_1}}(\gamma_{H_{\bb_1}})|^{-\frac{1}{2}}_p SO^{H_{\bb_1}}_{\gamma_{H_{\bb_1}}}(\phi^{H_{\bb_1}}_p),
\end{align*}
where the sums run over $[b_1]\in \fkD_p(I_{0,1},G_1,\bb_1)$.
The first equality above holds because either $\delta_{[b_1]}$ is not $\ov{\nu}_{\bb_1}$-acceptable (in which case $O^{J_{\bb_1}}_{\delta_{[b_1]}}(\phi_{1,p}) = 0 = O^{J_{\bb_1}}_{\delta_{[b_1]}}(\phi^0_{1,p})$ ~), or it is acceptable and then so is $\gamma_1$, which implies   $|D^{G^*_1}_{M_{\bb_1}}(\gamma_1)| = \delta_{P(\ov{\nu}_{\bb_1})}(\gamma_1)$ by \cite[Lem.~3.4]{Shin10}. The second equality comes from the transfer of orbital integrals from $J_{\bb_1}$ to $H_{\bb_1}$.

Finally, we need to rewrite this expression in terms of stable orbital integrals of $H_{1,p}$. By Lemma \ref{lem:ascentexistance} generalized to the central character setting as in \cite[\S3.5]{KretShin}, there exists a $\nu$-ascent $f^{\fke_{\bb_1}}_{\bb_1} \in \cH(H_{1,p}(\Q_p), \chi^{-1}_{H_{1,p}})$ of $\phi^{H_{\bb_1}}_p$. (We are inserting the superscript to emphasize the dependence on $\fke_{\bb_1}$.) The central character should indeed be $\chi^{-1}_{H_{1,p}}$ as one can see from comparing Lemmas \ref{transfactLevi} and \ref{lem:equivariance}.
 We then have
\begin{equation*}
    |D^{H_{1,p}}_{H_{\bb_1}}(\gamma_{H_{\bb_1}})|^{-\frac{1}{2}}_p SO^{H_{\bb_1}}_{\gamma_{H_{\bb_1}}}(\phi^{H_{\bb_1}}_p) =  SO^{H_{1,p}}_{\gamma_{H_{1,p}}}(f^{\fke_{\bb_1}}_{\bb_1}).
\end{equation*}
In conclusion, we have proven that
\begin{equation}{\label{eq: stab at p formula}}
     \mu_{h_1}(y_{1,p} g^{-1})SO^{H_{1,p}}_{\gamma_{H_{1,p}}}(f^{\fke_{\bb_1}}_{\bb_1})= \Delta_p(\gamma_{H_{1,p}}, \gamma_{0,1})\sum_{[b]\in \fkD_p(I_0,G,\bb)}  \langle \tilde\beta_p(\gamma_{0},[b]),\tilde{\kappa} \rangle^{-1} e_p(\gamma_0, [b])O^{J_{\bb}}_{\delta_{[b]}}(\phi_p).
\end{equation}

We repeat the construction of $f^{\fke_{\bb_1}}_{\bb_1}$  for each $\fke_{\bb_1}\in \sE^{\heartsuit,\Rig}_{\bb_1}(H_{1,p})$ and define
$$h_{1,p} :=  \sum\limits_{\fke_{\bb_1}\in \sE^{\heartsuit,\Rig}_{\bb_1}(H_{1,p})}  \mu_{h_1}(y_{1,p} g^{-1})f^{\fke_{\bb_1}}_{\bb_1}~ \in ~\cH(H_{1,p}(\Q_p), \chi^{-1}_{H_{1,p}}).$$
\begin{prop}\label{prop:SO-at-p}
For each $(G_1,H_1)$-regular $\gamma_{H_{1,p}} \in \Sigma(H_{1,p}(\Q_p))$, we have
\begin{equation*}
  SO_{\gamma_{H_{1,p}}}(h_{1,p}) = \Delta_p(\gamma_{H_{1,p}}, \gamma_{0,1})\sum_{[b]\in \fkD_p(I_0,G,\bb)}  \langle \tilde\beta_p(\gamma_{0},[b]),\tilde{\kappa} \rangle^{-1} e_p(\gamma_0, [b])O^{J_{\bb}}_{\delta_{[b]}}(\phi_p),
\end{equation*}
if $(H_1, \cH_1, s_1, \eta_1, \gamma_{H_{1,p}}) \in \ES^{\Iso}_{\eff}(G_1)$ and otherwise,
\begin{equation*}
  SO_{\gamma_{H_{1,p}}}(h_{1,p})=0.
\end{equation*}
\end{prop}
\begin{proof}
We suppose first that $(H_1, \cH_1, s_1, \eta_1, \gamma_{H_{1,p}}) \notin \ES^{\Iso}_{\eff}(G_1)$ and suppose for contradiction that $SO_{\gamma_{H_{1,p}}}(f^{\fke_{\bb_1}}_{\bb_1}) \neq 0$ for some $f^{\fke_{\bb_1}}_{\bb_1}$ corresponding to an element of $\sE^{\heartsuit, \Rig}_{\bb_1}(H_{1,p})$. Then by the definition of $f^{\fke_{\bb_1}}_{\bb_1}$ (Definition \ref{def:ascent}) we must have that $\gamma_{H_{1,p}}$ is stably conjugate in $H_{1,p}$ to some $\nu$-acceptable $\gamma_{H_{\bb_1}} \in H_{\bb_1}$.
The class of 
$$(H_{\bb_1}, \cH_{\bb_1}, H_{1,p},\cH_{1,p},s_{1,p}, \eta_{\bb_1}, \gamma_{H_{\bb_1}})\quad
\mbox{in}\quad\ES^{\emb}(M_{\bb_1}, G_1)$$
cannot lie in $\ES^{\emb}_{\eff}(J_{\bb_1}, G_1)$, since otherwise the image in $\ES^{\Iso}(G_1)$ would lie in $\ES^{\Iso}_{\eff}(G_1)$. This means (see the paragraph below Definition \ref{def:EKisoeff}) that the corresponding class of $\EK^{\Iso}(M_{\bb_1}, G_1)$ does not lie in $\EK^{\Iso}_{\eff}(J_{\bb_1}, G_1)$. Let $(\gamma, \lambda)$ be a representative of this class in $\EK^{\Iso}(M_{\bb_1}, G_1)$. Then either the stable conjugacy class of $\gamma$ does not transfer to $J_{\bb_1}$ or it transfers to an element that is not $\ov{\nu}_{\bb_1}$-acceptable. Either way, we have $SO_{\gamma_{H_{1,p}}}(f^{\fke_{\bb_1}}_{\bb_1}) = 0$ since $f^{\fke_{\bb_1}}_{\bb_1}$ is the transfer of a function on $J_{\bb_1}$ supported on acceptable elements.

We now consider the case that $(H_1, \cH_1, s_1, \eta_1, \gamma_{H_{1,p}}) \in \ES^{\Iso}_{\eff}(G_1)$. By Equation \eqref{eq: stab at p formula}, it suffices to show that
$$SO^{H_{1,p}}_{\gamma_{H_{1,p}}}(f^{\fke_{\bb_1}}_{\bb_1}) \neq 0\qquad \mbox{for a unique}\quad \fke_{\bb_1}\in\sE^{\heartsuit, \Rig}_{\bb_1}(H_{1,p}),$$
and that for each $\fke_{\bb_1}$, there is a unique $\nu_{\bb_1}$-acceptable $\gamma_{H_{\bb_1}} \in \Sigma(H_{\bb_1}(\Q_p))$ that is stably conjugate to $\gamma_{H_{1,p}}$ in $H_{1,p}(\Q_p)$. To this end, suppose
\begin{itemize}
    \item $\fke_{\bb_1}, \fke'_{\bb_1} \in \sE^{\heartsuit, \Rig}_{\bb_1}(H_{1,p})$,
    \item $\gamma_{H_{\bb_1}} \in \Sigma(H_{\bb_1}(\Q_p))$ (resp.~$\gamma_{H'_{\bb_1}} \in \Sigma(H'_{\bb_1}(\Q_p))$) is $\nu$-acceptable relative to $\fke_{\bb_1}$ (resp.~$\fke'_{\bb_1}$), 
    \item $\gamma_{H_{\bb_1}}$ and $\gamma_{H'_{\bb_1}}$ are both stably conjugate to $\gamma_{H_{1,p}}$ in $H_{1,p}$.
\end{itemize}
The tuple underlying $\fke'_{\bb_1}$ is labeled in the obvious way as for $\fke_{\bb_1}$.
Then the tuples 
\begin{equation}\label{eq:two-tuples}
    (H_{\bb_1}, \cH_{\bb_1}, H_{1,p}, \cH_{1,p}, s_{1,p}, \eta_{1,p})\quad\mbox{and}\quad(H'_{\bb_1}, \cH'_{\bb_1}, H_{1,p}, \cH_{1,p}, s_{1,p}, \eta'_{1,p})
\end{equation}
must project to elements of $\ES^{\emb}_{\eff}(J_{\bb_1}, G_1)$ or otherwise the contributions to the integral $SO^{H_{1,p}}_{\gamma_{H_{1,p}}}(f^{\fke_{\bb_1}}_{\bb_1})$ are zero by the previous part of the proof. Now consider the two classes in $\EK^{\Iso}(M_{\bb_1})$ that we get via the maps $$\ES^{\emb}_{\eff}(J_{\bb_1}, G_1) \to \ES^{\Iso}(M_{\bb_1}) \to \EK^{\Iso}(M_{\bb_1})$$
and let $(\gamma, \lambda)$ and $(\gamma', \lambda')$ be representatives of these classes. Since both $\gamma_{H_{\bb_1}}$ and $\gamma_{H'_{\bb_1}}$ are conjugate to $\gamma_{H_{1,p}}$ in $H_{1,p}(\Q_p)$, we have that  $\gamma$ and $\gamma'
$ are stably conjugate in $G_1(\Q_p)$ by \eqref{EQSSdiagram}. Then, by \cite[Lem.~3.5]{Shin10} (since $\gamma, \gamma'$ are $\ov{\nu}_{\bb_1}$-acceptable), we have that $\gamma$ and $\gamma'$ are stably conjugate in $M_{\bb_1}$. This implies that $(\gamma, \lambda)$ and $(\gamma', \lambda')$ are in the same class in $\EK^{\Iso}(M_{\bb_1}, G_1)$ and hence that the classes of the two tuples \eqref{eq:two-tuples} in $\ES^{\emb}_{\eff}(J_{\bb_1}, G_1)$ are equal. Then by Lemma \ref{lem: uniquenesslem}, we must have that $H'_{\bb_1} = H_{\bb_1}$ and that $\gamma_{H_{\bb_1}}$ and $\gamma_{H'_{\bb_1}}$ are stably conjugate in $H_{\bb_1}$. This completes the proof.
\end{proof} 

\subsection{Final steps}{\label{ss:final-steps}}
For each $\fke=(H, \cH, s, \eta) \in \sE^\heartsuit_{\el}(G)$, we have $\fke_1=(H_1,{}^L H_1,s_1,\eta_1)\in \sE^\heartsuit_{\el}(G_1)$. In the last subsection,we constructed the functions $h_1^{\infty,p}$, $h_{1,\infty}$, and $h_{1,p}$. Take
$$h_1:=h_1^{\infty,p} h_{1,\infty} h_{1,p} \in \cH(H_1(\A),\chi_{H_1}^{-1}).$$
For each $\gamma_{H_1}\in H_1(\Q)$, write
$\tu{Stab}_{\fkX_{H_1}}(\gamma_{H_1})$ for the group of $x\in \fkX_{H_1,\Q}$ such that $x\gamma_{H_1}$ is stably conjugate to $\gamma_{H_1}$. This has an obvious analogue with $H$ in place of $H_1$. 
For $h'_1\in \cH(H_1(\A),\chi_{H_1}^{-1})$, we define
\begin{equation}\label{eq:def-of-ST}
    \ST^{H_1}_{\el,\chi_{H_1}}(h) := \tau_{\fkX_{H_1}}(H_1)
\sum_{\gamma_{H_1}} |\tu{Stab}_{\fkX_{H_1}}(\gamma_{H_1})|^{-1} SO_{\gamma_{H_1}}(h'_1),
\end{equation}
where the sum runs over $\Sigma_{\el,\fkX_{H_1}}(H_1)$. We recall from \cite[Lem.~8.3.8]{KSZ} that for every $\gamma_{H_1}\in H_1(\Q)_{(G_1,H_1)\tu{-reg}}$, if we write $\gamma_H\in H(\Q)$ for the image of $\gamma_{H_1}$ then
\begin{equation}\label{eq:Stab(H1)=Stab(H)}
    |\tu{Stab}_{\fkX_{H_1}}(\gamma_{H_1})| = |\tu{Stab}_{\fkX_{H}}(\gamma_{H})|\, \ol{\iota}_H(\gamma_H),
\end{equation}
where $\ol{\iota}_H(\gamma_H)$ is as in \S\ref{sss:notation-for-TF}. 
Put $\lambda(\fke):=|\Out(\fke)|\in \Z_{>0}$ and define the coefficient
\begin{equation}\label{eq:def-of-iota(e)}
    \iota(\fke):=\tau(G)\tau(H)^{-1}\lambda(\fke)^{-1}=\tau_{\fkX}(G)\tau_{\fkX}(H)^{-1}\lambda(\fke)^{-1} \in \Q_{>0}.
\end{equation}

Recalling the notation from \S\ref{ss:endoscopic-data} (with $F=\Q$), we have a composite map
\begin{equation}\label{eq:ESG1-EKG1-EKG}
    \ES^{\Iso}(G_1)\stackrel{\eqref{EKESbij}}{\longrightarrow} \EK^{\Iso}(G_1) \ra \EK^{\Iso}(G),
\end{equation}
where the second map is induced by the projection $G_1\ra G$.
By $\EK^{\Iso}(G)_{\KP_{\bb}}$ we denote the subset of $(\gamma_0,\lambda)\in \EK^{\Iso}(G)$ such that there exists an \emph{acceptable} Kottwitz parameter $(\gamma_0, a, [b])\in \KP_{\bb}$ (see Definition \ref{def:Kottwitz-parameter}).

\begin{lem}\label{lem:SO-adelic}
Let $\gamma_{H_1}\in H_1(\Q)_{\semis}$ and assume that $\gamma_{H_1}$ is elliptic in $H_1(\Q)$. If either
\begin{enumerate}
    \item $(\fke_1,\gamma_{H_1})\notin \ES^{\Iso}(G_1)$ (namely $\gamma_{H_1}\notin H_1(\Q)_{(G_1,H_1)\tu{-reg}}$, or $\gamma_{H_1}\in H_1(\Q)_{(G_1,H_1)\tu{-reg}}$ but does not transfer to $G_1(\Q)$), or
    \item  the image of $(\fke_1,\gamma_{H_1})$ under \eqref{eq:ESG1-EKG1-EKG} lies outside $\EK^{\Iso}(G)_{\KP_{\bb}}$, then
\end{enumerate}
$$SO_{\gamma_{H_1}}(h_1)=0.$$
If $(\fke_1,\gamma_{H_1})$ belongs to $\ES^{\Iso}(G_1)$ and has image $(\gamma_0,\kappa)\in \EK^{\Iso}(G)_{\KP_{\bb}}$, then 
$$SO_{\gamma_{H_1}}(h_1)= \sum_{(a,[b])} N(\gamma_0,\kappa,a,[b],j),$$
where the sum runs over $\fkD(I_0,G;\A^{\infty,p})\times \fkD_p(I_0,G;\bb)$.
\end{lem}

\begin{proof}
We begin with proving the second assertion. In that case, 
$(\fke_1,\gamma_{H_1})$ transfers to $(\gamma_{0,1},\tilde \kappa_1)\in \EK^{\Iso}(G_1)$. Moreover
$(\fke_1, \gamma_{H_{1,p}}) \in \ES^{\Iso}_{\eff}(G_1)$ by definition, and $\gamma_{0,1},\gamma_0$ are elliptic in $G_1(\R),G(\R)$, respectively. The formula for $SO_{\gamma_{H_1}}(h_1)$ then follows from the defining formula \eqref{eq:def-of-N} and the equality 
$\langle \alpha(\gamma_0,a,[b]),\kappa\rangle=\langle \beta(\gamma_0,a,[b]),\tilde\kappa\rangle$
as well as the local formulas: Lemmas \ref{lem:SO-away-from-pinfty}, \ref{lem:SO-at-infty}, and Proposition \ref{prop:SO-at-p}. 

To check the first assertion, we may assume that $\gamma_{H_1}\in H_1(\Q)$ is $(G_1,H_1)$-regular and transfers over $\R$ to an element of $G_1(\R)_{\el}$ by Lemma \ref{lem:SO-at-infty}. In light of Lemma \ref{lem:SO-away-from-pinfty}, we may also assume that $\gamma_{H_1}$ transfers over $\A^{\infty,p}$ to an element of $G_1(\A^{\infty,p})$, as we already see that $SO_{\gamma_{H_1}}(h_1)=0$ otherwise. By Hypothesis \ref{hypo:quasi-split-at-p}, $\gamma_{H_1}$ transfers to an element of $G_1(\Q_p)$ as well. Now that $\gamma_{H_1}$ transfers locally everywhere to an element of $G_1(\A)$, the argument of \cite[p.188]{KottwitzAnnArbor} (using $\R$-ellipticity) shows that $\gamma_{H_1}$ transfers to an element $\gamma_{0,1}\in G_1(\Q)$. In particular, $(\fke_1,\gamma_{H_1})\notin \ES^{\Iso}(G_1)$, and its image in $\EK^{\Iso}(G_1)$ is $(\gamma_{0,1},\lambda_1)$ for some $\lambda_1$. Write $(\gamma_0,\lambda)\in \EK^{\Iso}(G)$ for the image of $(\gamma_{0,1},\lambda_1)$.

It remains to show that, assuming $SO_{\gamma_{H_1}}(h_1)\neq 0$, the element $\gamma_0$ can be extended to an acceptable parameter in $\KP_{\bb}$. It suffices to verify that $\gamma_{0,1}$ can be extended to an acceptable parameter $(\gamma_{0,1}, a_1, [b_1])\in\KP_{\bb_1}$. We have already seen that $\gamma_{0,1}$ is $\R$-elliptic. We can take $a_1$ to be trivial.
Finally, let us find $b_1$ as desired from Proposition \ref{prop:SO-at-p}. Since $SO_{\gamma_{H_1}}(h_{1,p})\neq 0$, the proposition tells us that there exists $(\delta_1,\lambda'_1)\in \EK^{\Iso}_{\eff}(J_{\bb_1},G_1)$ corresponding to $(\fke_1,\gamma_{H_1})$. Indeed, we can find $b_1\in \fkD_p(I_{0,1},G_1;\bb_1)$ from $\delta_1\in J_{\bb_1}(\Q_p)$ using the bijection from $\fkD_p(I_{0,1},G_1;\bb_1)$ to the set of conjugacy classes in $J_{\bb_1}(\Q_p)$ which become stably conjugate to $\gamma_{0,1}$ in $G(\ol\Q_p)$. (This bijection uses $G_{1,\tu{sc}}=G_{1,\der}$. See the proof of \cite[Prop.~7.4.16]{KSZ}.) The proof of the first assertion is complete.
\end{proof}

\begin{thm}\label{thm:stable-Igusa}
Assume that Conjecture \ref{conj:Igusa-general} is true. Then \eqref{eq:MackCrane} may be rewritten as
  $$\tr\big(\phi^{\infty,p}\phi^{(j)}_p \big|\iota H_c(\Ig_{\bb},\cL_\xi) \big) = 
  \sum_{\fke\in \sE^\heartsuit_{\el}(G)}\iota(\fke) \ST^{H_1}_{\el,\chi_1^{\fke}} (h_1).$$
  \end{thm}

\begin{proof}
This last step in the stabilization is nearly identical to the proof of \cite[Thm.~8.3.10]{KSZ} but we provide some details for completeness. 
Write $\Sigma\fkK_{\el}(G)$ for the set of equivalence classes of pairs $(\gamma_0,\kappa)$, where $\gamma_0\in G(\Q)$ is elliptic, $\kappa\in \fkK(I_0/\Q)$; two pairs $(\gamma_0,\kappa)$ and $(\gamma'_0,\kappa')$ are equivalent if there exists $g\in G(\Qbar)$ such that $g\gamma_0 g^{-1}=\gamma'_0$, $g^{-1}\tau(g)\in I_0(\Qbar)$ for $\tau\in \Gamma_{\Q}$, and the inner twisting $I_{0,\Qbar}\cong I_{\gamma'_0,\Qbar}$ induced by $g$ carries $\kappa$ to $\kappa'$.
We define the set 
$${\ES}^{\sim}_{\el}(G) := \{ (\fke, \gamma_H) : \fke \in E^{\heartsuit}_{\el}(G), \gamma_H \in \Sigma(H(\Q))_{\el,(G,H)\tu{-reg}}\},$$
on which $x\in \fkX_\Q=\fkX\cap Z(\Q)$ acts by $(\fke,\gamma_H)\mapsto (\fke,x\gamma_H)$. 
Write ${\ES}^{\sim}_{\el,\fkX}(G)$ for the quotient set by this action. Similarly, $\Sigma\fkK_{\el,\fkX}(G)$ is the quotient of $\Sigma\fkK_{\el}(G)$ by the multiplication action of $\fkX_{\Q}$. The natural map $\tilde\fkE:{\ES}^{\sim}_{\el}(G)\ra \Sigma\fkK_{\el}(G)$, which is constructed in the same way as \eqref{EKESbij}, is $\fkX_{\Q}$-equivariant, thus induces a map $\tilde\fkE_{\fkX}:{\ES}^{\sim}_{\el,\fkX}(G)\ra \Sigma\fkK_{\el,\fkX}(G)$.
(See also \cite[\S8.3.1]{KSZ}, where the same maps $\tilde\fkE$ and $\tilde\fkE_{\fkX}$ are defined.) The right hand side of the theorem is computed as follows. We start by plugging in \eqref{eq:def-of-ST}, \eqref{eq:def-of-iota(e)} and \eqref{eq:tau(H1)=tau(H)}, where $\gamma_{H_1}\in H_1(\Q)$ denotes an arbitrary lift of $\gamma_H\in H(\Q)$ in the second, third, and fourth lines (each summand is independent of the choice of $\gamma_{H_1}$):

\begin{eqnarray}
&  &\tau_{\fkX}(G)  \sum_{\fke\in \sE^\heartsuit_{\el}(G)}  \lambda(\fke)^{-1} \sum_{\gamma_{H_1}\in \Sigma_{\el,\fkX_{H_1}}(H_1)} |\mathrm{Stab}_{\fkX_{H_1}}(\gamma_{H_1})|^{-1} SO_{\gamma_{H_1}}(h_1)
\nonumber\\
& \stackrel{\eqref{eq:Stab(H1)=Stab(H)}}{=\joinrel=} & \tau_{\fkX}(G)  \sum_{\fke\in \sE^\heartsuit_{\el}(G)}  \lambda(\fke)^{-1} \sum_{\gamma_{H}\in \Sigma_{\el,\fkX}(H)} \ol\iota_H(\gamma_H)^{-1} SO_{\gamma_{H_1}}(h_1)
\nonumber\\
& \stackrel{\tu{Lem.}~\ref{lem:SO-adelic}}{=\joinrel=}  &
\tau_\fkX(G)\sum_{(\fke,\gamma_H)\in \cE\Sigma^\sim_{\el,\fkX}(G)} \lambda(\fke)^{-1}\ol\iota_H(\gamma_H)^{-1}SO_{\gamma_{H_1}}(h_1)
\nonumber\\
& =\joinrel= &
\tau_\fkX(G) \sum_{(\gamma_0,\kappa)\in \Sigma\fkK_{\el,\fkX}(G)} \sum_{(\fke,\gamma_H)\in \cE\Sigma^\sim_{\el,\fkX}(G)
\atop \tilde\fkE_\fkX:(\fke,\gamma_H)\mapsto (\gamma_0,\kappa) } \lambda(\fke)^{-1}\ol\iota_H(\gamma_H)^{-1}SO_{\gamma_{H_1}}(h_1) \nonumber\\
&  \stackrel{\tu{Lem.}~\ref{lem:SO-adelic}}{=\joinrel=} &
\tau_\fkX(G) \sum_{(\gamma_0,\kappa)\in \Sigma\fkK_{\el,\fkX}(G)\atop \gamma_0:\,\R\textrm{-elliptic}} \sum_{a \in \fkD(I_0,G;\A_f^p) \atop [b]\in \fkD_p(I_0,G;\bb)}  \ol\iota_G(\gamma_0)^{-1} N(\gamma_0,\kappa,a,[b],j).
\nonumber
	\end{eqnarray}
More precisely, the second and third equalities follow from re-parameterizing the sums and imposing additional conditions on $\gamma_H$ based on the vanishing condition (i) in Lemma \ref{lem:SO-adelic}. 
The last equality is obtained from Lemma \ref{lem:SO-adelic} (the last assertion and vanishing condition (ii)) as well as \cite[Cor.~8.3.5]{KSZ} (which relates $\ol\iota_G(\gamma_0)^{-1}$ to a sum of $\lambda(\fke)^{-1}\ol\iota_H(\gamma_H)^{-1}$).
\end{proof}

\clearpage 

\setlength{\glsdescwidth}{\textwidth}

\printglossaries

\bibliographystyle{amsalpha}
\bibliography{biblio}

\end{document}

%% file: entries.tex
\newglossaryentry{alphac}
{
    name=\ensuremath{\alpha(\fkc)},
    description={Kottwitz invariant associated to $\bb$-admissible Kottwitz parameter $\fkc$},
    sort=alphac,
}
\newglossaryentry{BG}
{
    name=\ensuremath{B(G)},
    description={Kottwitz set of Frobenius conjugacy classes of $G(\breve{\Q_p})$},
    sort=BG,
}
\newglossaryentry{betav}
{
    name=  \ensuremath{\beta_v(\fkc)}   ,
    description={term in construction of Kottwitz invariant $\alpha(\fkc)$},
    sort=betav,
}
\newglossaryentry{tildebetav}
{
    name=  \ensuremath{\tilde{\beta}_v(\fkc)}   ,
    description={term in construction of Kottwitz invariant $\alpha(\fkc)$},
    sort=betav,
}
\newglossaryentry{c}
{
    name= \ensuremath{\fkc}    ,
    description={$\bb$-admissible Kottwitz parameter},
    sort=c,
}
\newglossaryentry{c1c}
{
    name=  \ensuremath{c_1(\fkc)}   ,
    description={$\vol(I_\fkc(\Q)\backslash I_{\fkc}(\A^\infty)/\fkX^\infty)$},
    sort=cc,
}
\newglossaryentry{c2gamma}
{
    name=  \ensuremath{c_2(\gamma_0)}   ,
    description={$|\ker(\ker^1(\Q,I_0)\ra \ker^1(\Q,G))|\in \Z_{>0}$},
    sort=cgamma,
}

\newglossaryentry{CKP}
{
    name=   \ensuremath{\mathcal{CKP}_{\bb}}  ,
    description={set of $\bb$-classical Kottwitz parameters},
    sort=CKP,
}
\newglossaryentry{CQ}
{
    name=  \ensuremath{\ov{C}_{\Q}}   ,
    description={closed dominant Weyl chamber},
    sort=CQ,
}

\newglossaryentry{D}
{
    name=  \ensuremath{\D}   ,
    description={pro-torus with character group $\Q$},
    sort=D,
}

\newglossaryentry{DIGA}
{
    name=   \ensuremath{\fkD(I,G;\A^S)}  ,
    description={ {$\ker(H^1(\A^S,I)\ra H^1(\A^S,G))$} },
    sort=DIGA,
}
\newglossaryentry{DIGb}
{
    name=  \ensuremath{\fkD_p(I,G;\bb)}    ,
    description={pre-image of $[\bb]$ under $B(I) \to B(G)$},
    sort=Digb,
}

\newglossaryentry{Delta0}
{
    name=   \ensuremath{\Delta_0}  ,
    description={canonical local transfer factor for $G^*$},
    sort=Delta,
}

\newglossaryentry{DeltaA}
{
    name=   \ensuremath{\Delta_{\A}}  ,
    description={canonical adelic transfer factor},
    sort=deltaa,
}
\newglossaryentry{deltaP}
{
    name=  \ensuremath{\delta_{P_{\nu}}(m)}   ,
    description={modulus character},
    sort=deltap,
}

\newglossaryentry{Deltaw}
{
    name=   \ensuremath{\Delta[\fkw]}  ,
    description={Whittaker normalized transfer factor for $G^*$},
    sort=deltaw,
}

\newglossaryentry{deltawrig}
{
    name=   \ensuremath{\Delta[\fkw, z^{\Rig}]}  ,
    description={local transfer factor for $G$ with rigid normalization },
    sort=deltawzrig,
}
\newglossaryentry{sEemb(M,G)}
{
    name=  \ensuremath{\sE^{\emb}(M,G)}   ,
    description={set of embedded endoscopic data},
    sort=Eembmg,
}

\newglossaryentry{mEemb(M,G)}
{
    name=   \ensuremath{\mE^{\emb}(M,G)}  ,
    description={set of isomorphism classes of embedded endoscopic data},
    sort=Eembmg,
}

\newglossaryentry{mEemb(M,G;H)}
{
    name=  \ensuremath{\mE^{\emb}(M,G;H)}   ,
    description={isomorphism classes of embedded endoscopic data relative to $(H, \cH, s, \eta)$},
    sort=Eembmgh,
}
\newglossaryentry{mEi(M,G;H)}
{
    name=  \ensuremath{\mE^i(M,G;H)}    ,
    description={inner isomorphism classes of embedded endoscopic data},
    sort=eimgh,
}

\newglossaryentry{sE(G)}
{
    name=   \ensuremath{\sE(G)}  ,
    description={set of standard endoscopic data},
    sort=eg,
}

\newglossaryentry{sEheartel(G)}
{
    name=  \ensuremath{\sE^\heartsuit_{\el}(G)}   ,
    description={set of representatives of $\mE_{\el}(G)$},
    sort=eelg,
}
\newglossaryentry{sEheartrig(G1)}
{
    name=  \ensuremath{\sE^{\heartsuit,\Rig}_{v}(G_1)}    ,
    description={localizations of $\sE^\heartsuit_{\el}(G_1)$ using rigid inner twists},
    sort=erigg,
}

\newglossaryentry{sEheartb(H)}
{
    name=  \ensuremath{\sE^\heartsuit_{\bb_1}(H_{1,p})}   ,
    description={representatives of $\mE^{\emb}_{\eff}(J_{\bb_1}, G^*_{1,p} ; H_{1,p})$},
    sort=ebh,
}

\newglossaryentry{mE(G)}
{
    name=  \ensuremath{\mE(G)}   ,
    description={set of isomorphism classes of standard endoscopic data},
    sort=eg,
}

\newglossaryentry{cEiso}
{
    name=  \ensuremath{ \cE^{\Iso}_F }   ,
    description={local Kottwitz gerbe},
    sort=eiso,
}
\newglossaryentry{kerHab}
{
    name=  \ensuremath{ \fkE(I,G;\A^S) }   ,
    description={$\ker(H^1_{\ab}(\A^S,I)  \ra H^1_{\ab}(\A^S,G))$ },
    sort=eiga,
}
\newglossaryentry{mEembeff(J,G)}
{
    name=  \ensuremath{ \mE^{\emb}_{\eff}(J_{\bb}, G) }   ,
    description={},
    sort=eembeffjg,
}
\newglossaryentry{mEisoef(M,G)}
{
    name=  \ensuremath{ \mE^{\Iso}_{\ef}(M_{\bb}, G) }   ,
    description={},
    sort=eisoefmg,
}
\newglossaryentry{mEisoeff(J,G)}
{
    name=  \ensuremath{ \mE^{\Iso}_{\eff}(J_{\bb},G) }   ,
    description={},
    sort=eisoeffjbg,
}
\newglossaryentry{EKiso(G)}
{
    name=  \ensuremath{ \EK^{\Iso}(G) }   ,
    description={},
    sort=ekisog,
}
\newglossaryentry{EKiso(M,G)}
{
    name=  \ensuremath{ \EK^{\Iso}(M,G) }   ,
    description={},
    sort=ekisomg,
}
\newglossaryentry{EKisoef(M,G)}
{
    name=  \ensuremath{ \EK^{\Iso}_{\ef}(M_{\bb},G) }   ,
    description={},
    sort=ekisoefmbg,
}
\newglossaryentry{EKisoeff(J,G)}
{
    name=  \ensuremath{ \EK^{\Iso}_{\eff}(J_{\bb}, G) }   ,
    description={},
    sort=ekisoeffjbg,
}
\newglossaryentry{mEisoeff(J,G;H)}
{
    name=  \ensuremath{ \mE^{\Iso}_{\eff}(J_{\bb}, G;H) }   ,
    description={},
    sort=eisoeffjbgh,
}
\newglossaryentry{ESemb(M,G)}
{
    name=  \ensuremath{ \ES^{\emb}(M,G) }   ,
    description={},
    sort=esembmg,
}
\newglossaryentry{mEembeff(J,G;H)}
{
    name=  \ensuremath{ \mE^{\emb}_{\eff}(J_{\bb}, G; H) }   ,
    description={},
    sort=eembeffjbgh,
}
\newglossaryentry{ESembeff(J,G)}
{
    name=  \ensuremath{ \ES^{\emb}_{\eff}(J_{\bb},G) }   ,
    description={},
    sort=esembeffjbg,
}
\newglossaryentry{ESisoef(G)}
{
    name=  \ensuremath{ \ES^{\Iso}_{\ef}(G) }   ,
    description={},
    sort=esisoefg,
}
\newglossaryentry{ESisoeff(G)}
{
    name=  \ensuremath{ \ES^{\Iso}_{\eff}(G) }   ,
    description={},
    sort=esisoeffg,
}
\newglossaryentry{ESisoef(M,G)}
{
    name=  \ensuremath{ \ES^{\Iso}_{\ef}(M_{\bb}, G) }   ,
    description={},
    sort=esisoefmbg,
}
\newglossaryentry{ESisoeff(J,G)}
{
    name=  \ensuremath{ \ES^{\Iso}_{\eff}(J_{\bb}, G) }   ,
    description={},
    sort=esisoeffjbg,
}
\newglossaryentry{cErig}
{
    name=  \ensuremath{ \cE^{\Rig}_F }   ,
    description={local rigid gerbe},
    sort=erigf,
}
\newglossaryentry{cErigdotV}
{
    name=  \ensuremath{ \cE^{\Rig}_{\dot{V}} }   ,
    description={global rigid gerbe},
    sort=erigv,
}
\newglossaryentry{ESiso(M,G)}
{
    name=  \ensuremath{ \ES^{\Iso}(M,G) }   ,
    description={},
    sort=esisomg,
}
\newglossaryentry{favgphi}
{
    name=  \ensuremath{ f^{\tu{avg}}_\varphi }   ,
    description={averaged Lefschetz function},
    sort=favgphi,
}
\newglossaryentry{sigmab}
{
    name=  \ensuremath{  \varsigma_{\bb} }   ,
    description={map from $\IG_{\bb}$ to $ \mS_{K_p,\ol{k}}$},
    sort=sigmab,
}
\newglossaryentry{Gn}
{
    name=  \ensuremath{ G_n }   ,
    description={$G/Z_n$},
    sort=gn,
}
\newglossaryentry{hatbarG}
{
    name=  \ensuremath{ \hat{\bar{G}} }   ,
    description={inverse limit of $\hat{G_n}$},
    sort=g,
}
\newglossaryentry{GammaEF}
{
    name=  \ensuremath{ \Gamma_{E/F} }   ,
    description={ Galois group of the extension $E/F$},
    sort=gammaef,
}
\newglossaryentry{GammaF}
{
    name=  \ensuremath{ \Gamma_F }   ,
    description={absolute Galois group of $F$},
    sort=gammaf,
}
\newglossaryentry{Gamma(G(F))}
{
    name=  \ensuremath{ \Gamma(G(F)) }   ,
    description={ set of conjugacy classes in $G(F)$},
    sort=gammagf,
}
\newglossaryentry{hatbarH}
{
    name=  \ensuremath{ \hat{\bar{H}}  }   ,
    description={inverse limit of $\hat{H_n}$, where $H_n = H/Z_n$ and $Z_n \subset Z(G)$},
    sort=h,
}
\newglossaryentry{H(F)(G,H)-reg}
{
    name=  \ensuremath{ H(F)_{(G,H) \tu{-} \reg} }   ,
    description={$(G,H)$-regular elements of $H(F)$},
    sort=hfghreg,
}
\newglossaryentry{H(F)ss}
{
    name=  \ensuremath{ H(F)_{\tu{ss}} }   ,
    description={semi-simple elements of $H(F)$},
    sort=hfss,
}
\newglossaryentry{H(F)G-sr}
{
    name=  \ensuremath{ H(F)_{G\tu{-}\sr} }   ,
    description={$G$-strongly regular elements of $H(F)$},
    sort=hfgsr,
}
\newglossaryentry{Hacc(M)}
{
    name=  \ensuremath{ \cH_{\tu{acc}}(M_{\nu}, \chi^{-1}) }   ,
    description={subspace of Hecke algebra supported on $\nu$-acceptable elements},
    sort=haccmnuchi,
}
\newglossaryentry{cH(G,chi)}
{
    name=  \ensuremath{ \cH(G,\chi^{-1}) }   ,
    description={Hecke algebra relative to $(\fkX, \chi)$},
    sort=hgchi,
}
\newglossaryentry{cHur(G,chi)}
{
    name=  \ensuremath{  \cH^{\ur}(G,\chi^{-1}) }   ,
    description={unramified Hecke algebra relative to $(\fkX, \chi)$},
    sort=hurgchi,
}
\newglossaryentry{Ic}
{
    name=  \ensuremath{ I_{\fkc} }   ,
    description={inner twist associated to $\fkc$},
    sort=ic,
}
\newglossaryentry{IGb}
{
    name=  \ensuremath{ \IG_{\bb} }   ,
    description={ Igusa variety},
    sort=igb,
}
\newglossaryentry{inv[ziso]}
{
    name=  \ensuremath{ \inv[z^{\Iso}] }   ,
    description={invariant used in isocrystal normalization of local transfer factors},
    sort=invziso,
}
\newglossaryentry{inv[zrig]}
{
    name=  \ensuremath{ \inv[z^{\Rig}] }   ,
    description={invariant used in rigid normalization of local transfer factors},
    sort=invzrig,
}
\newglossaryentry{J}
{
    name=  \ensuremath{ J_b }   ,
    description={connected reductive group associated to $b \in G(\breve{F})$},
    sort=jb,
}
\newglossaryentry{JP(pi)}
{
    name=  \ensuremath{ J_{P^{\op}_{\nu}}(\pi) }   ,
    description={Jacquet module},
    sort=jpopnupi,
}
\newglossaryentry{kappa}
{
    name=  \ensuremath{ \kappa }   ,
    description={Kottwitz map from $B(G)$ to $\pi_1(G)_{\Gamma_F}$},
    sort=kappa,
}
\newglossaryentry{K(I/Q)}
{
    name=  \ensuremath{ \fkK(I/\Q) }   ,
    description={Kottwitz group},
    sort=kiq,
}
\newglossaryentry{KPbb}
{
    name=  \ensuremath{ \KP_{\bb} }   ,
    description={set of $\bb$-admissible Kottwitz parameters},
    sort=kpb,
}
\newglossaryentry{cLxi}
{
    name=  \ensuremath{ \cL_\xi }   ,
    description={lisse $\ell$-adic sheaf on $\mS_{K_p}$},
    sort=skp,
}
\newglossaryentry{lambdaH}
{
    name=  \ensuremath{ \lambda_H }   ,
    description={character arising interaction of transfer factors and $Z(G)$},
    sort=lambdah,
}
\newglossaryentry{Mb}
{
    name=  \ensuremath{ M_b  }   ,
    description={Levi subgroup of $G^*$ associated to $b \in G(\breve{F})$},
    sort=mb,
}
\newglossaryentry{Mnu}
{
    name=  \ensuremath{ M_{\nu} }   ,
    description={centralizer of cocharacter $\nu$ in $G$},
    sort=mnu,
}
\newglossaryentry{MrigEN}
{
    name=  \ensuremath{ M^{\Rig}_{E,N} }   ,
    description={finite Galois module associated to local rigid gerbe},
    sort=mrigen,
}
\newglossaryentry{MrigES}
{
    name=  \ensuremath{ M^{\Rig}_{E_i, \dot{S_i}} }   ,
    description={finite Galois module associated to global rigid gerbe},
    sort=mrigeisi,
}
\newglossaryentry{nub}
{
    name=  \ensuremath{ \nu_b }   ,
    description={slope cocharacter attached to $b \in G(\breve{F})$},
    sort=nub,
}
\newglossaryentry{[nu]}
{
    name=  \ensuremath{ [\nu] }   ,
    description={slope morphism from $B(G)$ to $(\Hom_{\breve{F}}( \D, G)/ G(\breve{F}) )^{\langle \sigma \rangle}$},
    sort=nu,
}
\newglossaryentry{ovnu}
{
    name=  \ensuremath{ \ov{\nu} }   ,
    description={ Newton map from $B(G)$ to $\ov{C}_{\Q}$},
    sort=nu,
}
\newglossaryentry{OGgamma(f)}
{
    name=  \ensuremath{ O^G_{\gamma}(f) }   ,
    description={ orbital integral},
    sort=oggammaf,
}
\newglossaryentry{PrigES}
{
    name=  \ensuremath{ P^{\Rig}_{E, \dot{S}_i} }   ,
    description={ multiplicative group with character group $M^{\Rig}_{E_i, \dot{S_i}}$},
    sort=prigesi,
}
\newglossaryentry{PrigV}
{
    name=  \ensuremath{ P^{\Rig}_{\dot{V}} }   ,
    description={inverse limit of $P^{\Rig}_{E, \dot{S}_i}$},
    sort=prigv,
}
\newglossaryentry{phi(j)}
{
    name=  \ensuremath{ \phi^{(j)} }   ,
    description={translate of $\phi$},
    sort=phij,
}
\newglossaryentry{s}
{
    name=  \ensuremath{\dot{s}_v}   ,
    description={element arising in compatible normalization $\Delta[\fkw_v, z^{\Rig}_v]$ at each place $v$},
    sort=sv,
}
\newglossaryentry{ShKp}
{
    name=  \ensuremath{ \Sh_{K_p} }   ,
    description={Shimura variety with parahoric level structure at $p$},
    sort=shkp,
}
\newglossaryentry{mSKp}
{
    name=  \ensuremath{ \mS_{K_p} }   ,
    description={ integral model of $\Sh_{K_p}$},
    sort=skp,
}
\newglossaryentry{SOGgamma(f)}
{
    name=  \ensuremath{ SO^G_{\gamma}(f) }   ,
    description={stable orbital integral},
    sort=soggammaf,
}
\newglossaryentry{SigmaG(F)}
{
    name=  \ensuremath{ \Sigma(G(F)) }   ,
    description={set of stable conjugacy classes in $G(F)$},
    sort=sigmagf,
}
\newglossaryentry{tauX'(G')}
{
    name=  \ensuremath{ \tau_{\fkX'}(G') }   ,
    description={$\vol(G'(\Q)\backslash G'(\A)/\fkX')$},
    sort=tauxg,
}
\newglossaryentry{TN}
{
    name=  \ensuremath{ \TN}   ,
    description={Tate--Nakayama morphism for $H^1$ of local or global rigid gerbes},
    sort=tn,
}
\newglossaryentry{uEFN}
{
    name=  \ensuremath{ u_{E/F, N} }   ,
    description={multiplicative group with character group $M^{\Rig}_{E,N}$},
    sort=uefn,
}
\newglossaryentry{u}
{
    name=  \ensuremath{ u }   ,
    description={pro-finite group associated with local rigid gerbe},
    sort=u,
}
\newglossaryentry{(X,chi)}
{
    name=  \ensuremath{ (\fkX, \chi) }   ,
    description={central character datum},
    sort=xchi,
}
\newglossaryentry{yv}
{
    name=  \ensuremath{\dot{y}'_v}   ,
    description={element arising in compatible normalization of $\Delta[\fkw_v, z^{\Rig}_v]$ at each place $v$},
    sort=yv,
}
\newglossaryentry{Zac}
{
    name=  \ensuremath{  Z_{\tu{ac}} }   ,
    description={anti-cuspidal part of $Z_G^{\circ}$},
    sort=zac,
}
\newglossaryentry{ziso}
{
    name=  \ensuremath{  z^{\Iso} }   ,
    description={algebraic $1$-cocycle of local Kottwitz gerbe},
    sort=ziso,
}
\newglossaryentry{zrig}
{
    name=  \ensuremath{ z^{\Rig} }   ,
    description={algebraic $1$-cocycle of rigid gerbe},
    sort=zrig,
}
\newglossaryentry{Zn}
{
    name=  \ensuremath{ Z_n }   ,
    description={pre-image in $Z(G)$ of $(Z(G)/Z(G_{\der}))[n]$},
    sort=zn,
}
\newglossaryentry{ZhatG+}
{
    name=  \ensuremath{ Z(\hat{\bar{G}})^+ }   ,
    description={pre-image in $Z(\hat{\bar{G}})$ of $Z(\hat{G})^{\Gamma_F}$},
    sort=zg,
}
\newglossaryentry{ZhatbarH+}
{
    name=  \ensuremath{ Z(\hat{\bar{H}})^+ }   ,
    description={ pre-image of $Z(\hat{H})^{\Gamma_F}$ in $Z(\hat{\bar{H}})$},
    sort=zh,
}
\newglossaryentry{rhoG}
{
    name=  \ensuremath{\rho_G}   ,
    description={Map $\rho_G: W_F \to \Out(\widehat{G})$ induced by action of $W_F$ on $\widehat{G}$},
    sort=rhog,
}